\documentclass[11pt,a4paper,reqno]{amsart}
\usepackage{amsmath,amssymb,graphics,epsfig,color,enumerate,psfrag}
\usepackage{dsfont}  % ale per ale: da in{\thserire e latexare su plutone for final version
\usepackage{verbatim}
\newtheorem{Theorem}{Theorem}[section]
\newtheorem{TheoremA}{Theorem}
\newtheorem{Lemma}[Theorem]{Lemma}
\newtheorem{Proposition}[Theorem]{Proposition}

\newtheorem{Remark}[Theorem]{Remark}

\newtheorem{Claim}[Theorem]{Claim}
\newtheorem{Definition}[Theorem]{Definition}
\newtheorem{Fact}[Theorem]{Fact}

 \definecolor{darkgreen}{rgb}{0,0.4,0}

\definecolor{light}{gray}{.9}

\newcommand{\cA}{\ensuremath{\mathcal A}}

\newcommand{\cC}{\ensuremath{\mathcal C}}

\newcommand{\cE}{\ensuremath{\mathcal E}}

\newcommand{\cG}{\ensuremath{\mathcal G}}

\newcommand{\cI}{\ensuremath{\mathcal I}}

\newcommand{\cO}{\ensuremath{\mathcal O}}
\newcommand{\cP}{\ensuremath{\mathcal P}}

\newcommand{\cR}{\ensuremath{\mathcal R}}

\newcommand{\cT}{\ensuremath{\mathcal T}}

\newcommand{\cV}{\ensuremath{\mathcal V}}

\newcommand{\cZ}{\ensuremath{\mathcal Z}}

\newcommand{\bbC}{{\ensuremath{\mathbb C}} }

\newcommand{\bbE}{{\ensuremath{\mathbb E}} }

\newcommand{\bbN}{{\ensuremath{\mathbb N}} }

\newcommand{\bbP}{{\ensuremath{\mathbb P}} }
\newcommand{\bbQ}{{\ensuremath{\mathbb Q}} }
\newcommand{\bbR}{{\ensuremath{\mathbb R}} }

\newcommand{\bbZ}{{\ensuremath{\mathbb Z}} }
\let\a=\alpha    \let\d=\delta  \let\e=\varepsilon
\let\f=\varphi \let\g=\gamma     \let\k=\kappa  \let\l=\lambda
            
   \let\t=\tau   \let\th=\vartheta
  
\let\D=\Delta   \let\G=\Gamma  \let\L=\Lambda 
\let\O=\Omega      

\newcommand{\ds}{\displaystyle}

\author[A.\ Faggionato]{Alessandra Faggionato}
\address{Alessandra Faggionato.
  Dipartimento di Matematica, Universit\`a di Roma `La Sapienza'
  P.le Aldo Moro 2, 00185 Roma, Italy}
\email{faggiona@mat.uniroma1.it}

\author[V. Silvestri]{Vittoria Silvestri}
\address{Vittoria Silvestri. 
DAMTP, Centre for Mathematical Sciences,
Wilberforce Road,
Cambridge,
CB3 0WA,
United Kingdom.}
\email{V.Silvestri@maths.cam.ac.uk}

\title[Random walks on quasi one dimensional lattices]{Random walks on quasi one dimensional lattices: large deviations and fluctuation theorems}

\begin{document}

\begin{abstract}  Several stochastic processes 
modeling molecular motors   on a 
linear track are given by random walks (not necessarily Markovian)  on 
 quasi 1d lattices and  share a common regenerative structure.
 % and their mathematical investigation can be reduced to the study of a time changed sum of i.i.d. random vectors.
Analyzing  this abstract common structure, we derive information on the  large fluctuations  of the  stochastic process by proving  large deviation principles for the first--passage times and for the position. We focus our attention on the 
Gallavotti--Cohen--type symmetry of the position rate function  (fluctuation theorem), showing
its equivalence with the independence of suitable random variables. In the special case of Markov random walks, we show that this symmetry is universal only inside a suitable class of quasi 1d lattices. 
\medskip


\medskip

%Secondary
82C05.  %Classical dynamic and nonequilibrium statistical mechanics (general)
\end{abstract}

\maketitle
%\setcounter{tocdepth}{1}
%\tableofcontents

%\club  Attenzione terminology: random walks a volte intendiamo quelli con tempi esponenziali, a volte processi stocastici sul grafo $\cG$.

\section{Introduction}
Molecular motors are proteins working as nanomachines \cite{H}:  they usually convert chemical energy coming from ATP hydrolysis   to produce mechanical work fundamental e.g. for cargo transport inside the cell, cell division, genetic transcription, muscle contraction.
Molecular motors   are object of intensive study in biology and biophysics. They are crucial in several fundamental biological processes and   are also relevant  from a theoretic viewpoint in  statistical physics, since  they are  small systems operating inside  an environment with large  thermal fluctuations (differently from  macroscopic motors) and in a out--of--equilibrium regime.  
 We concentrate here on the large family of  molecular motors working in a non--cooperative way and  moving along the cytoskeletal filaments, which are given by polarized homogeneous polymers.

The theoretical study of molecular motors has been developed  from two main 
  modelizations. In the so called \emph{Brownian ratchet} model \cite{JAP,PJA,Re} the dynamics of the molecular motor is given by a one--dimensional diffusion in a switching force field  (i.e.   the force field changes at random times). 
 The other paradigm, on which we concentrate here,   is given by continuous time random walks\footnote{By \emph{random walks} we mean  stochastic jump processes on a  given graph. When we restrict to  random walks given by  Markov chains \cite{N} (hence with exponential waiting times),   we call them \emph{Markov random walks}}  on quasi linear graphs having a periodic structure \cite{FK1,FK2,K2,KF1,KF2,KF3,TF}. We call these graphs \emph{quasi 1d lattices}, since they are obtained by gluing together several copies of a fundamental cell
in a linear fashion. The geometric  complexity of the fundamental cell (given by a finite, oriented and connected graph) corresponds to  the   conformational transformations of the molecular motor in its mechanochemical cycle. The simplest example is given by a random walk on $\bbZ$ with periodic jump rates \cite{D} (in this case the fundamental cell is given by an interval with $N$ sites, $N$ being the periodicity), while random walks on  other classes of   quasi 1d lattices (parallel--chain models and divided--chain models) have been studied motivated by experimental evidence of a richer structure \cite{DK1,DK2,K1}.

In a companion paper \cite{FS} we have studied in full generality  both the asymptotic velocity (law of large numbers) and the gaussian fluctuations (invariance principle) for random walks on quasi 1d lattices. We focus here on their large deviations. 
Large deviations and Gallavotti--Cohen--type symmetries (also called  \emph{fluctuation theorems}), which are given by  special identities satisfied by the rate function,  have received in the last decade much attention inside  non--equilibrium statistical physics of small systems and in particular for molecular motors
(cf. \cite{A,AG,FD,LM,S,SPWS} and references therein). 
 
 We treat random walks on quasi 1d lattices in full generality.  All relevant information concerning the position of the random walk is  encoded in an associated random walk on $\bbZ$ with nearest neighbor jumps and typically non--exponential holding times, that we call \emph{skeleton process}. 
 %Adapting to the latter the techniques developed in \cite{DGZ} 
 We derive for the latter the LDP for the first--passage times as well  as for the position  % whose rate function admits a different alternative representation compared  to \cite{DRL}
  (cf. Theorem \ref{baxtalo}). % Differently from \cite{DGZ} 
We  also obtain a  detailed qualitative analysis of the rate functions of the above LDPs (cf. Theorem \ref{I_study} and Proposition \ref{J_study}).   The tools developed in this part are fundamental to investigate the Gallavotti--Cohen  symmetry (shortly, GC symmetry) of the form $I (\th)= I(-\th)+c \th$, where $I$ is the LD rate function for the position of the skeleton process, 
$\th \in \bbR$ and $c$ is a suitable constant.
The  GC   symmetry   has been derived in \cite{LLM,LM} 
for Markov random walks on $\bbZ$ with periodic rates of period 2.  These random walks and their large deviations have been analyzed in \cite{LLM} by matrix  methods, allowing to study also an enriched process taking into account the ATP consumed  by the molecular motor. We restrict here to the molecular motor position (i.e. the skeleton process) and show that the GC symmetry pointed out in \cite{LLM} cannot hold for a generic Markov  random walk on a quasi 1d lattice. Indeed, we show that there exists a class of  quasi  1d lattices (called $(\underline v, \overline v)$--\emph{minimal}) such that  the GC  symmetry is verified for any choice of the rates, while outside that class the 
GC  symmetry is violated for Lebesgue any  choice of the rates.
  This result implies that a priori one cannot expect    to observe this symmetry  even if nanotechnology would allow the observations of large deviations. Moreover, it answers the basic question of how universal the GC symmetry discovered by \cite{LLM} in a simple model is. The relevance of both these issues (possible experimental evidence and universality) has been stressed
  also  in \cite{LM}.  In  \cite{FS1}   we will continue our analysis discussing more in detail the connection with the GC functional \cite{LS} and why the validity of the GC   symmetry for the above class of quasi 1d lattices is indeed a consequence of a universal symmetry for algebraic currents \cite{FD}. In \cite{FS1} we will also consider some examples.

We conclude this introduction with some comments on technical aspects.  
When considering Markov random walks the proof  of the position LD principle  is simpler, obtained by   the G\"artner--Ellis theorem  \cite{dH} and by  generalizing the matrix approach introduced by \cite{LLM} (cf. Theorem \ref{GE}).   On the other hand it gives no insight on the mechanism leading to the GC symmetry.
The  results, presented in Section \ref{quasi}, concerning the LD principles for first--passage times and for the position (Theorems \ref{baxtalo} and \ref{I_study}) hold  also for  non--Markov random walks on quasi 1d lattices and indeed for stochastic processes on quasi 1d lattices with a suitable regenerative structure (Theorems \ref{teo2} and \ref{acqua}). More precisely,  
they hold for   stochastic processes $(Z_t)_{t \in \bbR_+}$ obtained as follows.  
Consider a sequence $( w_i, \t_i )_{i \geq 1}$ of i.i.d.   2d vectors with values in
$\bbR \times (0,+\infty)$.   Defining $W_m:= \sum _{i=1}^m w_i$ and $
\cT_m:=\sum_{i=1}^m \t_i$ for $m \geq 0$ integer, set $Z_t:= W _{ \max \{ m\geq 0 :\, \cT_m \leq t \}}$. 
 Sums of i.i.d. random variables have many nice properties and random time changes are not troublesome for what concerns the  LLN and the invariance principles \cite{FS}.  On the other hand, 
 the derivation of the LDP for $(Z_t)_{t \in \bbR_+}$ from the large deviation properties of  $(W_m)_{ m\geq 0}$ and $( \cT_m)_{m \geq 0}$ is much more delicate.   In \cite{DRL} a LDP is  obtained under the condition that  the $\t_i$'s have finite logarithmic moment generating function.  
  This condition  
 is not satisfied when considering Markov random walks on quasi 1d lattices, hence in our case the results  of \cite{DRL}, an the similar ones of \cite{Ru}, cannot be applied. In the context of   LDPs for processes  under  random time changes  we also mention the new progresses  obtained in \cite{LMZ,MSZ}. 
 Restricting  to the case $w_i \in \{-1,1\}$ (which covers the applications to molecular motors), the process 
 $(Z_t)_{t \in \bbR_+}$  becomes a random walk on $\bbZ$ with generic  holding times (not necessarily exponential). Following the main scheme presented in \cite{DGZ} we derive the LDP for the process  $(Z_t)_{t \in \bbR_+}$. We point out some technical issues making our analysis different from \cite{DGZ}:  we allow correlations between $w_i$ and $\t_i$ (absent in \cite{DGZ}), moreover the minimum in the  support of the  law of $\t_i$ can be  zero or positive (the first case is excluded in \cite{DGZ}). Hence, although we have no random environment (thus of course simplifying the analysis) in our case there is a richer scenario for the possible behavior of the rate functions  of the process $(Z_t)_{t\in \bbR_+}$ and of the associated first--passage times, and this behavior has to be investigated and kept in consideration  in order to prove LDPs (see Section  \ref{ospedale}).

The  theorems concerning the GC symmetry are the most innovative ones from a mathematical viewpoint. Using the above LD analysis, in Theorems \ref{giova} and \ref{GC} we prove several characterizations of the  GC  symmetry for  $(Z_t)_{t\in \bbR_+}$, including the fact that it holds  if and only if $w_i$ and $\t_i$ are independent, thus clarifying the probabilistic mechanism leading to the GC symmetry. Using the above characterizations, we study the GC symmetry for Markov random walks (Theorem \ref{teo3}).
The validity of the GC symmetry  for Markov random walks on  $(\underline{v}, \overline{v})$--minimal 1d lattices  is derived by introducing a special path transformation  and comparing the original paths with the transformed ones. On the other hand, the proof of the almost everywhere breaking of the GC symmetry 
  outside the class  of $(\underline{v}, \overline{v})$--minimal quasi 1d lattices   is based on complex analysis methods.

 \section{Random walks on quasi 1d lattices} \label{quasi}

%We now focus on  random walks  on quasi 1D lattices, giving first an abstract definition covering several cases treated in the physical and biophysical literature (e.g. \cite{CMB,DK1,DK2,D,FK,K,KF1,KF2,TF}). 

We start by defining  quasi 1d lattices.
Consider first a finite oriented graph $G= (V,E)$, $V$ being the set of vertices and $E$ being the set of oriented edges,  $E\subset \{ (v,w)\,:\, v \not = w \text{ in } V \}$.  We fix  in $V$ two vertices $\underline{v}, \overline{v}$. 
%We write $E_s$ for the set of unoriented edges, i.e.\
%$$ E_s:=\left\{ \, \{v,w\}\,:\, (v,w) \in E \text{ or } (w,v) \in E \,\right\}\,.
%$$ 
 We assume that  the oriented graph $G$  is connected, i.e. for any $v,w \in V$ there is an oriented path in $G$ from $v $ to $w$. 
 % and assume that $(\underline{v}, \overline{v}) \not \in E$, $(\overline{v}, \underline{v}) \not \in E$.  \club 
 Then the quasi 1d lattice $\cG$ associated to the triple $\bigl(G,\underline{v},\overline{v}\bigr)$ is the oriented graph obtained by gluing together countable copies of $G$ such that   the point $\overline{v}$ of one copy is identified with the point $\underline{v}$ of the next copy. To give a formal definition,  we define $\cG$ as $\cG=(\cV, \cE)$ with vertex set $\cV$ and edge set $\cE$ as follows
 (see Figure \ref{pocoyo100}):
\begin{align*}
& \cV:=\left\{ v_n :=(v,n) \in (V \setminus \{\overline{v}\})\times \bbZ\right\} \,\\
& \cE :=\cE_1 \cup \cE_2 \cup \cE_3\,,
\end{align*}
where 
\begin{align*}
& \cE_1:= \left\{ (v_n, w_n)\,:\, (v,w) \in E \,,\; n \in \bbZ \right \} \,,\\
& \cE_2:= \cup _{n \in \bbZ} \left\{ (v_n, \underline{v}_{n+1} )\,:\, (v,\overline{v}) \in E \right\} \,,\\
&\cE_3:= \cup_{n \in \bbZ}  \left\{ (\underline{v}_{n+1},v_n )\,:\, (\overline{v},v) \in E \right\} \,.
\end{align*}

\begin{figure}[!ht]
    \begin{center}
     \centering
  \mbox{\hbox{
  \includegraphics[width=0.9\textwidth]{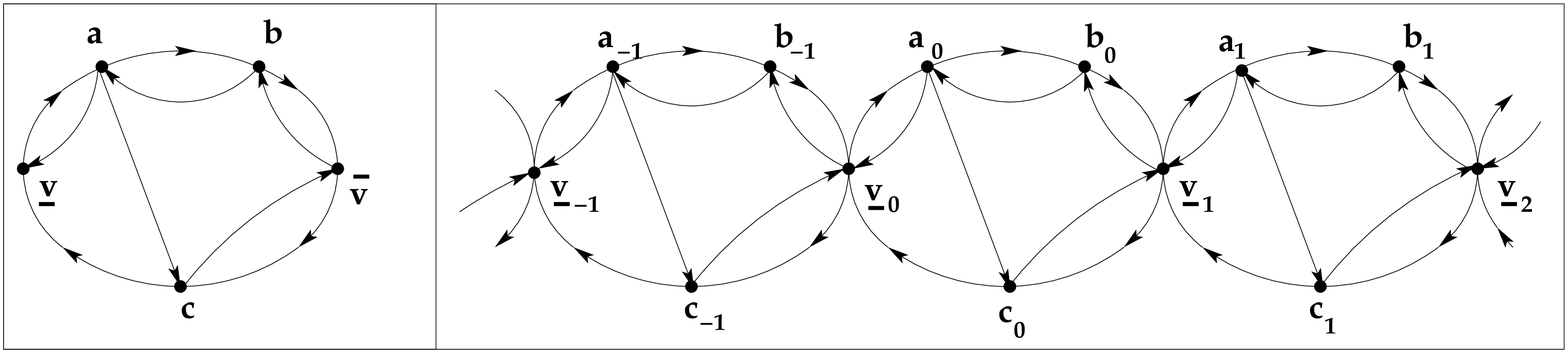}}}
            \end{center}
            \caption{The  graph $G=(V,E)$ with vertices $\underline{v}, \overline{v}$ (left) and the associated quasi 1d lattice $\cG= (\cV, \cE)$ (right) }\label{pocoyo100}
  \end{figure}

To simplify notation we set 
\begin{equation*}
n_*:= \underline{v}_n\,, \qquad n \in \bbZ\,.
\end{equation*}

On the graph $\cG$ we define the shift $\cT: \cV \to \cV$ as $\cT(v_n)= v_{n+1}$. Note that 
the graph $\cG$ is left invariant by the action of $\cT$.
We can now define the class of stochastic processes on quasi 1d lattices we are interested in:

\begin{Definition}\label{corpo1}
 Given a  quasi 1d lattice $\cG$ associated to the triple $\bigl(G,\underline{v},\overline{v}\bigr)$,  %\club \pecetta{in realta' potrebbe avere non exponential holding times}
  we consider a stochastic process   $( X_t)_{t \in \bbR_+}$  with paths in the Skohorod space $D( \bbR_+; \cV)$  starting at any site $n_*$ (we denote by $\bbP_{n_*}$ the associated  law on $D( \bbR_+; \cV)$) 
and fulfilling the following properties: %where $\bbP_{v_n}$ denotes the law of the 
\begin{itemize}
\item[(i)]  for each $ n \in \bbZ$, $\bbP_{n_*}$--a.s., 
  jumps are possible only  along the edges in $\cE$,  %\item[(ii)] % the process  looses  memory of its past when it visits sites of the form $\underline{v}_n$, 
\item[(ii)]  for each $ n \in \bbZ$, when   $(X_t)_{t \in \bbR_+}$ is sampled with law $\bbP_{n_*}$ then the 
law of $( \cT (X_t))_{t \in \bbR_+}$  equals $\bbP_{(n+1)_*}$,
\item[(iii)]  defining $S$ as  the random time 
\begin{equation}\label{pietro}S
  := \inf \left\{ t \geq 0 \,:\, X_t  \in \{-1_*,1_*\}    \right \}\,,
  \end{equation}
it holds $\bbE_{0_*}(S)< \infty$. 
\item[(iv)]  under  $\bbP_{0_*}( \cdot \,|\, X_S=\pm 1_*)$ the random path $(X_{S+t})_{ t \in \bbR_+}$ is independent from $(X_{t})_{ t \in [0,S]}$ and has law 
$\bbP_{\pm 1_*}$.
\end{itemize}
\end{Definition}

  \bigskip
 In the applications, typically 
    $ (X_t)_{t \in \bbR_+}$  is a Markov  random walk:
    
    \begin{Lemma}
    Let  $ (X_t)_{t \in \bbR_+}$  be a Markov  random walk   with state space $\cV$ and with positive  jump rates $r(x,y)$,  $(x,y) \in \cE$, such that 
 \begin{equation}\label{simm1}
r(x,y)= r( \cT x, \cT y)\,. 
 \end{equation}
 Then the above random walk is well defined for all times $t$ (no explosion takes place), fulfills the properties of Definition \ref{corpo1} and moreover $\bbE_{0_*}( e^{\l S}) <+\infty$ for  $\l >0$ small enough.
 \end{Lemma}
 The proof of the above lemma is simple and therefore omitted. 
 The finite exponential moments  for $\l$ small follow from the exponential decay of hitting probabilities for irreducible Markov chains with finite state space.
 
 \smallskip

  We point out  that in the applications another relevant example is given by a random walk   $( X_t)_{t \in \bbR_+} $  on the graph $\cG$ with non exponential holding times (cf. \cite{KF1}).

\smallskip

 %In what follows, we write $\bbP_{v_n}$ for the law on the Skohorod space $D( \bbR_+, \cV)$ of the random walk starting at $v_n \in \cV$.
Note that the states $n_*$'s behave as gates which have to be crossed by the stochastic process $X_t$ in order to move from one fundamental cell to the neighboring  ones in the quasi 1d lattice $\cG$. In the applications to molecular motors, %the graph $\cG=(\cV,\cE)$ with the associated jump rates encodes the mechanochemical cycle of the molecular motors. Each 
each site  $n_*$ corresponds to a spot in the $n$--monomer of the polymeric filament where the molecular motor can bind. The other states $v_n$ correspond to intermediate conformational states that  the molecular motor
achieves in its mechanochemical transformations, which are described by jumps along edges in $\cE$.  In particular,  states  $v_n$  do not  encode only a spatial position and jumps do not necessarily correspond to spatial jumps.

\bigskip  %From now on we consider the random walk $X$ starting at $ \underline{v}_0$.
We   now introduce the fundamental object of our investigation:

\begin{Definition}\label{corpo2} Given the stochastic process $X$ as in Definition \ref{corpo1},  the \emph{skeleton process} 
 $X^*= (X_t^*)_{t \in \bbR_+}$
 is defined as  $ X_t^*:= \Phi(X_\iota)$ where $\Phi(n_*)=n$ and 
 $$\iota := \sup \left\{ s \in [0,t]: X_t=n_* \text{ for some } n \in \bbZ\right\}\,.$$
 $X^*_t$    has values in   $\bbZ$ and  records the last visited 
state of the form  $n_*$ up to time $t$.
\end{Definition}

In the applications to molecular motors,  the process $(X^*_t)_{t \in \bbR_+} $ contains  all the relevant information, indeed it allows to determine the position of the molecular motor up to an error of the same order of the monomer size. %Our main results concern the asymptotic behavior of  the skeleton process (law of large numbers, invariance principe and large deviations principle) as well as the analysis of a Gallavotti--Cohen type symmetry.

 % We first explain how this process enters in the abstract class introduced in the previous section. 
%To this aim,
 % let $S$ be the random time $$S
 % := \inf \left\{ t \geq 0 \,:\, X_t  \in \{-1_*,1_*\}    \right \}%
%$$
%(recall that, if not  mentioned differently,  the random walk $X_t$  starts at $0_*$).
%Note that $X^*_S = X_S$ by definition.
%Consider  the random  vector $(X_S,S)\in \{-1,1\} \times (0,+\infty)$ where $\{ X_t\}_{t \in \bbR_+} $ is sampled according to $\bbP_{0_*}$ 
%As a  consequence of Theorem \ref{teo1}, Theorem \ref{teo2} and the above lemma we get

\section{Main results for random walks  on quasi 1d lattices}\label{panettone}
 Let $S$ be the random time defined in \eqref{pietro}.  As proved in \cite{FS}, one can easily obtained a strong   law of large numbers for the skeleton process since $\bbE_{0_*}(S) < +\infty$ (cf. Theorems 1 and 2 in \cite{FS}):
 
\begin{equation}
\lim _{t \to \infty} \frac{X^*_t}{t}=
 \frac{  \bbP_{0_*}  (X_S=1_*)-\bbP _{0_*} (X_S=-1_*)}{\bbE_{0_*}(S)}=:v\,, \qquad \qquad \text{$\bbP_{0_*}$-a.s.}
 \,.\label{trenino}
\end{equation}
In \cite{FS} we study also the gaussian fluctuations of the skeleton process, proving  an invariance principle  if $\bbE_{0_*}(S^2) < +\infty$. We concentrate here on large deviations. 

\subsection{Large deviations}

From now on, in addition to the requirements in Definition  \ref{corpo1},  we assume that 
\begin{equation}
 \bbP_{0_*}  (X_S=1_*)>0 \text{ and }\bbP _{0_*} (X_S=-1_*)>0\,,
 \end{equation}
which holds for  molecular motors.

\begin{TheoremA}\label{baxtalo}
Consider the process  $( X_t )_{t \in \bbR_+}$ starting at $0_*$.    
 Call   $T_n$ the first time the skeleton process hits $n\in \bbZ$, i.e.
 \begin{equation}\label{chejo}
 T_n := \inf\left\{ t \in \bbR_+\,:\, X_t^* =n \right\} \in [0,+ \infty]\,.
 \end{equation}
 Then  the following holds:
 
 \begin{itemize}
 \item[(i)]
 As $n \to \pm \infty$ the random  variables  $T_n\,/\,|n|$ 
satisfy a LDP with speed $|n|$ and convex rate function
 \begin{equation}\label{pocoyo1}
 J_\pm (u):= \sup_{ \lambda \in \bbR} \left\{ \lambda u - \log \varphi_\pm (\lambda) \right\}\,, \qquad u \in \bbR \,, \end{equation}
where  
\begin{equation} \label{pocoyo2}
 \varphi_\pm (\lambda ):= \bbE_{0_*}\left( e^{\lambda T_{\pm 1}} \mathds{1}(T_{\pm 1} <\infty )\right) \in (0,+\infty]\,, \qquad \l \in \bbR\,.
\end{equation}
The rate function $J_\pm$ is  good\footnote{We use the same terminology of \cite{DZ} } if and only if $\bbP_{0_*}(T_1 < \infty)\not = \bbP_{0_*}(T_{-1}< \infty)$.

\item[(ii)]
As $t \to +\infty$,  the random variables   $X^*_t /t$ satisfy a LDP with speed $t$ and good  and convex rate function $I$
given by 
\begin{equation}% \label{rate_f100}
I(\th)= 
\begin{cases}
\th J_+ (1 / \th) & \text{ if }\qquad  \th >0\,,\\
|\th |J_-(1 / |\th |) & \text{ if }\qquad  \th  <0\,,\\
\end{cases}
\end{equation}
and $I(0) = \lim _{\th \to 0} I(\th )$.
\end{itemize}
\end{TheoremA}
Theorem \ref{baxtalo} is an immediate consequence of Lemma \ref{torlonia} and Theorem \ref{teo2}
 in Section \ref{viaggio_astratto}.
Since $T_n$ can take value $+\infty$, the meaning of the LDP for $T_{n}/|n|$ as $n \to \pm \infty$ is the following: for each close subset $\cC \subset \bbR$ and each open subset $\cO \subset \bbR$ it holds \begin{align*}
\limsup _{n \to \pm \infty } \frac{1}{|n| } \log \bbP_{0_*} \left( \frac{T_n}{|n|} \in \cC\right)& \leq - \inf_{\cC } J_\pm\,,    \\
 \liminf _{n \to \pm \infty}  \frac{1}{|n| } \log \bbP_{0_*}\left ( \frac{T_n}{|n|} \in \cO\right )& \geq - \inf_{\cO} J_\pm\,.
\end{align*}
  %We point out that Assumption \eqref{sar_san?}  is the only relevant case in the applications. The degenerate cases $ \bbP_{0_*} ( X_S= 1_* )=1$,  $  \bbP _{0_*}( X_S= -1_* )=1$ are rather trivial and can be treated easily.

%\medskip
We now collect information on the qualitative behavior of the rate function $I$.  The qualitative behavior of the rate functions $J_-,J_+$ is described in Proposition \ref{J_study} in Section \ref{montagna}. Here we concentrate on the rate function $I$ since the large deviations of $X_t^*/t$ are more relevant in the applications.

\begin{Definition} We define 
$\a_\pm$ as the minimum of the support of the law of $T_{\pm 1}$.
\end{Definition}
We point out that   $\a_\pm$ is the minimum of  the support of the Borel measure $A\to \bbP_{0_*}( S \in A \,,\; X_S= \pm 1_*)$ (see  Prop. \ref{sirius_black} in Section \ref{viaggio_astratto}).
 Below $1/\a_\pm$ is intended to be $+\infty$ if $\a_\pm = 0$. Note that $\a_\pm=0$ in the case of Markov  random walks.

\begin{TheoremA}\label{I_study} 
   The  following holds:
 \begin{itemize}
\item[(i)] $I$ is infinite outside $\big[-\frac{1}{\a_-} , \frac{1}{\a_+} \big]$, $I$ is   finite and $C^1$ on $\big( -\frac{1}{\a_-} , \frac{1}{\a_+} \big)$, moreover it is 
 smooth on $(-1/\a_-,1/\a_+)\setminus\{0\}$.

\item[(ii)] The following  holds:
%\begin{itemize}
%\item[$\bullet$] If $\bbP(T_1 =\a_+) > 0$ and  $\bbP (T_{-1} = \a_-) > 0$ then 
	\begin{align}
	 \lim_{ \th \nearrow \frac{1}{\a_+} } I(\th ) & = 
	 \begin{cases}
	 +\infty  \;\;\;\;&\mbox{if }\;\bbP_{0_*}(T_1 =\a_+) = 0 \\
	 I \big( \frac{1}{\a_+} \big) < \infty\;\; \;\;&\mbox{otherwise.}
	 \end{cases} \label{destra}\\
	\lim_{ \th \searrow -\frac{1}{\a_-} } I(\th )  & = 
	 \begin{cases}
	 +\infty  &\mbox{if }\; \bbP_{0_*}(T_{-1} =\a_-) = 0 \\
	 I \big( -\frac{1}{\a_-} \big) < \infty &\mbox{otherwise.} \label{sinistra}
	 \end{cases}
	\end{align}
%and $I(\th ) = +\infty$ for $\th <-\frac{1}{\a_-}$ or  $ \th > \frac{1}{\a_+} $.
%In particular $I$ has compact level sets, so it is a good rate function.
\item[(iii)] The derivative of $I$ satisfies $\ds \lim_{\th \nearrow \frac{1}{\a_+} } I'(\th ) = +\infty $ and $\ds \lim_{\th \searrow -\frac{1}{\a_-} } I'(\th ) = -\infty $. 
\item[(iv)] $I$ is lower semicontinuous and  convex on $\bbR$, it is strictly convex on $\big( -\frac{1}{\a_-} , \frac{1}{\a_+} \big)$.
\item[(v)] $I$ 
 has a unique global minimum, which is  given by $0$ and is attained at  $v\in (-1/\a_-, 1/\a_+)$, where $v$ is the asymptotic velocity defined in \eqref{trenino}. Moreover $I$ is strictly decreasing on $(-1/\a_-, v)$ and is strictly increasing on $(v, 1/\a_+)$.
\end{itemize}
\end{TheoremA}
 Theorem \ref{I_study} is an immediate consequence of Lemma \ref{torlonia} and Theorem \ref{acqua} in Section \ref{viaggio_astratto}.

\medskip
When the process $X$ is a Markov random walk with periodic rates (i.e. satisfying \eqref{simm1}), the derivation of the large deviation principle is simpler. In this case given $x \in \cV$  we set 
$r(x):= \sum_{y: (x,y) \in \cE} r(x,y)$ and, 
given $v \not = w$ in $ V \setminus \{ \overline{v}\}$, we set (using the convention that $r(y,z)=0$ if $(y,z) \not \in \cE$)
\begin{equation*}%\label{venticello}
 r(v):= r(v_n) \,, \;\;  r_-(w,v):= r(w_{n-1},v_n)\,,\;\;
 r_0(w,v):=r( w,v)\,, \; \; r_+(w,v):= r(w_{n+1},v_n)\,. \end{equation*}
  The above  definition is well posed due to \eqref{simm1}.
Finally, given $\l \in \bbR$ we 
consider the finite  matrix $\cA(\l)$, with entries parameterized by $( V\setminus\{\overline{v}\} ) \times (V\setminus\{\overline{v}\} )$, defined as :
\begin{equation}\label{Freitag} \cA _{v,w}(\l) := 
\begin{cases} -r(v) & \text{ if } v=w \,,\\
e^\l r_-(w,v)  +r_0(w,v)+e^{-\l} r_+(w,v) & \text{ if } v \not = w\,.
\end{cases}
\end{equation}
Applying G\"aertner--Ellis Theorem we will derive the following result:
\begin{TheoremA}\label{GE}
%Suppose that the process $X$ is a  random walk with exponential holding times. 
Suppose $X$ is a Markov random walk on the quasi 1d lattice $\cG= (\cV, \cE)$ with transition rates satisfying \eqref{simm1}. 
Then, as $t \to +\infty$, the random variables  $ X_t^*/t$ satisfy  a large deviation principle with speed $t$  and  convex and good  rate function  $I(\th )$ given by 
$$ I(\th)=\sup_{ \l \in \bbR} \{ \th \l - \L(\l) \} \,, \qquad  \th \in \bbR $$
where $\L(\l)$ is the  finite value
$$ \L(\l):= 
\max \{ \Re (\g) \,:\, \g \text{ eigenvalue of } \cA (\l) \}$$
and the matrix $\cA(\l)$ is defined \eqref{Freitag}.
\end{TheoremA}

\subsection{Fluctuation theorems (Gallavotti--Cohen type symmetry)}\label{amiciGC}

\begin{TheoremA}\label{giova}
The following  facts are equivalent:
\begin{itemize}
\item[(i)] For some  $c \in \bbR$  the Gallavotti--Cohen type symmetry\footnote{Sum is thought in $[0,+\infty]$}
$I( \th)=I(-\th)+c \th$
%\end{equation}
holds for all $\th \in \bbR$;
\item[(ii)] The random variables $X_{S}$ and $S$ are independent. 
\end{itemize}
Moreover, when (i),(ii) hold it must be $c= \log \frac{\bbP_{0_*} (X_S=-1_*)}{\bbP_{0_*} (X_S=1_*)}$.
\end{TheoremA}
 Theorem \ref{giova} is an immediate consequence of Lemma \ref{torlonia} and Theorem \ref{GC} in Section \ref{viaggio_astratto}.

\medskip

We continue our investigation of the Gallavotti--Cohen type symmetry (shortly, GC symmetry)  as in Theorem \ref{giova} restricting now    to Markov  random walks $( X_t)_{t \in \bbR_+}$ on quasi 1d lattices.  Recall that we write 
$r(x,y)$, $(x,y) \in \cE$, for the positive jump rates of the Markov  random walk and assume the periodicity \eqref{simm1} to hold.
We restrict our discussion to the case of fundamental graphs $G=(V,E)$ such that 
\begin{equation}\label{mercato}
(x,y) \in E \text{ if and only if  }(y,x) \in E\,,
\end{equation} which is  the standard setting in the investigation of GC symmetry for Markov chains \cite{LS}.

In what follows, given an edge  $(u,v) \in E$  in the fundamental graph $G=(V,E)$,  we define 
\begin{equation}\label{def_r_G} r(u,v) = r (\pi(u), \pi(v) )\,,
\end{equation} where $\pi$ is the map $ V \to \cV$ defined as $\pi(u)= u_0$ if $ u\not = \overline{v}$ and $\pi(\overline{v}) = \underline{v}_1$. Note that, fixed positive numbers $a(e)$, $e\in E$, it is univocally determined a Markov  random walk on $\cG$ whose rates satisfy \eqref{simm1} and such that $r(e) = a(e)$ for all $e \in E$. We call it the Markov random walk induced by $a(e)$, $e \in E$.

% Note that $\pi$ defines a graph isomorphism between $G$ and $G_0=(V_0,E_0)$ where $V_0=\left\{ u_0\,:\, u \in V \setminus \{\bar v\} \right\}\cup \bigl\{ \underline{v}_1 \} $ and $E_0 = \bigl\{ (x,y) \in \cE\,:\, x,y \in V_0\}$. 

We introduce a special class of graphs $G$ which includes in particular trees. Recall that $G$ has connected support when disregarding the orientation of the edges.

\begin{Definition}\label{pasopanori}
We say that the graph  $G= (V,E)$ is \emph{$(\underline{v}, \overline{v})$--minimal} if it   satisfies  \eqref{mercato} and moreover there is  a unique path $\g_*= (z_0,z_1, \dots, z_n)$ such that $\mathrm{(i)}$ $z_0= \underline{v}$,  $\mathrm{(ii)}$ $z_n = \overline{v}$, $\mathrm{(iii)}$ $(z_i,z_{i+1})\in E$ and $\mathrm{(iv)}$ the points $z_0,\dots, z_n$ are all distinct.
\end{Definition}
Note that, given a generic fundamental graph $G=(V,E)$, there exists at least one path $\g= (z_0,z_1, \dots, z_n)$ 
satisfying the above properties (i),...,(iv). Indeed, since $G$ is connected, there exists a path from $\underline{v}$ to $\overline{v}$. Take such a path and prune iteratively the loops. Since each time a loop is pruned away   the length of the path decreases, after a finite number of prunes one gets a minimal path satisfying the above properties (i),...,(iv).

Now suppose that $G=(V,E)$ is  \emph{$(\underline{v}, \overline{v})$--minimal} and take two points $z_i \not = z_j$ (the $z_k$'s are as in the Def. \ref{pasopanori}). Then it cannot exist  a path from $z_i $ to $z_j$ whose intermediate points are in $V \setminus \{z_0,z_1, \dots, z_n\}$. In particular, the graph $G$ must be as in Fig. \ref{sinti} (due to property \eqref{mercato} we only draw the support of $G$, disregarding orientation).  More precisely, the graph is given by the linear chain $\g_*$ of Def. \ref{pasopanori}, to which one attaches some subgraphs, in such a way that each attached subgraph has exactly one point in common with $\g_*$.% We call $G_1, G_2, G_k$ the above attached subgraphs and 
 \begin{figure}[!ht]
    \begin{center}
     \centering
  \mbox{\hbox{
  \includegraphics[width=0.5\textwidth]{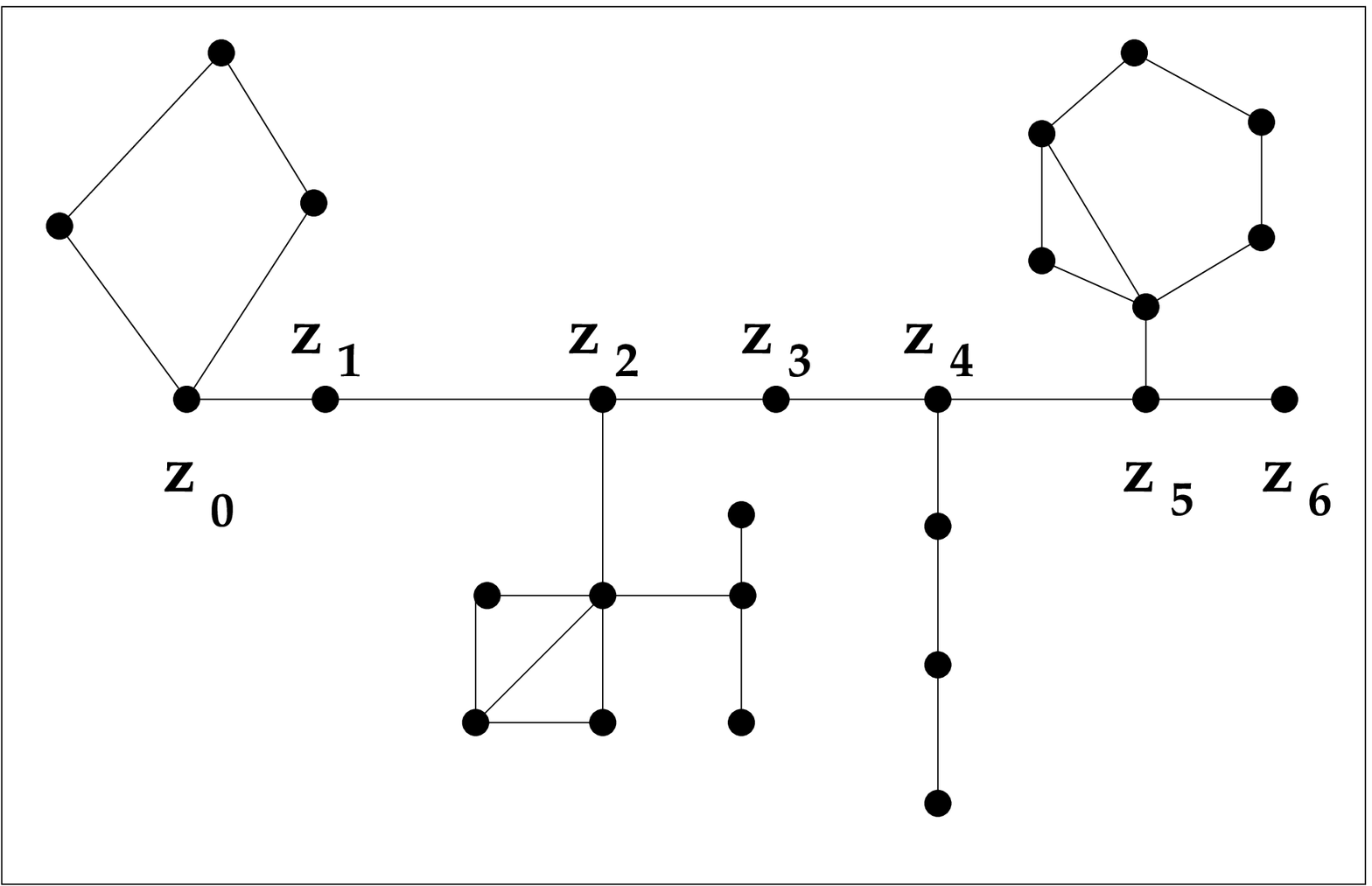}}}
            \end{center}
            \caption{A  \emph{$(\underline{v}, \overline{v})$--minimal}  graph $G=(V,E)$ }\label{sinti}
  \end{figure}

 \begin{TheoremA}\label{teo3}
Suppose that $X$ is a Markov random walk and that $G= (V,E)$ is a graph satisfying \eqref{mercato}.
If $G$ is \emph{$(\underline{v}, \overline{v})$--minimal}, then 
 the  random variables  $S$ and $X_S^*$ are independent, and in particular the large deviation  rate function $I$ associated to the skeleton process $X^*$ satisfies the Gallavotti--Cohen type symmetry 
\begin{equation}\label{verabila}
I(\th)=I(-\th)-\D \,\th\,,\qquad \forall \th\in \bbR\,,
\end{equation}
where 
\begin{equation}\label{alexey}
 \D= \log \frac{ r(z_0,  z_1) r(z_1, z_2) \cdots r(z_{n-1},z_n ) }{  r(z_1,z_0) r(z_2,z_1) \cdots r(z_n,z_{n-1} ) }
\end{equation}
and $\g_*=(z_0,z_1,z_2, \dots, z_{n-1} ,z_n ) $ is the path in Definition \ref{pasopanori}.

Vice versa,  if $G$ is not \emph{$(\underline{v}, \overline{v})$}--minimal then the vectors  $\bigl( r(e) \,: e \in E\bigr) \in  (0,+\infty)^E$   for which the induced random walk on $\cG$ satisfies  \eqref{verabila} 
for some constant $\D$ (depending on the numbers  $r(e)$, $e \in E$)  has zero Lebesgue measure  in $(0,+\infty)^E$.

\end{TheoremA}
The proof of the above theorem is given in Section \ref{secGC_bis}.

%%%%%%%%%%%%%%%%%%%%%%%%%%%%

%%%%%%%%%%%%%%%%%%%%%%%%%%%%%%%%%%%%%%

%%%%%%%%%%%%%%%%%%%%%%%%%%%%%%%%%%%%%%%%

\section{Random time changes of cumulative processes}\label{viaggio_astratto}

As already mentioned, the results presented in Section \ref{panettone} hold in a more general context that we now describe.
Consider a sequence $( w_i, \t_i )_{i \geq 1}$ of i.i.d.   2d vectors with values in
$\bbR \times (0,+\infty)$.  For each 
integer $ m \geq 1$ we define 
\begin{align}
&W_m:= w_1+w_2 + \cdots + w_m\,,\\
&\cT_m:= \t_1+\t_2 + \cdots+ \t_m\,.
\end{align}
We set $W_0=\cT_0=0$.
Note that 
%y the law of large numbers, $ \lim _{m \to \infty} \cT_m/m= \bbE \t_i$ a.s.,
 %thus implying that  
$ \lim _{m \to \infty} \cT_m= +\infty$ a.s.  As a consequence, we can univocally define a.s. a random process $\left\{\nu(t)   \right\}_{t \in \bbR_+}$ with values in $\{0,1,2,3, \dots\}$ such that
\begin{equation}\label{warwick}
\cT_{\nu(t)}\leq t < \cT_{\nu(t)+1 }\,, \qquad t \geq 0 
\,.
\end{equation}
Note that $\nu(t)= \max \{ m\in \bbN :\, \cT_m \leq t \}$.
Finally, we define the process $Z:[0,\infty) \to \bbR$ as
\begin{equation}\label{zorro} Z_t:= W_{\nu(t) }\,.
\end{equation}
Note that $Z_0=0$. 
The resulting process $Z= (Z_t)_{t \in \bbR_+}$   is therefore obtained from the cumulative process $(W_m)_{m \geq 0}$ by  a random time change, and  generalizes the concept of (time--homogeneous)  random walk on $\bbR$. For example, if $w_i $ and $\t_i$ are independent and $\t_i$ is an exponential variable of parameter $\l$, then the process $Z$ is a continuous time Markov  random walk with exponential holding times of parameter $\l$ and 
with jump probability given by the law of $w_i$. If $\t_i \equiv 1$ for all $i$, then $Z_t= W_{\lfloor t \rfloor }$ ($\lfloor \cdot \rfloor$ denoting the integer part) and  $(Z_n)_{n \in \bbN}$ 
 is a discrete time  Markov random walk on $\bbR$ with jump probability given by the law of $w_i$.

\smallskip

Due to Definitions \ref{corpo1} and \ref{corpo2}   the skeleton process $X^*$ is indeed a special case of  process $Z$ (recall the definition of the random time $S$ given in \eqref{pietro}):
%%%%%%%%%%%%%%%%%%%%
\begin{Lemma}\label{torlonia}
Consider a sequence $( w_i, \t_i )_{i \geq 1}$ of i.i.d.   vectors, with the same law of  the random  vector $\bigl(X^*_S,S\bigr)\in \{-1,1\} \times (0,+\infty)$   when the random walk $(X_t)_{t \in \bbR_+}$ 
starts at $0_*$. We define $(Z_t )_{t \in \bbR_+}$ as the stochastic process built from  $( w_i, \t_i )_{i \geq 1}$ according to \eqref{zorro}. Then 
$(Z_t)_{t \in \bbR_+}$ has the same law of $( X_t^*)_{t \in \bbR_+} $ with $X^*_0= 0$. \end{Lemma}
The proof of the above lemma is  very simple and therefore omitted.
	 We recall also the LLN discussed in \cite{FS}[Appendix A]:
 \begin{Proposition}[FS]\label{teo1}  If  $ \bbE( \t_i) < \infty$, then almost 
 surely $ \lim_{t \to \infty} \frac{Z_t}{t} = v:= \frac{ \bbE (w_i) }{\bbE (\t_i)}$.
\end{Proposition}

 We now  state our main results for $(Z_t)_{t \in \bbR_+}$:
\begin{TheoremA} \label{teo2} [LDP]
Suppose that 
\begin{itemize}

\item[(A1)] $w_i \in \{-1,1\}$ a.s. 

\item[(A2)] $\bbP( w_1=+1)>0$ and $\bbP( w_1=-1)>0$.

%\item[(A3)] \club 
\end{itemize}
Set $T_n := \inf \left\{ t\in \bbR_+\,:\, Z_t=n \right\}\in [0,+\infty]$. Define $J_\pm$ and $\varphi_\pm$ as in \eqref{pocoyo1} and \eqref{pocoyo2}. 
Then  the following holds:
\begin{itemize}
\item[(i)]
 As $n \to \pm \infty$ the random  variables  $T_n\,/\,|n|$ 
satisfy a LDP with speed $|n|$ and convex rate function $J_\pm$. 
The rate function $J_\pm$ is  good  if and only if $\bbP(T_1 < \infty)\not = \bbP(T_{-1}< \infty)$.
 
 \item[(ii)]
 As $t \to +\infty$,  the random variables   $Z_t /t$ satisfy a LDP with speed $t$ and good and convex rate function $I$
given by  
\begin{equation} \label{rate_f100}
I(\th)= 
\begin{cases}
\th J_+ (1 / \th) & \text{ if }\qquad  \th >0\,,\\
|\th |J_-(1 / |\th |) & \text{ if }\qquad  \th  <0\,,\\
\end{cases}
\end{equation}
and $I(0) = \lim _{\th \to 0} I(\th )$.
\end{itemize}
 \end{TheoremA} 
 The proof of the above result is given in Sections \ref{montagna} and \ref{mec_yek3}. 
%Above we have used the terminology of \cite{dH}. In particular,

 %%%%%%%%%%  ho spostato

%\centerline{\club \club Modificare parte rimanente \club \club}

We  introduce the  functions $f_\pm$ on $\bbR$ as 
	\begin{equation}\label{cioccolato}
   f_\pm (\lambda ) := \bbE \big( e^{\lambda \tau_1} \mathds{1}(w_1 =\pm 1) \big) \in (0,+\infty] \,.
 \end{equation}
 Note that $f_\pm(\l) >0$ under  Assumption (A2). We call $\a_\pm$ the minimum value in the support of the law of $T_{\pm 1}$.

%%%%%%%%%%%%%%%%%%%%%%%%%%%%%%%%%%%%%%%%% 
\begin{Proposition}\label{sirius_black} 
Under Assumptions (A1) and (A2)  of the previous theorem the following holds:

\begin{itemize}
\item[(i)]  The function $\varphi_\pm(\l)$ satisfies
\begin{equation} \label{cena}
\varphi_\pm (\l) = \frac{ 1 - \sqrt{ 1- 4 f_-(\l)f_+(\l) } }{ 2 f_\mp (\l) }
\end{equation}
for $\l \leq \l_c$, where $\l_c\in [0, +\infty )$ is the unique value of $\l$ such that $ f_-(\l)f_+(\l) = 1/4$, while $\varphi_\pm (\l) =+\infty$ for $\l > \l_c$.

\item[(ii)]
Consider the measure $\mu_\pm$ on $[0,+\infty)$ such that $\mu_\pm (A)= \bbP( \t_1 \in A, w_1=\pm 1)$ for any Borel $A \subset \bbR$. Then $\a_\pm$ is the minimum  value in  the support of $\mu_\pm$.
Moreover $\bbP(T_{\pm 1}= \a_\pm)= \bbP (\t_1= \a_\pm , w_1=\pm 1)$. 

\end{itemize}
\end{Proposition}
The proof is given at the beginning of  Subsection \ref{ospedale}.

\medskip

The qualitative behavior of the rate function $I(\th)$ is described by the following theorem (for the qualitative behavior of $J_\pm$ see Proposition \ref{J_study}):
\begin{TheoremA}\label{acqua}
Theorem \ref{I_study} remains valid in the present more general context, with $v$ defined as in Proposition \ref{teo1}.\end{TheoremA}
We conclude with a result on the presence of a  Gallavotti--Cohen type symmetry in the rate function $I$:

\begin{TheoremA} \label{GC}
%There exists a constant $c \in \bbR$ such that $I( \th)-I(-\th)=c \th $ if and only if $w_i, \t_i$ are independent.
The following  facts are equivalent:
\begin{itemize}
\item[(i)] There exists a constant $c \in \bbR$ such that the Gallavotti--Cohen type symmetry
\begin{equation} \label{casa1}
I( \th)=I(-\th)+c \th
\end{equation}
holds for all $\th \in \bbR$;
\item[(ii)] fixed $i$, the random variables $w_i, \t_i$ are independent; 
\item[(iii)] the functions $\f_+(\l)$ and $\f_-(\l)$ are proportional where finite, that is: 
\[\f_+(\l) = C \f_-(\l) \quad \mbox{ for all }  \l \leq \l_c \, . \]
\end{itemize}
Moreover, if we let $ p := \bbP (w_i =1)$ and $q := \bbP(w_i = -1)$ (with $p,q>0$  by Assumption (A2)), then $C = p/q$ and $c=\log(q/p)=-\log C$.
%\item[(iv)] the random variable $\t_1$ under $\bbP( \cdot\,|\, =1)$
\end{TheoremA}
The proof of this result is given in Section \ref{secGC}.

%%%%%%%%%%%%%%%%%%%%%%%%%%%%%%%%%%%%%%%%%

\section{Proof of Theorem \ref{teo2}--(i)  and Theorem \ref{acqua} }\label{montagna}
In this section we prove Item (i) of Theorem \ref{teo2} and we study the behavior of the functions $I,J_\pm$ defined in Theorem \ref{teo2}. In particular, we prove Theorem \ref{acqua} at the end of this section.

%The proof relies on the hitting time method (cf. \cite{DGZ}, \cite{dH}[Ch. VII]).  It is divided in three parts: proof of Item (i) with exception of the possible goodness of $J_\pm$ (Section \ref{mec_yek1}), qualitative study of the functions $J_\pm, I$ (Section \ref{ospedale}) and finally the proof of Item (ii) (Section \ref{mec_yek3}). 

\subsection{Proof of Item (i) in Theorem \ref{teo2}}\label{mec_yek1}
For $n \geq 1$ the random variable  $T_n$ has the same law of $\sum _{k=1}^n \hat \t_n$, where $\hat \t_n $'s are i.i.d. random variables taking value in $(0,+\infty]$, distributed as  $T_1$.
%We can represent $T_n$ as    $T_n = \sum _{k=1}^n \hat \t_n$, where 
% one has that $\hat \t_n$ represents the time required for the process $Z$ to reach the site $n$ after
 % its first visit to the site $n-1$. In particular, the %  and Cramer's Theorem gives a LDP for $T_n/n$ as $n \to \infty$. 
 The thesis can then be obtained from 
 Cram\'er Theorem. 
We give the proof in the case $n\to \infty$. Call $\alpha:= \bbP (T_1<\infty)$ and note that $P(T_n <\infty)= \alpha^n$. 
 Then  for each subset $\cA \subset \bbR$  we can write  $  \bbP( T_n /n \in \cA)=\alpha ^n  \bbP(T_n/n \in \cA \,| \,T_n<\infty)$.
 Now we observe that, conditioning on the event $T_n <\infty$, $T_n $ can be represented as $ \sum _{k=1}^n \hat \t_n'$, where the real  random variables  $\t_n'$ are i.i.d. and distributed as $T_1$ conditioned to be finite. In conclusion
 $\bbP(T_n/n \in \cA | T_n<\infty) = P( \frac{1}{n} \sum_{k=1}^n \hat \t_n' \in \cA)$. At this point one only need to apply 
 Cram\'er Theorem for i.i.d. real random variables (cf. \cite{DZ}[Th. 2.2.3]) observing that the moment generating function of $\hat \t_n'$ is $\varphi_+/\a$. 
The fact that $J_\pm$ is good if and only if 
 $\bbP (T_1 < \infty ) \neq \bbP (T_{-1} < \infty )$ is proved in the next Subsection (see Remark \ref{Jgood} below).

\subsection{Qualitative study of the functions $J_\pm (\th ) $, $I(\th )$}\label{ospedale} 
In this subsection we  first   prove some properties of the   function $I(\th)$ defined in Theorem \ref{teo2} by \eqref{rate_f100} and the identity $I(0) = \lim _{\vartheta \to 0} I(\vartheta)$. In the next subsection we will indeed prove that $I(\vartheta)$ is the rate function of the LDP for $Z_t/t$.

\smallskip

We start by  proving Proposition \ref{sirius_black}.

\begin{proof}[Proof of Proposition \ref{sirius_black}]
Let us prove Point (i). Recall the definition of the positive functions  $f_\pm$ given in \eqref{cioccolato}.
Distinguishing on the value of $w_1$ we can write
 \begin{equation}\label{rinnovo99}
T_1=   \mathds{1}( w_1=1) \t_1 +  \mathds{1}(w_1=-1)(\t_1+ T_1'+T_1'')%= \t_1+  \mathds{1}(w_1=-1)( T_1'+T_1'')
\end{equation}
where $T_1', T_1''$ are independent random variables, independent from $w_1, \t_1$ and distributed as $T_1$  (roughly, $T_1'$ is the time  for $Z$ to go from $-1$ to $0$ and $T_1''$ is the time for $Z$ to go from $0$ to $1$).
The above identity implies that
\begin{equation}\label{base} 
\varphi_+(\l) = f_+(\l) + f_-(\l) \varphi_+(\l)^2 \,.
\end{equation}
From this we deduce that $\varphi_+(\l) < +\infty$ if and only if $f_+(\l) f_-(\l) \leq 1/4$, and moreover in this case \eqref{cena} holds. By symmetry, one obtains that the same condition implies $\varphi_-(\l) <+\infty$. Trivially $f_\pm$ is increasing, 
 $\lim _{\l \to -\infty} f_\pm(\l)=0$ and $\lim_{\l \to +\infty} f_\pm (\l) =+\infty$.  Moreover, $f_\pm (\l)$ is smooth and strictly  increasing on the open set $\{ f_\pm < +\infty\}$.
 % (see the proof of Lemma 2.2.5--(c) in \cite{DZ}). 
 As a consequence there exists a unique value $\l_c\in \bbR$ such that $f_+(\l) f_-(\l) = 1/4$ and therefore $f_+(\l) f_-(\l) \leq 1/4$ for $\l \leq \l_c$.  Since trivially $\varphi_+(\l) < +\infty$ for $\l \leq 0$ it must be $\l_c \geq 0$. This completes the proof of Point (i).

We now move to Point (ii).
To see that $\bbP ( \t_1 < \a_+ , w_1 = 1 ) = 0$ observe that by \eqref{rinnovo99} $\{ w_1 = 1 \} \subseteq \{ T_1 = \t_1 \}$ and therefore 
$ \bbP ( \t_1 < \a_+ , w_1 = 1 ) = \bbP ( T_1 < \a_+ , w_1 = 1 ) \leq \bbP ( T_1 < \a_+ ) =0$. 
To get the thesis it remains to show that $\forall \e > 0$ $\bbP ( \t_1 \in [ \a_+ , \a_+ + \e ) , w_1 = 1 ) >0$.
Assume the contrary. Then there exists $\hat{\e} >0 $ such that $\bbP ( \t_1 \in [ \a_+ , \a_+ +\hat{ \e} ) , w_1 = 1 ) =0$.
By definition of $\a_+$, on the other hand, we have $\bbP ( T_1 \in [ \a_+ , \a_+ + \hat{\e} )  ) 
%= \bbP ( T_1 <\a_+ + \hat{\e}   ) 
> 0$. Combining this with the decomposition in \eqref{rinnovo99},  we find
	\[ \begin{split}
	0 < \bbP ( T_1 <  \a_+ + \hat{\e} ) &
	= \bbP ( \t_1 < \a_+ + \hat{\e} , w_1 = 1 ) + \bbP ( \t_1 + T_1 ' + T_1 '' < \a_+ + \hat{\e} , w_1 = -1 ) \\
	& \leq \bbP ( \t_1 < \a_+ + \hat{\e} , w_1 = -1 ) \bbP ( T_1 < \a_+ + \hat{\e} )^2 \\ &
	\leq 
	\bbP ( w_1 = -1 ) \bbP ( T_1 < \a_+ + \hat{\e} )^2 \, . 
	\end{split} \]
Dividing both sides by the positive quantity  $\bbP (T_1 < \a_+ + \hat{\e} )$ and recalling that  by (A2) $\bbP ( w_1 = -1 ) <1$,  we get the contradiction and this concludes the proof.
\end{proof}

We now focus on the behavior $\log \varphi_\pm$. Recall the definition of $\l_c, \a_\pm$ given in Proposition \ref{sirius_black}. 
%Below we list some facts about the functions $ \log \varphi_\pm (\lambda)$:
\begin{Lemma} \label{silente} The following holds:
\begin{itemize}
%\item[(i)] there exists a unique critical value $0 \leq \lambda_c < \infty$ such that $\varphi_\pm (\lambda)<\infty$ for all $\lambda \leq \lambda_c$, $\varphi_\pm (\lambda) = +\infty$ for all  $\lambda > \lambda_c $;
\item[(i)] $\log \varphi_\pm$ is strictly increasing and continuous on $(-\infty , \lambda_c ]$, convex and smooth on $(-\infty , \lambda_{c} )$ and moreover  $\lim_{\lambda \to  -\infty} \log \varphi_\pm (\lambda) = -\infty $; 
\item[(ii)] the second derivative  $(\log \varphi_\pm )''$ is strictly positive on $(-\infty , \l_c)$ (in particular $(\log \varphi_+ (\lambda))'$ is strictly increasing on $(-\infty , \l_c)$ )
 and	\begin{align}
	  & \lim_{\lambda \to  -\infty} (\log \varphi_\pm (\lambda))' = \a _\pm \,, \label{derivo1}\\
	  &  \lim_{ \lambda \nearrow \lambda_c }  (\log \varphi_\pm (\lambda))' = +\infty  \,.\label{derivo2}
	\end{align}
%where we recall that $\a_+$, $\a_-$  have been defined as the infimum of the support of $T_1 $, $T_{-1}$ respectively.
\end{itemize}
\end{Lemma}
%\begin{Claim} 
%It holds:
%	\[ \lim_{ \lambda \nearrow \lambda_c }  (\log \varphi_\pm (\lambda))' = +\infty \, . \]
%\end{Claim}
%\club Note that, due to equation \eqref{derivo2} in Point (iv),  $J_\pm$ does not present any linear region.
\begin{proof} 
The proof of Point (i) is rather standard (see Lemma 2.2.5 in \cite{DZ}, the fact that $\log \varphi_\pm$ is strictly increasing on $(-\infty , \lambda_c ]$ and convex  follows also  from Point (ii)).

\medskip

We  prove  Point (ii) restricting to $\varphi_+$ without loss of generality.  Note that, for $\l<\l_c$,
$$  (\log \varphi_+(\lambda))''= \frac{\bbE\left( T_1^2 e^{\l T_1} \mathds{1}(T_1<\infty)\right) }{
\bbE\left( e^{\l T_1} \mathds{1}(T_1<\infty)\right)
}-   \frac{\bbE\left( T_1 e^{\l T_1} \mathds{1}(T_1<\infty)\right)^2 }{
\bbE\left( e^{\l T_1} \mathds{1}(T_1<\infty)\right)^2
}
=
 {\rm Var}_\bbQ( T_1)\,,$$ where  $\bbQ$ it the probability defined as
  $\bbQ(A)=  \bbE\left(\mathds{1}(A) e^{\l T_1} \mathds{1}(T_1<\infty)\right)/  \bbE\left(e^{\l T_1} \mathds{1}(T_1<\infty)\right)$.
Since $T_1$ is non constant $\bbQ$--a.s. by Assumption (A1), we conclude that $ (\log \varphi_+ (\lambda))''>0$ for $\l < \l_c$, hence $(\log \varphi_+ (\lambda))'$ on $(-\infty , \l_c)$ is strictly increasing.

 \medskip 
 
We first  derive \eqref{derivo1} in the case $\a_+=0$.  It is convenient to prove the thesis for a generic nonnegative random variable $T_1$, non necessarily defined as  in Theorem \ref{teo2}.
 Suppose first that   $\bbP (T_1 = 0) >0$.  Since  $\varphi_+(\l) \geq \bbP(T_1=0)$, while
$\lim _{\l \to -\infty} \varphi'(\l)= \lim _{\l\to -\infty} \bbE ( T_1 e^{\lambda T_1} \mathds{1}(T_1 <\infty ))
=0$ by the monotone convergence theorem, we get \eqref{derivo1}.

We now consider the case $\bbP(T_1=0)=0$, thus implying $\bbP ( T_1 \in (0, \varepsilon )) >0$ for all $\varepsilon > 0$. We fix any $c>0$ and take $\lambda < -1/c$. By  this choice it holds  $\sup_{x \geq c} x e^{\lambda x} = c e^{\lambda c} $. Moreover we fix $c_1 , c_2 : 0 < c_1 < c_2 < c$ such that 
$\bbP ( c_1 \leq T_1 \leq c_2 ) >0$.  
Define:
	\begin{align*}
	& e_1(\lambda ) := \bbE ( T_1 e^{\lambda T_1} \mathds{1} (c \leq T_1 < \infty )) \leq c e^{\lambda c} \,, \\
	& e_2(\lambda ) := \bbE ( T_1 e^{\lambda T_1} \mathds{1} ( T_1 < c ))  \geq  c_1 e^{\lambda c_2 } \bbP ( c_1 \leq T_1 \leq c_2 )   \, , \\
	& e_3(\lambda ) := \bbE ( e^{\lambda T_1} \mathds{1} (c \leq T_1 < \infty )) \leq  e^{\lambda c}\,,	  \\
	& e_4(\lambda ) := \bbE ( e^{\lambda T_1} \mathds{1} ( T_1 < c )) \geq  e^{\lambda c / 2 } \bbP (  T_1 < c/2  ) > 0  \, .  	\end{align*}
By the previous bounds we have 
%Combining \eqref{elefante1} and \eqref{elefante3} we get 
	$ \lim_{\lambda \to -\infty} e_1(\lambda )/e_2(\lambda )=0$
%	Moreover, 	
%the above bounds in \eqref{elefante4} and \eqref{elefante5}  imply that 
and $ \lim_{\lambda \to -\infty}  e_3(\lambda )/e_4(\lambda ) =0$. In conclusion
	\begin{equation*}
	\begin{split}
	0  \leq \lim_{\lambda \to -\infty}  ( \log \varphi_+(\l) )'  & =  \lim_{\lambda \to -\infty} \frac{\varphi_+'(\lambda )}{\varphi_+ (\lambda )}  
	= \lim_{\lambda \to -\infty}  
	\frac{e_2(\lambda ) \big( 1 + e_1(\lambda )/e_2(\lambda ) \big) }{e_4 (\lambda ) \big( 1 + e_3(\lambda )/e_4(\lambda ) \big) } \\
&	= \lim_{\lambda \to -\infty}  \frac{e_2(\lambda )}{e_4(\lambda )} 
	=  \lim_{\lambda \to -\infty} \frac{\bbE ( T_1 e^{\lambda T_1} \mathds{1} ( T_1 < c )) }{\bbE ( e^{\lambda T_1} \mathds{1} ( T_1 < c )) }
	\leq  c \, . 
	 \end{split}
	 \end{equation*}
Since $c>0$ is arbitrary we get \eqref{derivo1}.

To complete the proof of \eqref{derivo1} it remains to discuss the case $\a_+>0$. To this aim note that $0$ is 
the minimum in 
the support of the law of $\hat T_1
:= T_1- \a_+$.  Hence, by what just proven, it holds 
$ \lim_{\lambda \to -\infty}  ( \log \hat \varphi_+(\l) )'  =0
$,
where  $\hat \varphi_+(\l):= \bbE ( e^{\lambda \hat T_1} \mathds{1}(\hat T_1 <\infty ))$. Since
$\varphi_+(\l)= e^{\l \a_+} \hat \varphi_+ (\l)$, we get \eqref{derivo1}.

\medskip

To conclude the proof of Point (ii) we need to justify \eqref{derivo2}.
%The above limit exists, since the first derivative of $\log \varphi_\pm$ is an increasing function.
Since by Point (i) $\log \varphi_+$ is smooth and convex on $(-\infty, \l_c)$, the derivative $ (\log \varphi_+ (\lambda))'= \varphi_+ '(\lambda )/\varphi_+(\lambda ) $ is increasing and therefore the limit in \eqref{derivo2} exists.
%For any $\lambda < \lambda_c$ we have:
%	\[ (\log \varphi_+ (\lambda ))' = \frac{\varphi_+ '(\lambda )}{\varphi_+(\lambda )}
%	= \frac{ \bbE \big( T_1 e^{\lambda T_{ 1}} \mathds{1}(T_{ 1} <\infty ) \big) }{\bbE \big( e^{\lambda T_{ 1}} \mathds{1}(T_{ 1} <\infty ) \big)}
%	\, . \]
Moreover, since $\ds \lim_{\lambda \nearrow \lambda_c } \varphi_+ (\lambda) = \varphi_+ (\lambda_c) < \infty$, we only need  to show that 
$\ds \lim_{\lambda \nearrow \lambda_c } \varphi_+ '(\lambda) = +\infty $.
To this aim recall \eqref{base}.
%, recall that for all $\lambda < \lambda_c$ $\varphi_+(\lambda )$ solves
%	\[ f_-(\lambda ) \varphi_+ (\lambda)^2 - \varphi_+(\lambda) + f_+(\lambda) = 0 \]
%Fsecso that, 
Differentiating such identity for  $\l < \l_c$ (note that everything is finite and smooth) we get
	\begin{equation}\label{derivative}
	\big( 1 - 2 f_-(\lambda )\varphi_+(\lambda ) \big) \varphi_+ ' (\lambda )
	=
	f_+ ' (\lambda) + f_- '(\lambda) \varphi_+ (\lambda )^2 \, .
	\end{equation}
By the monotone convergence theorem we get that $f_-(\lambda ) , \varphi_+(\l)$ and the derivative 
$ f_\pm '(\lambda ) = 
	 \bbE \big( \t_1 e^{\lambda \t_1} \mathds{1}(w_1 =\pm1 ) \big)$ converge to $f_-(\lambda_c ) , \varphi_+(\l_c)$ and 
$ f_\pm '(\lambda_c )$ respectively as $\l \nearrow \l_c$. Observing that  
\begin{equation}\label{pisolo}
\varphi_+ (\lambda_c ) = \frac{1}{2 f_-(\lambda_c)}
\end{equation} due to  \eqref{cena} and the identity $f_+(\l_c) f_-(\l_c)=1/4$, 
we get that 
	%$\ds \lim_{\lambda \nearrow \lambda_c } f_-(\lambda ) = f_-(\lambda_c ) <\infty$ and $\ds \lim_{\lambda \nearrow \lambda_c } \varphi_+ (\lambda ) = \varphi_+ (\lambda_c ) = \left(2 f_-(\lambda_c)\right)^{-1}$ (the last identity follows from \eqref{base} since $f_+(\l_c) f_-(\l_c)=1/4$). In particular, 
	$1 - 2 f_-(\lambda )\varphi_+(\lambda ) $ converges to zero as $\l \nearrow \l_c$.
  On the other hand,  as  $\lambda \nearrow \lambda_c$ the 
  r.h.s. of \eqref{derivative} converges to its value at $\l_c$, which is nonzero.
% in  \eqref{derivative}, we have
%	\[ \underbrace{\big( 1 - 2 f_-(\lambda )\varphi_+(\lambda ) \big)}_{\searrow 0}  \varphi_+ ' (\lambda ) 
%	= \underbrace{ f_+ ' (\lambda) + f_- '(\lambda) \varphi_+ (\lambda )^2 }_{\rightarrow C >0} \]
%which implies
%	\[ \lim_{\lambda \nearrow \lambda_c } \varphi_+ '(\lambda ) = +\infty \, . \]
The limit \eqref{derivo2} is therefore the only possibility as $\l \nearrow \l_c$ in \eqref{derivative}.
\end{proof}

Recall the  definition of the asymptotic velocity $v$ given  in \eqref{trenino}.

\begin{Remark}\label{allergia}  
Taking the expectation in \eqref{rinnovo99} and in the analogous expression for $T_{-1}$,  one concludes that   $\bbE (T_{\pm 1})<+\infty$ implies that  $ \bbE (T_{\pm 1})\, \bbE(w_1)= \pm\bbE(\t_1)$. Since $\bbE(\t_1)\not =0$, we conclude that   if $ \bbE (T_{\pm 1}) < \infty$ then $\bbE(w_1)\not =0$ and  $\bbE(T_{\pm1})= \pm \bbE(\t_1)/ \bbE(w_1)= \pm 1/v$.
\end{Remark}

From Lemma \ref{silente} we deduce the qualitative behavior of the rate  function  $J_\pm (\th ): = \sup_{\lambda \in \bbR } \{ \lambda \th - \log \varphi_\pm (\lambda ) \} $:
\begin{Proposition} \label{J_study} The following holds:
\begin{itemize}
\item[(i)] $J_\pm $ is lower semicontinuous, convex and takes values in $[0 , +\infty ]$. 
\item[(ii)] $J_\pm$ is finite   on $(\a_\pm , +\infty )$ and infinite on $(-\infty , \a_\pm )$. 
\item[(iii)]  $J_\pm$ is  smooth on $(\a_\pm , +\infty )$  and 
the derivative $J'_\pm$ satisfies $\ds \lim_{\th \to +\infty } J_\pm ' (\th ) = \lambda_c$. In particular, $\ds\lim _{\th \to +\infty} J_\pm (\th)=+\infty$ if $\l_c >0$.
\item[(iv)] If $\lambda_c = 0$ then $J_\pm$ is   strictly decreasing on $ (\a_\pm , +\infty )$. 
If $\lambda_c >0$ then there exist $\th_c^\pm  \in (\a_\pm , +\infty )$ such that $J_\pm$ is  strictly decreasing on $(\a_\pm , \th_c^\pm )$, strictly increasing on $(\th_c^\pm , +\infty )$. Moreover:
	\begin{itemize}
	\item[$\bullet$] $v=0$ if and only if $\lambda_c = 0$\,,
	\item[$\bullet$] if $v>0$ then $\th_c^+ = 1/v$ and $J_+ ( \th_c^+ ) =0$\,,
	\item[$\bullet$] if $v<0$ then $\th_c^- = -1/v$ and $J_- ( \th_c^- ) =0$\,.
	\end{itemize}

\item[(v)] The value $J_\pm(\a_\pm)$ admits the following characterization
	\begin{equation} \label{alpha}
	%\begin{cases}
	J_\pm  (\a_\pm  ) = \ds \lim_{ \th \searrow \a_\pm  } J_\pm  (\th ) = 
	\begin{cases}
	+ \infty \, \mbox{ if } \bbP (T_{\pm 1 }= \a_\pm  ) =0 \,,\\
	< \infty \, \mbox{ otherwise}\,.
	\end{cases}
	%\\
	%J_- (\a_- ) = \ds \lim_{ \th \searrow \a_- } J_- (\th ) = 
	%\begin{cases}
	%+ \infty \, \mbox{ if } \bbP (T_{-1} = \a_- ) =0\,,\\
	%< \infty \, \mbox{ otherwise}\,.
	%\end{cases}
	%\end{cases} 
	\end{equation}
\end{itemize}
\end{Proposition}
\begin{Remark} \label{Jgood} Due to the above result,  
 $J_\pm$ is a  good rate function (i.e.  $\{J_\pm \leq u \}$ is compact for all $u \in \bbR$) if  and only if  $\lambda_c >0$.
\end{Remark}

\begin{proof} Without loss of generality 
we prove the above statements only  for $J_+$.

The proof of Point (i) is standard (cf. \cite{DZ}[Ch.2]) and we omit it.
 We now prove Point (ii). The fact that $J_+(\th ) = \infty$ for $\th \leq 0$ follows from Lemma \ref{silente}, Item (i). We now show that if $\a_+ >0$ then $J_+(\th )=\infty$ also for $\th \in (0 , \a_+ )$. For such $\th$,
%Take $0<\th < \a_+$, then 
by \eqref{derivo1} in Lemma \ref{silente} it holds
	\[ \lim_{\lambda \to -\infty} \frac{\log \varphi_+(\lambda )}{\lambda \th} = 
	\lim_{\lambda \to -\infty} \frac{(\log \varphi_+)'(\lambda )}{\th} = \frac{\a_+}{\th} >1 \]
and therefore
	\[J_+(\th ) \geq \lim_{\lambda \to -\infty} \lambda \th \bigg( 1 - \frac{\log \varphi_+(\lambda )}{\lambda \th} \bigg) = +\infty \, . \]
Take now  $\th > \a_+$. Since by Lemma \ref{silente}  $(\log \varphi_+  )'(\lambda )$ is a strictly increasing function which takes values in $(\a_+ , +\infty )$, there exists a unique $\tilde{\lambda}_+ (\th )$ such that 
	\begin{equation} \label{lambda_tilde}
	\th = (\log \varphi_+ )' (\tilde{\lambda}_+(\th ) ) \, . 
	\end{equation}
Then $J_+(\th ) = \th \tilde{\lambda}_+(\th ) - \log \varphi_+(\tilde{\lambda}_+(\th ))$ which is finite. This concludes the proof of (ii).

 \medskip

%%%%%%%%%%%%%%%%%%%%%%%%%%%%%%
We now move to Point (iii). Observe that, by 
  \eqref{lambda_tilde}  and Lemma \ref{silente},  $\tilde{\lambda}_+$ is the inverse function of $(\log \varphi_+ )'$.  By Lemma \ref{silente} $(\log \varphi_\pm)'$ is smooth on $(-\infty, \l_c)$ and
	$( \log \varphi_\pm ) ''>0$ on  $(-\infty, \l_c)$. Hence, by the implicit function theorem, the function $\tilde{\lambda}_+$ 
	 is   smooth on $(\a_+ , +\infty )$ and tending to $\l_c$ as $\th \to +\infty$ (see \eqref{derivo2}). 
Hence $J_+$ is smooth on $(\a_+ , +\infty )$ and 
	\begin{equation}\label{mattina}
	 J_+ '(\th ) = \th \tilde{\lambda}_+ '(\th ) + \tilde{\lambda}_+(\th ) - 
	(\log \varphi_+ )' (\tilde{\lambda}_+(\th ))  \tilde{\lambda}'_+(\th ) 
	= \tilde{\lambda}_+(\th )\, ,
	\end{equation}
thus implying that $ \lim_{\th \to +\infty } J_+ ' (\th ) = \lambda_c$. This  concludes the proof of Item (iii).

\medskip

%%%%%%%%%%%%%%%%%%%%%%%%%%%%%%%%
We now prove Point   (iv). By Lemma \ref{silente}  $\tilde{\l}_+$ is strictly increasing on $(\a_+,+\infty)$,  $\lim _{\th \searrow \a_+ }\tilde{\l}_+ (\th)= -\infty$ and   $\lim _{\th \to \infty}\tilde{\l}_+ (\th)= \l_c$.  If $\l_c=0$, then $\tilde{\l}_+$ must be negative and from \eqref{mattina} we conclude that 
  $J_+$ is strictly decreasing on $(\a_+ , +\infty )$.  If $\l_c >0$, then there exists a unique $\th_c^+ $
  such that $\tilde{\lambda}_+(\th_c^+ ) =0$, $\tilde{\lambda}_+$ is negative on the left of $\th_c^+$ and is positive  on the right  of $\th_c^+$.
% assuming first $\l_c=0$. As already observed, $\tilde{\l}_+$ is strictly increasing on $(\a_+,+\infty)$ and $\lim _{\th \to \infty}\tilde{\l}_+ (\th)= \l_c=0$.  In particular, $\tilde {\l}_+$ must be negative, and by \eqref{mattina} we conclude that  $J_+$ is strictly decreasing on $(\a_+ , +\infty )$.
Hence,   by \eqref{mattina} we see that $J_+$ has a unique minimum at $\th = \th_c^+$.
In this case, from \eqref{lambda_tilde} we have
	\[ \th_c^+ = (\log \varphi_+ )' (0) = \frac{\varphi_+ '(0)}{\varphi_+(0)} 
	= \frac{\bbE (T_1 \mathds{1}(T_1 <\infty ))}{\bbP (T_1 <\infty )} \, . \]
If $v>0$, then by the LLN in Proposition \ref{teo1}   we get that $T_1 <\infty$ a.s. and $\th_c^+ = \bbE (T_1 )=1/v$  (cf. Remark \ref{allergia}). Hence, recalling that $\tilde{\lambda}_+(\th_c^+ )=0$, 
	\[ \inf_{\th \in \bbR } J_+(\th ) = J_+(\th_c^+ ) 
	=\tilde{\lambda}_+(\th_c^+ )  \th_c^+ - \log \varphi_+ (\tilde{\lambda}_+(\th_c^+ ))
	= - \log \bbP (T_1 < \infty ) = 0 \, . \]
	The case $v<0$ can be treated similarly. 
We conclude the proof by showing that $v=0 \Leftrightarrow \lambda_c =0$.  Trivially,
%  Since $v= \bbE(w_1)/ \bbE(\t_1)$, 
	$v=0 \Leftrightarrow  \bbP( w_1=\pm 1)=\frac{1}{2} \Leftrightarrow \bbP (w_1 =1 ) \bbP (w_1 = -1) = \frac{1}{4} $. The last identity can be rewritten as 
	$ f_+(0) f_-(0) = \frac{1}{4}$, where the function $f_+,f_-$ are   defined as in \eqref{cioccolato}. Due to the characterization of $\l_c$ given after \eqref{pisolo},  we conclude that the last identity is equivalent to $\l_c=0$.
	
	\medskip 
 
%%%%%%%%%%%%%%%%%%%%%%%%
To derive Point (v) we   note that by Point (iv) the limit in \eqref{alpha} exists. We first  assume  $\bbP ( T_1 = \a_+ ) =0$. Taking $\d>0$ and $\l <0$ we can bound
$$ \varphi_+( \l) \leq e^{\l  \a_+}  \bbP (T_1 \leq \a_+ +\d)+ e^{\l (\a_++\d) } \bbP( T_1 > \a_++\d)\,,$$
thus implying 
$$J_+(\a_+) \geq  \l \a_+ - \log \varphi_+(\l) \geq -\log \bigl[  \bbP (T_1 \leq \a_+ +\d)+ e^{\l \d } \bbP( T_1 > \a_++\d)
\bigr] \,.$$ To get that $J_+(\a_+)=+\infty$ it is enough to take first the limit  $\l \to -\infty$ and afterwards the limit  $\d \to 0$.  Since $J_+$ is also l.s.c. one has $\lim_{ \th \searrow \a_+ } J_+(\th ) \geq J_+ (\a_+ )$ and therefore one gets \eqref{alpha}.

 Assume, on the other hand, that $\bbP ( T_1 = \a_+ ) >0$. %We want to show that in this case $J_+(\a_+ )$ is finite and equal to the right--limit of $J_+$ in $  \a_+$. 
 The fact that $J_+ (\a_+ ) <\infty$ follows  by the LDP for $T_n$ (cf. Subsection \ref{mec_yek1}) and the characterization of $\a_+$ given in Proposition \ref{sirius_black}--(ii). Indeed we can bound
	\[ \begin{split}
	 -J_+ (\a_+ ) & \geq 
	\limsup_{n \to \infty} \frac{1}{n} \log \bbP \left( \frac{T_n}{n} = \a_+ \right) 
	= \limsup_{n \to \infty} \frac{1}{n} \log \bbP ( \t_1 = \a_+ , w_1 = 1 )^n \\ &
	= \log \bbP ( \t_1 = \a_+ , w_1 = 1 )= \log \bbP(T_1= \a_+) > -\infty\,.
	\end{split} \]
%from which $J_+ (\a_+ ) \leq - \log \bbP ( \t _1= \a_+ , w_1 = 1 ) < \infty$.
 To see that $J_+$ is right--continuous at $\a_+$ observe that by lower semicontinuity
$ J_+(\a_+ ) \leq \lim_{ \th \searrow \a_+ } J_+(\th )$.
We claim that   $ J_+(\a_+ ) \geq \lim_{ \th \searrow \a_+ } J_+(\th )$.
  Indeed, fixed $\a_0> \a_+$, by convexity it holds
\[ J_+(\a_+) \geq\frac{1}{1-\l } J_+\bigl( \, (1-\l) \a_+ +\l \a_0\, \bigr) - \frac{\l}{1-\l} J_+( \a_0)\,.
\]
The claim then follows from the monotonicity of $J_+ $ on the right of $\a_+$. 
Combining the last observations we get 
	$\lim_{ \th \searrow \a_+ } J_+(\th ) = J_+ (\a_+ ) < \infty $
and this concludes the proof of Point (v).
\end{proof}

We now move to the study of the  function $I(\th )$ defined on $\bbR \setminus\{0\}$ as 
\begin{equation} \label{rate_fun}
	I(\th ) = 
	\begin{cases}
	\ds I_+(\th ): = \sup_{\lambda \in \bbR} \{ \lambda - \th \log \varphi_+(\lambda ) \} , \; \th >0 \,,\\
%	\ds \lim_{\th \nearrow 0} I_-(\th ) = \lim_{\th \searrow 0} I_+(\th ) = \lambda_c \, , \quad \th=0 \\
	\ds I_-(\th ): = \sup_{\lambda \in \bbR} \{ \lambda + \th \log \varphi_-(\lambda ) \} , \; \th <0 \,.
	\end{cases}
\end{equation}

%Using Lemma \ref{silente} we can study the behavior of  $I$ around $\th = 0$:
\begin{Lemma}\label{biancaneve}
It holds
\begin{align}
&  \lim_{\th \nearrow 0} I_-(\th ) = \lim_{\th \searrow 0} I_+(\th ) = \lambda_c \,,\label{nano1}\\
&   \lim_{\th \nearrow 0} I_- ' (\th ) =  \lim_{\th \searrow 0} I_+ ' (\th ) \,. \label{nano2}
\end{align}
In particular, the definition of $I(\th)$ in Theorem \ref{teo2} is well posed and   $I(0)=\l_c$. Moreover,  $I$ is finite and  $C^1$ on $(-1/\a_-,1/\a_+)$, and it is smooth on $(-1/\a_-,1/\a_+)\setminus\{0\}$.
\end{Lemma}
\begin{proof}
For any $\th >0$ we have 
$ I(\th ) = \sup_{\lambda \leq \lambda_c } \{ \lambda - \th \log\varphi_+ (\lambda) \} $ since $\varphi_+(\l)=+\infty$ if $\l > \l_c$. %Note that the supremum has been restricted to $\lambda \in (-\infty , \lambda_c ]$, being the argument $-\infty$ for any $\lambda > \lambda_c$, and that
Moreover,  always by Lemma \ref{silente},  for $ 0< \th <1/\a_+ $ the above supremum is attained at the unique  value  
%the  it is attained when 
%	\[ (\log\varphi_+ (\lambda))' = \frac{1}{\th }\, . \]
%Recall that the LHS takes values in $(\a , +\infty )$ and that it is strictly increasing, with $\lim_{ \lambda \nearrow \lambda_c }  (\log \varphi_+ (\lambda))' = +\infty$. Therefore for each $ 0< \th <1/\a $ there exists a unique value
 $\lambda_+ (\th) < \lambda_c $ such that 
	\begin{equation}\label{lori_pierpi}(\log\varphi_+ )' \bigl( \l_+(\th) \bigr)   = 1/\th
	%(\log\varphi_+ (\lambda))'\Big |_{\l = \l_+(\th)}  = \frac{1}{\th } 
	\, ,\end{equation}
	thus implying that $\l_+(\th)$ is a strictly decreasing function and  $\lim_{\th \searrow 0} \lambda_+ (\th ) = \lambda_c$ (due to Lemma \ref{silente}). In particular,  $I(\th)= \l_+(\th) - \th \log \varphi_+\bigl(\l_+(\th)\bigr )$ is finite on $(0, 1/\a_+)$ and moreover
	$$ \lim  _{\th \searrow 0}I(\th)=  \lim  _{\th \searrow 0} \{ \lambda_+(\th)  - \th \log\varphi_+ (\lambda_+(\th)) \} = \l_c$$
	since $ \lim_{ \l \nearrow \l_c} \log \varphi_+(\l)= \log \varphi_+(\l_c)$ which is finite due to \eqref{pisolo}.  This concludes the proof of \eqref{nano1} for $I_+$.  By similar arguments one gets that, given  $\th \in (-1/\a_-, 0)$   
	there is a unique value $\l_-(\th)$ solving  the equation
	\begin{equation}\label{lori_pierpi77}(\log\varphi_- )' \bigl( \l_-(\th) \bigr)   =- 1/\th\,.
	%(\log\varphi_+ (\lambda))'\Big |_{\l = \l_+(\th)}  = \frac{1}{\th } 
	\end{equation}
	The function $\l_-$ is strictly increasing on $(-1/\a_-,0)$ where it holds $I(\th)= \l_-(\th)+\th \log\varphi_- (\l_-(\th) ) $. As above one gets that $ \lim_{\th \nearrow 0} I_-(\th ) = \l_c$, hence \eqref{nano1}. Note that \eqref{nano1} implies that $I$ is well defined  in Theorem \ref{teo2} and that $I(0)=\l_c$. By the previous results we conclude also that $I$ is finite on $ (-1/\a_-, 1/\a_+)$.
	
	Let us now prove \eqref{nano2} and that $I$ is $C^1$ on $ (-1/\a_-, 1/\a_+)\setminus\{0\}$.
  By the implicit function theorem and Lemma \ref{silente}, the function $(0,1/\a_+) \ni \th \to \l_+(\th) \in (-\infty,\l_c) $ is smooth. In particular, using \eqref{lori_pierpi},  $I_+$ is smooth  on  $(0,1/\a_+)$  where it holds
	%As already observed the above function is decreasing and  $\lim_{\th \searrow 0} \lambda_+ (\th ) = \lambda_c$.
 \begin{equation}\label{matteo}
\begin{split}
 I_+ '(\th )  
	& =\frac{\mathrm{d}}{\mathrm{d} \th } \big( \lambda_+(\th ) - \th \log\varphi_+ (\lambda_+(\th )) \big) \\
	& = \lambda_+ ' (\th ) - \log\varphi_+ (\lambda_+(\th )) -\th \cdot
	( \log\varphi_+ )'( \l_+(\th) ) \cdot \l_+' (\th) 
\\	
%	\underbrace{ ( \log\varphi_+ (\lambda_+(\th )))'}_{1/\th } \bigg) \\
	& =  - \log\varphi_+ (\lambda_+(\th ))   \, . 
\end{split} 
\end{equation} Hence,   $\lim_{\th \searrow 0} I_+ '(\th ) = - \log \varphi_+ (\lambda_c )$.
% Note that in the third identity we have used \eqref{lori_pierpi}.
By similar arguments and definitions  we get that  $I_-$ is smooth on  $(-1/\a_-,0)$ where it holds
	\[ \lim_{\th \nearrow 0} I_- '(\th ) =
	 \lim_{\th \nearrow 0} \frac{\mathrm{d}}{\mathrm{d} \th } \big( \lambda_-(\th ) + \th \log\varphi_- (\lambda_-(\th )) \big)
	 	= \log \varphi_- (\lambda_c )  \, . \]
To conclude the proof of \eqref{nano2} it remains to show that 
 $\log \varphi_- (\lambda_c ) = -\log \varphi_+ (\lambda_c ) $. To this aim we observe that 
	\[ \begin{split}
	\log \varphi_- (\lambda_c ) + \log \varphi_+ (\lambda_c ) = \log [ \varphi_+ (\lambda_c ) \varphi_- (\lambda_c) ] 
	= \log \left( \frac{1}{4 f_+(\lambda_c ) f_-(\lambda_c )} \right)  =0 
	\end{split} \]
	due to \eqref{pisolo}, its analogous version for $\varphi_-(\l_c)$ and 
since, by definition, $\lambda_c$ is the unique solution of $4 f_-(\lambda) f_+(\lambda) =1$
This concludes the proof of \eqref{nano2} and that $I$ is smooth on $ (-1/\a_-, 1/\a_+)\setminus\{0\}$. Due to \eqref{nano2} one easily gets that $I$ is differentiable at $0$ and $I'(0)$ equals the limits in \eqref{nano2}.  This implies that  $I$ is $C^1$ on $ (-1/\a_-, 1/\a_+)$.
\end{proof} 

Combining  Lemmas \ref{silente}, \ref{biancaneve} and Proposition \ref{J_study} we are finally able to prove Theorem \ref{acqua} and therefore also Theorem \ref{I_study} due to Lemma \ref{torlonia}.

\begin{proof}[Proof of Theorem   \ref{acqua}]  Below the labeling of items is as in Theorem \ref{I_study}.

The fact that $I$ is finite and $C^1$  on $\big( -\frac{1}{\a_-} , \frac{1}{\a_+} \big)$ and infinite outside  $[  -\frac{1}{\a_-} , \frac{1}{\a_+} ]$ follows from \eqref{rate_f100} and Proposition \ref{J_study}. This proves Item (i). %
%\footnote{Here we mean $(-\infty ,  -\frac{1}{\a_-} \big) = \emptyset$ if $\a_- = 0$ and $\big( \frac{1}{\a_+} , +\infty )=\emptyset$ if $\a_+ = 0$.}.

To prove Item  (ii) note that if $\bbP ( T_1 = \a_+ ) >0$ then $\a_+ >0$. Hence,  by \eqref{alpha} we get 	\[ \lim_{\th \nearrow \frac{1}{\a_+ }} I(\th ) = \lim_{\th \nearrow \frac{1}{\a_+} } \th J_+ \big(\frac{1}{\th} \big) 
	= \frac{J_+(\a_+ )}{\a_+} < \infty \]
and the last term equals $I(1/\a_+ ) $ by definition of $I$.
If, on the other hand, $\bbP ( T_1 = \a_+ ) = 0$, then  by \eqref{alpha} we get	\[ \lim_{\th \nearrow \frac{1}{\a_+} } I(\th ) = \lim_{u \searrow \a_+ } \frac{ J_+ (u)}{u} = +\infty\, . \]
The correspondent statements for $\th \searrow -1/\a_-$ are obtained in the same way, and this concludes the proof of Item (ii). 

To see (iii) recall that $I'(\th ) = -\log \varphi_+ (\lambda_+ (\th ))$ for $\l \in(0,1/\a_+)$ (see \eqref{matteo}). Observe now that 
 $ \lim_{\th \nearrow \frac{1}{\a_+} } \lambda_+(\th ) =% \lim_{u \searrow \a_+ } \tilde{\lambda}_+(u ) =
  -\infty$ (due to Lemma \ref{silente}--(ii) and \eqref{lori_pierpi}). This implies that
	\[\lim_{\th \nearrow \frac{1}{\a_+} } (-\log \varphi_+ (\lambda_+ (\th ))) = -\log \varphi_+ (-\infty) = +\infty\, . \]
Similarly one sees that 
	\[ \lim_{\th \searrow -\frac{1}{\a_-} } I'(\th ) =
	\lim_{\th \searrow -\frac{1}{\a_-} }  \log \varphi_- (\lambda_- (\th )) = \log \varphi_- (-\infty) = -\infty\, . \]

We now consider Item (iv). We observe that  $I_+,I_-$ are  l.s.c. because they can be expressed as pointwise suprema of  continuous functions, and by \eqref{nano1} they attach in $0$ in a continuous fashion.  	
We now  prove that $I$ is convex. Being suprema of families of linear functions, $I_+$ and $I_-$ are convex. Therefore $I$ is convex on $(0, +\infty)$ and $(-\infty , 0)$ separately. To prove the convexity on all $\bbR$ it remains to show that $I$ is also convex in $\th =0$.
Since the left and right branches of $I$ are differentiable, it suffices to show that the left derivative at $\th = 0$ is non greater than the right derivative. In fact,   they are equal due to  \eqref{nano2}. Let us now prove that $I$ is strictly convex on the closure of $(-1/\a_-, 1/\a_+)$. We know that $I'(\th)=
	\log \varphi_-(\l_-(\th) )$ on $(-1/\a_-, 0]$ and  $I'(\th)=
	-\log \varphi_+(\l_+(\th) )$ on $[0, 1/\a_+)$ (see the proof of Lemma \ref{biancaneve}).
	By Lemma \ref{silente}--(ii) $\log \varphi_\pm$ is strictly increasing with positive derivative, while we know that $\l_+$ is a strictly decreasing function on $(0, 1/\a_+)$  and $\l_-$ is a strictly 
	increasing on $(-1/\a_-,0)$. Using also that $I'$ is continuous at $0$ we conclude  that $I'$ is strictly increasing on $(-1/\a_-, 1/\a_+)$, hence $I$ is strictly convex on $(-1/\a_-, 1/\a_+)$.

We conclude with Item (v). We know that  $I'$ is strictly increasing on $(-1/\a_-, 1/\a_+)$. Due to Item (iii)  it simple to conclude that there exists a unique minimum point $\th_* \in (-1/\a_-, 1/\a_+)$ such that $I$ is strictly decreasing on $(-1/\a_-,\th_*)$ and strictly increasing on $(\th_*, 1/\a_+)$. It remains to prove that $\th_*=v$.

If $v>0$ then by Prop.  \ref{J_study}--(iv) $v = 1 / \th_c^+$ and so
	$ I(v) = v J_+ \big(1/v\big) = (1/\th_c^+) J_+ (\th_c^+ ) = 0$.
If $v<0$ then $v = -1 / \th_c^-$ and so
	$ I(v) = -v J_- \big( -1/v \big) =(1/\th_c^-) J_- (\th_c^- ) = 0$. If, finally, $v=0$ then again by Prop. \ref{J_study}--(iv) we have $0 = \lambda_c = I(0) = I(v)$. 
In all cases $I(v) =0$ and since $I$ is non--negative we conclude that $v= \th_*$.
\end{proof}

\section{Proof of Theorem \ref{teo2}--(ii)}\label{mec_yek3}
Below we show how one can deduce the  LDP for the process $Z $ itself
 from the LDP for the hitting times. Due to 
   Theorem 4.1.11 in \cite{DZ}, the LDP for $Z_t/t$ holds  with rate function $I$ and speed $t$ 
   if we show that 
	\begin{equation} \label{LB} \tag{LB}
	\lim_{\varepsilon \to 0} \liminf_{t\to\infty} \frac{1}{t} \log \bbP\left( \frac{Z_t}{t}\in  (\th -\varepsilon , \th +\varepsilon ) \right)\geq - I(\th )\,,
	\end{equation}
	\begin{equation} \label{UB} \tag{UB}
	\lim_{\varepsilon \to 0} \limsup_{t\to\infty} \frac{1}{t} \log   \bbP\left( \frac{Z_t}{t}\in  (\th -\varepsilon , \th +\varepsilon ) \right)\leq - I(\th ) \, .
	\end{equation}

%Note that \eqref{LB} and \eqref{UB} together imply that the limit of the above quantity exists and equals $-I(\th )$. 
%Nevertheless, we keep the two bounds separate in light of the fact that they will be proven using different methods.

\subsection{The lower bound}
Given  $\th \neq 0$ and  $\d ,c \in (0,1)$ we define the events 
\[ \begin{split}
	& A_t =A_t(\d, \th):= \{ (1-\d )t < T_{\lfloor \th t \rfloor} < (1+\d )t \} \\
	& B_t =B_t(\d,c):= \{ \nu(t + \d t) - \nu(t -\d t) \leq c  t  \} \, , 
	\end{split} \]
\begin{Lemma}\label{theta>0}
For any $\th \neq 0$ and  $\d ,c \in (0,1)$  there exists $\tilde{t} = \tilde{t}(c )>0$ such that	\[ A_t \cap B_t \subseteq \{ Z_t/t  \in (\th -2c , \th  + 2c )\} \,  \]
 for all $t>\tilde{t}$.
\end{Lemma}

\begin{proof}
Take any $t>0$ and assume $A_t \cap B_t$ holds. Then, due to Assumption (A1), 
	\begin{equation*}
	\begin{split}
	| Z_t - \lfloor \th t \rfloor |
	&  = | Z_t - Z_{T_{\lfloor \th t \rfloor }} | =
	 \bigl| W_{\nu(t) }- W_{\nu( T_{\lfloor \th t \rfloor }) }\bigr|\\
	 &
	\leq \nu ( t \vee T_{\lfloor \th t \rfloor } ) - \nu ( t \wedge T_{\lfloor \th t \rfloor } ) 
	 \leq 
	\nu(t + \d t) - \nu(t -\d t) \leq c  t  \, .	
	\end{split}
	\end{equation*}
Hence $Z_t \in [ \lfloor \th t \rfloor -c t , \lfloor \th t \rfloor +c t ]$, thus leading to the thesis.
% which implies $\ds \frac{\lfloor \th t \rfloor}{t} - c \leq \frac{Z_t}{t} \leq \frac{\lfloor \th t \rfloor}{t} +c  $. Now it  is trivial  to conclude.
\end{proof}

%\pecetta{Nel seguito bisogna sostituire $P$ con $\bbP$ per uniformare la notazione con la prima parte del file} 

\begin{Lemma}\label{ananas}
For any $\th \not =0$ and $\d \in (0,1)$ it
 holds $$ \liminf_{t \to \infty} \frac{1}{t} \log \bbP (A_t) \geq -I(\th)=
 \begin{cases}
  -\th J_+ \big(\frac{1}{\th} \big) & \text{ if }\, \th >0\,,\\
    \th J_- \big(-\frac{1}{\th} \big) & \text{ if } \,\th <0 \,.
    \end{cases}
    $$
\end{Lemma}
\begin{proof} 
We give the proof for $\th >0$, the one for $\th <0$ being the same. Note that, fixed $\e>0$,  
for $t$ large enough it holds
	\[ %\begin{split}
	\frac{1}{t} \log \bbP (A_t) 
	%& = \frac{1}{t} \log \bbP \bigg( \frac{(1-\d )t}{\lfloor \th t \rfloor} < \frac{T_{\lfloor \th t \rfloor}}{\lfloor \th t \rfloor} < \frac{(1+\d )t}{\lfloor \th t \rfloor} \bigg) \\
	%&
	 \geq \frac{1}{t} \log \bbP  \bigg( \frac{1-\d }{\th } +\varepsilon  < \frac{T_{\lfloor \th t \rfloor}}{\lfloor \th t \rfloor} < \frac{1+\d }{\th} \bigg) \,.
	%\end{split}
	 \]
Thanks to the LDP for the hitting times $T_n$  this implies that
	\[ \begin{split}
	\liminf_{t \to \infty} \frac{1}{t} \log \bbP (A_t) 
	%& \geq \liminf_{t\to \infty} 
	%\bigg\{ \frac{\lfloor \th t \rfloor}{t} \frac{1}{\lfloor \th t \rfloor} \log \bbP \bigg( \frac{1-\d }{\th } +\varepsilon  < \frac{T_{\lfloor \th t \rfloor}}{\lfloor \th t \rfloor} < \frac{1+\d }{\th} \bigg)\bigg\} \\
	& \geq \th 	\liminf_{n\to\infty} \frac{1}{n} \log \bbP \bigg( \frac{1-\d }{\th } +\varepsilon  < \frac{T_n}{n} < \frac{1+\d }{\th} \bigg) \\
	& \geq -\th \inf_{ \big(  \frac{1-\d }{\th } +\varepsilon  , \frac{1+\d }{\th} \big) } J_+ 
	\geq - \th J_+ \big( \frac{1}{\th} \big) 
	\end{split} \]
as long as $\varepsilon$ is chosen small enough so that $\frac{1}{\th} \in \big( \frac{1-\d }{\th } +\varepsilon  , \frac{1+\d }{\th} \big) $.
\end{proof}
%%%%%%%%%%%%%%%%%%%%
Since $\t_i$'s are positive i.i.d. random variables, for every $p_0 \in (0 , 1)$ we can find some  $\eta>0$ such that $p: = \bbP( \t_i \geq \eta) > p_0 $. In particular, the  i.i.d. random variables $r_i$'s with 
 $r_i := \mathds{1}(\t_i\geq \eta)$ are Bernoulli of parameter $p$. They are a useful tool to bound the  probability of $B_t^c$:
\begin{Lemma}\label{pesca} 
For any $\th \neq 0$ and any  $c \in (0,1)$, there exists a constant  $\d_*=\d_*(\theta, c)\in (0,1)$ depending only on $\theta,c$  such that, for all  $\d\in (0,\d_*]$, it holds
$$ \lim_{t \to \infty}  \frac{\bbP (B_t ^c ) }{ \bbP(A_t)} =0\,.
$$
\end{Lemma}
\begin{proof}
We restrict to the case $\th>0$, being the proof for $\th <0$ similar. 
We observe that the event $B_t^c$ implies the event 
$$ \bigcup _{j= 0}^\infty 
\Big\{ \nu(t-\d t) = j \,,\;  
\sum _{\ell= 1} ^{\lceil  c  t\rceil -1} \t_{j+1+ \ell}  \leq 2 \d t
\Big\}\,.
$$Since the event $\{  \nu(t-\d t) = j  \}$ depends only on $\t_1, \t_2, \dots, \t_{j+1}$, by independence we  get
\begin{equation}\label{nutella}
\begin{split}
\bbP (B_t^c) &
\leq \sum  _{j= 0}^\infty \bbP ( \nu(t-\d t) = j) \bbP\Big(   
\sum _{\ell= 1} ^{\lceil  c  t\rceil-1 } \t_{j+1+ \ell}  \leq 2 \d t
\Big)
\\& 
= \bbP\Big(   
\sum _{\ell= 1} ^{\lceil  c  t\rceil-1 } \t_\ell  \leq 2 \d t
\Big)\leq   \bbP\Big(   
\sum _{\ell= 1} ^{\lceil  c  t\rceil-1 } r_\ell   \leq 2 \d t /\eta
\Big ) \,.
\end{split}
\end{equation}
Above we have used that $\t_\ell \geq \eta\, r_\ell$. Now we use Cram\'er Theorem for sums  of i.i.d. $p$--Bernoulli r.v.'s. The associated  rate function is given by (cf. exercise 2.2.23 in \cite{DZ})
$$ \cI _p (x) =\begin{cases}
x \log \frac{x}{p} + (1-x) \log \frac{1-x}{1-p} & \text{ if } x \in [0,1]\,,\\
+\infty & \text{ otherwise}\,,
\end{cases} $$
with the convention that $0 \log 0:=0$.
Trivially,  $\cI_p$ is strictly decreasing on $[0,p]$ and strictly increasing on $[p, 1]$, while $\cI_p(p)=0$.
%We restrict to  $c$ such that  $ 2(\eta c)^{-1} < p$.
Let $t_*:= \lceil  c  t\rceil -1$.
Writing 
$$ 
\frac{1}{t} \log \bbP\Big(   
\sum _{\ell= 1} ^{\lceil  c  t\rceil -1} r_\ell   \leq 2 \d t /\eta 
\Big )
 = \frac{ t_*}{t} \frac{1}{t_* }\log \bbP\Big(  
 \frac{1}{ t_* } \sum _{\ell= 1} ^{t_* } r_\ell   \leq \frac{2 \d t }{\eta t_*} 
\Big ) 
$$
and using  that   $2 \d t (\eta t_*)^{-1} \leq 3 \d 
 ( \eta c)^{-1}$ for $t$ large enough, 
 we get
\begin{equation}\label{marmellata}
\limsup _{t \to \infty} \frac{1}{t} \log \bbP\Big(   
\sum _{\ell= 1} ^{\lceil  c t\rceil -1} r_\ell   \leq 2 \d t /\eta
\Big)
 \leq -c    \inf _{ (-\infty,\frac{ 3\d}{\eta c} ]} \cI_p \,.
 \end{equation}
Now we have to choose  carefully the constants in order to win.  Fix $\th >0$ and $c \in (0,1)$. The function $\cI_p(0)=\log \frac{1}{1-p}$ is increasing in $p$ and $\lim _{p \to 1} \cI_p(0)=+\infty$. In particular, 
there exists 
 $p_0>0$ such that $\cI_p(0) >\th J_+(1/\th)/c $ for all $p \geq p_0$. 
 %Since $\t_i>0$ it holds$\lim _{\eta \searrow 0} \bbP ( \t_i \geq \eta)=1$. 
We fix $\eta$ such that $p:=  \bbP ( \t_i \geq \eta)>p_0$. 

If $p=1$  then $\t_i \geq \eta $ a.s. In particular, equation \eqref{nutella} gives
$ \bbP (B_t^c ) \leq \mathds{1} \big( ct - 1 \leq \frac{2 \d t }{\eta} \big) $, so by setting $\d_* = \eta c / 4$ we have that for any $\d \leq  \d_* $ and $t$ large enough $\bbP (B_t^c ) = 0$. This, combined with Lemma \ref{ananas}, gives the thesis. 

Assume, on the other hand, that   $p<1$. Recall that $\cI_p(0) >\th J_+(1/\th)/c $. Since $\lim_{ \e \searrow 0} \cI_p(\e)= \cI_p(0)$ and $\cI_p$ is decreasing near $0$, we can fix $\e_0>0$ such that $\cI_p(\e) >\th J_+(1/\th)/c$ for all $\e \in [0, \e_0]$.  Note that the (now fixed) constants $\eta, p , \e_0$ depend only on  $\th,c$.  To conclude let $\d_*= (\eta c \e_0/4)\wedge 1 $. Then for each $\d \in (0,\d_*]$ we have $3\d(\eta c)^{-1} \leq \e_0$ and therefore the last term in \eqref{marmellata} is  strictly bounded from above by  $-\th J_+(1/\th)$. 
  Coming back to \eqref{nutella} and  \eqref{marmellata} we conclude that 
\begin{equation}
\limsup_{t \to \infty} \frac{1}{t} \log \bbP(B_t^c) < - \th J_+(1/\th)\,.
\end{equation} 
The above bound together with Lemma \ref{ananas} implies the thesis.
\end{proof}

Combining Lemmas \ref{ananas} and \ref{pesca} we can prove  the following key lower bound:
\begin{Lemma}\label{pizzetta}
For any $\th \not =0$   and $\e \in (0,1/2)$ the following holds
	\begin{equation}\label{liminf}
	 \liminf_{t\to \infty} \frac{1}{t} \log \bbP \Big( \frac{Z_t}{t} \in (\th -\e  , \th  + \e) \Big) 
	\geq - I(\th)   \, . 
	\end{equation}
\end{Lemma}
\begin{proof}
Given $\e>0$,  take   $c:=\e/2$ and $\d:= \d_*( \th, c)$  in the definition of $A_t,B_t$ given in Lemma \ref{theta>0},   where the constant $\d_*$ is as in Lemma \ref{pesca}.
 Due to Lemma \ref{theta>0} 
  for $t$ large enough we have 
	\[ \bbP \Big( \frac{Z_t}{t} \in (\th -\e  , \th  + \e ) \Big) 
	\geq \bbP (A_t \cap B_t ) \geq \bbP (A_t)-\bbP (B_t^c) 
	= \bbP(A_t) \Big( 1 - \frac{ \bbP (B_t^c )}{\bbP (A_t)} \Big) \]
which implies 
	\[ \frac{1}{t} \log \bbP \Big( \frac{Z_t}{t} \in (\th -\e , \th  + \e ) \Big) 
	\geq \frac{1}{t} \log \bbP (A_t) + \frac{1}{t} \log \Big( 1 - \frac{ \bbP (B_t^c )}{\bbP (A_t)} \Big) \, . \]
	Using Lemma \ref{ananas} to control the first term in the r.h.s. and Lemma \ref{pesca} to control the second term in the r.h.s. we get the thesis.	\end{proof}

Being \eqref{liminf} uniform in $\e \in (0 , 1/2 )$, one can let $\e \to 0$ to conclude that the lower bound  \eqref{LB} holds for all $\th \neq 0$.
If $\th = 0$, take any $\e >0$ and let $u = \e /2$. Then by Lemma \ref{pizzetta} one has
	\[   \liminf_{t \to \infty} \frac{1}{t} \log \bbP \Big( \frac{Z_t}{t} \in (-\e  ,   + \e ) \Big) 
	\geq  \liminf_{t \to \infty} \frac{1}{t} \log \bbP \Big( \frac{Z_t}{t} \in ( u -\e /4  , u  + \e /4 ) \Big) 
	\geq -I(u) \, . \]
Letting  then $\e \to  0$ and therefore $u \to 0$ gives (recall Theorems \ref{I_study}, \ref{acqua}) %kuka
	\[ \lim_{\e \to 0 } \liminf_{t \to \infty} \frac{1}{t} \log \bbP \Big( \frac{Z_t}{t} \in (-\e  ,    \e ) \Big) 
	\geq - \lim_{u \to 0} I(u) = -I(0) \, . \]
This concludes the proof of  \eqref{LB} for all $\th \in \bbR$.

%%%%%%%%%%%%%%%%%%%%%%%%%%%%%%%%%%%%%%%%%%%%%%
%%%%%%%%%%%%%%%%%%%%%%%%%%%%%%%%%%%%%%%%%%%%%%
%%%%%%%%%%%%%%%%%%%%%%%%%%%%%%%%%%%%%%%%%%%%%%
%
%
%
%          UPPER BOUND
%
%
%%%%%%%%%%%%%%%%%%%%%%%%%%%%%%%%%%%%%%%%%%%%%%
%%%%%%%%%%%%%%%%%%%%%%%%%%%%%%%%%%%%%%%%%%%%%%
%%%%%%%%%%%%%%%%%%%%%%%%%%%%%%%%%%%%%%%%%%%%%%
\subsection{The upper bound}
We now move to the proof of \eqref{UB}. This is rather easy if the asymptotic velocity $v$ vanishes.

\begin{Lemma}
If $v=0$ then \eqref{UB} holds for all $\th \in \bbR$.
\end{Lemma}
\begin{proof}
If $\th =0$ it is enough to observe that by the LLN $Z_t / t \to 0$ almost surely as $t \to \infty$, and therefore in probability. Since $I(0)=0$ by Theorems \ref{I_study}, \ref{acqua}, % kuka
we get the thesis.

%Hence for any $\e >0$ we have
%	\[ \limsup_{t \to \infty} \frac{1}{t} \log \bbP \bigg( \frac{Z_t}{t} \in (-\e  ,   \e ) \bigg) = 0 = - I(0) \]
%and we are done.

To deal with the case $\th \neq 0$, recall that $v =0 \Leftrightarrow \lambda_c =0 \Leftrightarrow J_\pm$ are strictly decreasing on $(\a_\pm , \infty )$ (see Prop. \ref{J_study}). Since $J_+ \equiv+\infty$ on $(-\infty, \a_+)$ and due to \eqref{alpha} we conclude that $J_+: \bbR \to [0,+\infty]$ is a decreasing extended function.  Fix now any $\th >0$ and $\e >0$ such that $\th - \e >0$.
Then, given any $\tilde \e>0$, for $t$ large it holds 
	\[ \begin{split}
	\bbP \Big( \frac{Z_t}{t} \in ( \th - \e  , \th  + \e ) \Big) &
	 \leq \bbP \Big( \frac{Z_t}{t} >\th  - \e \Big) 
	 \leq \bbP ( Z_t \geq  \lfloor (\th  - \e )t \rfloor ) \\
	& \leq \bbP ( T_{\lfloor (\th  - \e )t \rfloor} \leq t ) 
	= \bbP \Big( \frac{T_{\lfloor (\th  - \e )t \rfloor} }{\lfloor (\th  - \e )t \rfloor} \leq \frac{1}{\th - \e } +\tilde{\e} \Big) \,.
	\end{split} \]
%where $\tilde{\e} >0$ and $t$ is large enough so that $\ds \frac{t}{\lfloor (\th  - \e )t \rfloor} \leq \frac{1}{\th - \e } +\tilde{\e}$.
Hence, using the LDP for the hitting times $T_n$ and the fact that $J_+$ is decreasing,
	\begin{equation}\label{fuochetto} \begin{split}
	\limsup_{t\to\infty} \frac{1}{t} \log\,  & \bbP \Big( \frac{Z_t}{t} \in (\th -\e  , \th  + \e ) \Big) 
	  \\
	% & \leq \limsup_{t\to\infty} \Big( \frac{\lfloor (\th  - \e )t \rfloor}{t} \Big) 
	% \bigg( \frac{1}{\lfloor (\th  - \e )t \rfloor} \log \bbP \Big( \frac{T_{\lfloor (\th  - \e )t \rfloor} }{\lfloor (\th  - \e  )t \rfloor} \leq \frac{1}{\th - \e } +\tilde{\e} \bigg)\Big)  \\
	& \leq (\th -\e)	\limsup_{n\to\infty} \frac{1}{n} \log \bbP \bigg( \frac{T_n}{n} \leq \frac{1}{\th - \e } +\tilde{\e} \bigg) \\
	& \leq - (\th -\e ) \inf_{\big( -\infty , \frac{1}{\th - \e } +\tilde{\e} \big) } J_+ 
	= - (\th -\e ) J_+ \Big( \frac{1}{\th - \e } +\tilde{\e} \Big) \, . 
	\end{split} \end{equation}
  Letting $ \tilde{\e} \to 0$ and using that $J_+$ is l.s.c. (see Prop. \ref{J_study})  we get 
that the first member of \eqref{fuochetto} is bounded from above by $ - (\th -\e ) J_+ \bigl( 1/(\th - \e ) \bigr) $. Taking now  the limit $\e \to 0$ and using again that $J_+$ is l.s.c. we get the thesis 
 for $\th >0$. The proof of \eqref{UB}   for $\th <0$ follows by 
  similar arguments.
\end{proof}

We now prove that \eqref{UB} holds for all $\th \in \bbR$ assuming $v>0$. The case $v<0$ can be addressed in the same way. The proof we present is based on a method introduced in \cite{DGZ}, that we re--adapt to our setting.  The strategy consists in reducing the problem to proving the following:

\begin{Proposition} \label{cuore}
Assume $v>0$  and define   $S_t :=\inf \{ s\geq t : Z_s \leq 0 \}$.  Then it holds
	\begin{equation} \label{cuore1}
	\limsup_{t\to\infty} \frac{1}{t} \log \bbP ( S_t < \infty ) \leq - I(0) \, . 
	\end{equation}
\end{Proposition}
The fact that the above result implies that \eqref{UB} holds for all $\th \in \bbR$ can be seen reasoning as in \cite{DGZ}, page 1017, with minor modifications. 
For completeness we give a sketch of the proof in Appendix \ref{reduction}.

A detailed proof of Proposition \ref{cuore} is, on the other hand, given below.
This choice is due to the presence of a small gap \cite{G}  in the proof presented in \cite{DGZ} (see formula (4.14) on page 1020 there), and to the fact that some additional arguments are necessary since 
 our holding times can be in general  arbitrarily small while in \cite{DGZ} they are bounded from below by $ 1$.

\begin{proof}[Proof of Proposition \ref{cuore}] 
Due to   Prop. \ref{J_study}, since $v>0$,  $\lambda_c >0$ and the critical point $\th_c^\pm$ of $J_\pm$ is finite and positive. 
Take any $ u \in ( 0 , 1 / \th_c^+ ) $ and fix $c>1$ integer such that $c/u> \th_c^-$.
Let, in order to simplify the notation, $b_t := \bbP (S_t < \infty )$ with the convention that $b_t = 1$ if $t < 0$. 
Recall that  $T_{\lfloor tu\rfloor }$ is the hitting time of  $\lfloor tu\rfloor $, and define 
	\[ \tilde{T}_0 := 
	\begin{cases}
	\inf \{ s \geq T_{\lfloor tu\rfloor } : Z_s = 0 \}, \;  \mbox{ if } T_{\lfloor tu\rfloor } < \infty \,,\\
	\qquad \qquad + \infty  \qquad \qquad  \mbox{ otherwise}\,.
	\end{cases} \]
Then we have
	\begin{equation} \label{inizio} 
	\begin{split} 
	b_t  \leq & \, \bbP (T_{\lfloor tu\rfloor } \geq t ) + \bbP ( T_{\lfloor tu\rfloor } <t , \, \tilde{T}_0 - T_{\lfloor tu\rfloor } \geq ct , \,  S_t < \infty ) \\
	& 
	+ \bbP ( T_{\lfloor tu\rfloor }< t , \, \tilde{T}_0 - T_{\lfloor tu\rfloor }< ct ,\,  S_t < \infty ) \, . 
	\end{split}
	\end{equation}
For the first term in the r.h.s. of \eqref{inizio}  the LDP for the hitting times $T_n$, $n \to \infty$,  implies that 
	\begin{equation}\label{siberia1} \limsup_{t \to \infty } \frac{1}{t} \log \bbP \bigg( \frac{T_{\lfloor tu\rfloor }}{   \lfloor tu\rfloor }
	 \geq \frac{t }{\lfloor tu\rfloor } \bigg) \leq\limsup_{t \to \infty } \frac{1}{t} \log \bbP \bigg( \frac{T_{\lfloor tu\rfloor }}{   \lfloor tu\rfloor } \geq \frac{1}{ u} \bigg) 
	  \leq  -u J_+(1/u) = -I(u) \, . \end{equation}
	Above we have used that $J_+$ is  increasing   on $( \th_c^+, +\infty)$.

For the second term we apply the strong Markov property at time $T _{\lfloor tu\rfloor }$ (cf. Definition \ref{corpo1}--(iv)) to get
	\[ \bbP ( T_{\lfloor tu\rfloor }<t , \, \tilde{T}_0 - T_{\lfloor tu\rfloor } \geq ct , \,  S_t < \infty ) \leq \bbP ( T_{-\lfloor tu\rfloor } \geq ct ) \, . \]
Therefore, by the  LDP for the hitting times $T_{-n}$, $n \to \infty$, and the fact that $J_-$ is increasing on $(\th _c^-,+\infty)$, we obtain 
	\begin{equation}\label{siberia2}
	\begin{split} 
	\limsup_{t \to \infty} \frac{1}{t} \log \bbP 
	( T_{\lfloor tu\rfloor }<t , \, \tilde{T}_0 - T_{\lfloor tu\rfloor } \geq ct , \,  S_t < \infty ) 
	&\leq \limsup_{t \to \infty} \frac{1}{t} \log \bbP \bigg( \frac{T_{-  
	\lfloor tu\rfloor  
	}}{\lfloor tu \rfloor } \geq \frac{c}{u} \bigg) \\
	& \leq - u J_- ( c/u) = -c I(-u/c) \, . 
	\end{split} \end{equation}
For the third term in the r.h.s. of \eqref{inizio} one has to deal with the critical points of $J_\pm$, so the idea is to localize things.
 Fix $m \in \bbN$ positive. Fix $0<u'<u$, hence $1/ \lfloor tu \rfloor \leq 1/ t u'$ for $t$ large (as we assume). We take  $u'$ very near to $u$  such  that $1/u' > \th_c^+$ and $c/u'> 1/ \th _c^-$.
 Then
	\begin{equation}\label{siberia3} \begin{split}
	\bbP \bigg(  \frac{T_{\lfloor tu\rfloor }}{ \lfloor tu \rfloor  } & < \frac{1}{u'} , \,  \frac{\tilde{T}_0 - T_{\lfloor tu \rfloor }}{\lfloor tu \rfloor } < \frac{c}{u'} , \,  S_t < \infty \bigg) \\
	& =
	\sum_{k = 1}^m \sum_{\ell = 1}^{mc} \bbP \bigg( \frac{T_{ \lfloor tu \rfloor} }{\lfloor tu \rfloor }  \in \bigg[ \frac{k-1}{mu'} , \frac{k}{mu'} \bigg) , \, 
	\frac{\tilde{T}_0 - T_{\lfloor tu \rfloor }}{\lfloor tu \rfloor }  \in \bigg[ \frac{\ell -1}{mu'} , \frac{\ell }{mu'} \bigg) , \, 
	S_{t } < \infty \bigg) \\
	& \leq \sum_{k = 1}^m \sum_{\ell = 1}^{mc} \bbP \bigg( \frac{T_{\lfloor tu\rfloor }}{ \lfloor tu \rfloor  }  \in \bigg[ \frac{k-1}{mu'} , \frac{k}{mu'} \bigg] \bigg) 
	\bbP \bigg( 
	\frac{ T_{-\lfloor tu \rfloor }}{\lfloor tu \rfloor }  \in \bigg[ \frac{\ell -1}{mu'} , \frac{\ell }{mu'}\bigg]  \bigg) 
	b_{t - \frac{(k+\ell )t}{m} } 
	\end{split} \end{equation}
where we have applied the strong Markov property at times $T_{\lfloor tu \rfloor }$ and $\tilde{T}_0$ and used that if $s \leq  t$ then $b_s \geq  b_t$ since $S_s \leq S_t$.
Now we analyze each term separately.
Define
	\begin{equation} \label{oscillazioni} 
	\begin{split}
	 w_+ ( r , \d ) & := \max \{ | J_+ (s) - J_+ (t) | : s,t \in [ \th_c^+ , r ] , |s-t| \leq \d \} \,, \\
	w_- ( r , \d ) & := \max \{ | J_- (s) - J_- (t) | : s,t \in [ \th_c^- , r ] , |s-t| \leq \d \} \, ,
	\end{split} 
	\end{equation}
	with the convention that  $ w_\pm ( r , \d )=0$ if $ r < \th_c^\pm$.
The LDP for the hitting times $T_n$ then gives
%and the fact that $J_+$ is decreasing on $(-\infty, \th _c^+)$ and increasing on $(\th_c^+, +\infty)$ then give 
	%\begin{equation}
	%\begin{split}
\begin{multline}
	\limsup_{t\to\infty } \frac{1}{t} \log   \bbP  \Big( \frac{T_{\lfloor tu \rfloor }}{\lfloor tu \rfloor }  \in \Big[ \frac{k-1}{mu'} , \frac{k}{mu'} \Big] \Big) 
	%&
	 \leq - u \inf_{\big[ \frac{k-1}{mu'} , \frac{k}{mu'} \big]} J_+ 
\\	\leq  - u J_+ ( k / mu' ) +u w_+ \big( \frac{k}{mu'} , \frac{1}{mu'} \big) 
% \\ & 
= - \frac{k}{m} I \big( \frac{u'm}{k} \big) + u w_+ \big( \frac{k}{mu'} , \frac{1}{mu'} \big) \, . 
\label{hot_wheels_1}	
%	\end{split} \end{equation}
\end{multline}
Similarly we get 
	%\begin{equation}
	\begin{multline}
	\label{hot_wheels_2}
%&	\begin{split}
	\limsup_{t\to\infty } \frac{1}{t} \log   \bbP  \Big( \frac{T_{-\lfloor tu \rfloor }}{\lfloor tu \rfloor }  \in \Big[ \frac{\ell -1}{mu'} , \frac{\ell }{mu'} \Big] \Big) 
%&
  \leq - u \inf_{\big[ \frac{\ell -1}{mu'} , \frac{\ell }{mu'} \big]} J_- \\
  \leq  - u J_- ( \ell / mu' ) +u w_- \big( \frac{\ell}{mu'} , \frac{1}{mu'} \big)   = - \frac{\ell}{m} I \big( -\frac{um'}{\ell} \big) + u w_- \big( \frac{\ell }{mu'} , \frac{1}{mu'} \big)\, . 
\end{multline}%	\end{split} \end{equation}
	We set 
	\begin{align*}
	&
	W_{k,\ell}:=w_+ \big( \frac{k}{mu'} , \frac{1}{mu'} \big) + w_- \big( \frac{\ell }{mu'} , \frac{1}{mu'}  \big)\,,\\
	&W:= \max \Big\{ w_+ \bigl( \frac{1}{u'} , \frac{1}{mu'} \bigr) \,, \,  w_- \Big( \frac{c}{u'} , \frac{1}{mu'} \Big)\Big \}
	\end{align*}
		 The above   inequalities \eqref{hot_wheels_1} and \eqref{hot_wheels_2}, and the convexity of $I$, we have  for any $\e > 0$ and $t$ large enough that 
	\begin{equation}\label{siberia4} \begin{split}
	\bbP \bigg( \frac{T_{\lfloor tu\rfloor }}{ \lfloor tu \rfloor  }  \in \bigg[ \frac{k-1}{mu'} , \frac{k}{mu'} \bigg] \bigg) &
	\bbP \bigg( 
	\frac{ T_{-\lfloor tu \rfloor }}{\lfloor tu \rfloor }  \in \bigg[ \frac{\ell -1}{mu'} , \frac{\ell }{mu'} \bigg]  \bigg) \\
	% & \leq  e^{ - t \frac{k}{m} I \big( \frac{u'm}{k} \big) + u t w_+ \big( \frac{k}{mu'} , \frac{1}{mu'} \big)   + t\e }
%	 e^{ -t \frac{\ell}{m} I \big( -\frac{u'm}{\ell} \big) + u' t w_- \big( \frac{\ell }{mu'} , \frac{1}{mu'} \big) + t \e }\\
	 & \leq e^{t\e + u'tW_{k,\ell} }
	 e^{-t \big[ \frac{k}{m} I \big( \frac{u'm}{k} \big)  + \frac{\ell}{m} I \big( -\frac{u'm}{\ell} \big) \big] }\\
	 & \leq e^{t\e + u't W_{k,\ell}} e^{ -t \frac{(k+\ell)}{m} I(0)} \leq  e^{t\e  -t \frac{(k+\ell)}{m} (I(0)-\frac{W}{c_0})} \,,
	 \end{split} \end{equation}
where  $c_0:= \min \{ \th_c^+, \th_c^-\}$.	 We explain the last bound.
Note that $W_{k,\ell} =0$ if $ k \leq  \th_c^+ m u'$ and $\ell  \leq \th_c^- m u'$.  On the other hand,
 since $\frac{k}{mu'} \leq \frac{1}{u'}$ and $\frac{\ell}{mu'} \leq \frac{c}{u'}$, we have $W_{k, \ell} \leq W$. Hence it holds 
 \begin{equation}\label{bella_addormentata}
 u' W_{k,\ell} \leq u' W \mathds{1} ( k+\ell >c_0 m u')  \leq \frac{(k+\ell)}{c_0 m } W\,.
 \end{equation}

Let now \begin{equation}\label{carnia2}J := \min\left \{ I(u) , \, c I(-c/u) , \, I(0) \right \} -W/c_0  \,.\end{equation}
%Note that, calling $ x:= \limsup_{t \to \infty} \frac{1}{t} \log b_t  \in [-\infty,0]$, by taking $m$ large enough and $\e$ small we can assume that
%\begin{equation}\label{sofferenza}
Then putting \eqref{inizio}, \eqref{siberia1},  \eqref{siberia2},  \eqref{siberia3} and   \eqref{siberia4}    we have
	\[
	b_t   \leq e^{-t I(u) + t\e } +  e^{ -tc I(-u/c ) + t\e } + 
	\sum_{k=1}^m \sum_{\ell = 1}^{mc}  e^{t\e  -t \frac{(k+\ell)}{m} J}   b_{t - \frac{(k+\ell )t}{m}}
	\,.\]
	Note that 
	 if $k+\ell > m$, then 	$b_{t - \frac{(k+\ell )t}{m}} =1$.
	 Hence we get
	 \begin{equation}\label{trucco}b_t  \leq (2+ m^2 c) e^{- tJ+ t\e  }+e^{ t\e }
	 \sum_{
	 \substack{  (k,\ell) : 1\leq k \leq m , 
         \;1 \leq \ell \leq mc \\ k + \ell \leq m  }
}	e^{- t \frac{(k+\ell )}{m} J} b_{t - \frac{(k+\ell )t}{m}}\,.
	\end{equation}
 	\medskip

	Call $ x:= \limsup_{t \to \infty} \frac{1}{t} \log b_t  \in [-\infty,0]$. Since, given a finite family of functions $\{f_i(t) \}_{i \in I}$,  it holds $  \limsup_{t \to \infty} \frac{1}{t}\log\bigl( \sum _i f_i(t)\bigr) \leq \max _{i \in I}   \limsup_{t \to \infty} \frac{1}{t}\log\bigl( f_i(t)\bigr)  $ 	from \eqref{trucco} we get
	\[ x \leq  \e +  \max_{j: 2\leq j \leq m } \left\{ - \frac{jJ}{m}+
	 \left(1- \frac{j}{m} \right)x\right \}= x+ \e    -\min_{j: 2\leq j \leq m } \frac{j(J+x)}{m}\,.\]	
	 The above bound holds for any $\e>0$, hence we conclude that $0 \leq  -\min_{j: 2\leq j \leq m } \frac{j(J+x)}{m}$. This implies that $J+x \leq 0$, i.e.
	 	\begin{equation}\label{ercolino}
	 \limsup_{t \to \infty} \frac{1}{t} \log b_t \leq - J \,.
	 \end{equation}

Now, let $m \to \infty$ first, so that $W \to 0$ due to the fact that $J_\pm $ are even $C^1$ on $(\a_\pm, \infty)$ and $\a_\pm< \th _c ^\pm$ (see Prop. \ref{J_study} and recall  that $1/u' > \th_c^+$, $c/u'>  \th _c^-$). 
Now we let  $u \to 0$. By Theorems \ref{I_study}, \ref{acqua} and since  $v>0$, $ \min\left \{ I(u) , \, c I(-c/u) , \, I(0) \right \} $  converges to $I(0)$ as $u \to 0$. This leads to the thesis.
\end{proof}

%%%%%%%%%%%%%%%%%%%%%%%%%%%%%%%%%%%%%%%%%%%%%%%

\section{Proof of Theorem \ref{GC} (Gallavotti--Cohen type symmetry)} \label{secGC}
Due to the definition of $I$, $  I( \th)=I(-\th)+c \th$ for all $ \th \in \bbR $ if and only if 
\begin{equation} \label{casa2}
J_+( \th)=J_- (\th)+c  \,, \qquad \forall \th >0\,.
\end{equation}
We now prove that \eqref{casa2} and Item (iii) with $c=-\log C$ are equivalent.
 To this aim,  assume  that $\f_+(\l) = C \f_-(\l)$ for all $ \l \leq \l_c$ and some $C>0$. Then $\log \f_+(\l) = \log \f_-(\l) + \log C = \log \f_-(\l) -c$ for all $\l \in \bbR$. Hence, taking the Legendre transform and recalling the definition \eqref{pocoyo1} of $J_\pm$ as Legendre transform of $\log\varphi_\pm$,  we   get \eqref{casa2}.

%%%%%%%%%%%%%%%%%%%%
On the other hand suppose that \eqref{casa2} holds.   We claim that  $J_+(\th)=J_-(\th)+c$ also for  all $\th \leq 0$. Indeed, since $\a_\pm\geq 0$, the claim follows from Proposition \ref{J_study}--(ii) for $\th<0$. If $\a_+,\a_-$ are both  positive then
Proposition \ref{J_study}--(ii) implies the claim also for  $\th=0$. If $\a_+,\a_-$ are both zero, then \eqref{casa2} and the right continuity of $J_\pm$ at $\a_\pm$ (see Proposition \ref{J_study}--(v)) imply the claim for $\th=0$. 
We now show that $\a_-$ and $\a_+$ must be either both positive or both zero, thus concluding the proof of our claim.
Suppose for example that $\a_-=0$ and $\a_+>0$. Then we would have $J_+(\th)=\infty$ for $\th \in (0, \a_+)$  (by Proposition 
 \ref{J_study}--(ii)). This fact together with \eqref{casa2} implies that $J_-(\th)= +\infty$  for $\th \in (0,\a_+)$. Applying Proposition 
 \ref{J_study}--(ii) to $J_-$ we conclude that $\a_+ \leq \a_-$ thus getting a contradiction. 
 %%%%%%%%%%%%%%%%%%%%%%%%%%%%%%%%%%%%%%%%%%%%%%%%%

  Due to \eqref{casa2} and the above claim  we conclude that $J_+(\th)= J_-(\th)+c$ for all $\th \in \bbR$. $J_+(\th),\;J_-(\th)+c$ are the Legendre transforms of 
$\log \varphi_+,\log \varphi_--c$, respectively,  thought as   extended functions from $\bbR$ to $(-\infty,+\infty]$.   Due to Lemma \ref{silente}
 $\log \varphi_+,\log \varphi_--c$   are convex, l.s.c. and not  everywhere infinite.  Hence, by the 
Fenchel--Moreau Theorem (cf. \cite{Br}) we conclude that $\log \varphi_+=\log \varphi_--c$, i.e. Item (iii) holds with $c=-\log C$.

 We now prove that Item (ii) implies Item (iii).  To this aim assume that $\t_i $ and $w_i$ are independent. Then $f_+ (\l) = \bbE (e^{\l \t_i}) p$ and $f_- (\l) = \bbE (e^{\l \t_i}) q$ for all $\l \leq \l_c$. Combining this with \eqref{cena} we get
	\[ \frac{\f_+(\l)}{\f_-(\l)} = \frac{f_+(\l)}{f_-(\l)} = \frac{p}{q} =: C \,, \qquad \forall \l \leq \l_c \,,\]
which is Item (iii).

Finally we  prove that Item (iii) implies Item (ii). Hence 
assume $\f_+(\l) = C \f_-(\l) $ for all $\l \leq \l_c$.  By \eqref{cena} we have $ \frac{\f_+(\l)}{\f_-(\l)} = \frac{f_+(\l)}{f_-(\l)} = C$. Moreover, taking $\l=0$ in the previous identity, from the definition of $f_\pm$  we deduce that $C=p/q$.

In particular, 
given $\l, \g \leq 0$, we can write
%	\begin{equation*}
%	\bbE \big( e^{\l \t_i + \g w_i } \big) = \bbE \big( e^{\l \t_i} \big) \bbE \big( e^{\g w_i } \big) \qquad \forall \l , \g \in \bbR \, . 
%	\end{equation*} 
	\[ \begin{split}
	\bbE \big( e^{\l \t_i + \g w_i } \big) & 
	= \bbE \big( e^\g e^{\l \t_i  }  \mathds{1}(w_i =1)\big) + \bbE \big( e^{-\g} e^{\l \t_i  } \mathds{1}(w_i = -1)\big) \\
	&  = e^\g f_+(\l) + e^{-\g} f_-(\l) = f_+(\l) \bigg( e^\g + e^{-\g} \frac{q}{p} \bigg) \, . 
	\end{split} \]
On the other hand:
	\[ \bbE \big( e^{\l \t_i  } \big) = 
	\bbE \big( e^{\l \t_i  } \mathds{1}(w_i =1)\big) + \bbE \big( e^{\l \t_i  }  \mathds{1}(w_i =-1)\big)
	= f_+(\l) \bigg( 1 + \frac{q}{p} \bigg) = \frac{f_+(\l)}{p} \]
and
	\[ 	\bbE \big( e^{\g w_i  } \big) = e^\g p + e^{-\g} q = p \bigg( e^\g + e^{-\g} \frac{q}{p} \bigg) \, . \]
Putting all together, we conclude that 
	\[ \bbE \big( e^{\l \t_i + \g w_i } \big) = f_+(\l) \bigg( e^\g + e^{-\g} \frac{q}{p} \bigg) = 
	\bigg( \frac{f_+(\l)}{p} \bigg) \bigg(  p \big( e^\g + e^{-\g} \frac{q}{p} \big) \bigg) = 
	\bbE \big( e^{\l \t_i  } \big) \bbE \big( e^{\g w_i  } \big) \, ,\]
thus implying the independence of  $\t_i,w_i$.

%%%%%%%%%%%%%%%%%%%%%%%%%%%%%%%%%%%%%%%%%%

\section{Proof of Theorem \ref{GE} (LDP via G\"artner--Ellis theorem for Markov rw's)}\label{GE_proof}

We introduce a $\bbZ$--valued  process $( N_t )_{ t \in \bbR_+}$ given by the cell number of $X_t$. More precisely, we set $N_t:=n$ if  $X_t = v_n$ for some $ v \in V \setminus \{\overline{v}\}$.
% For later use it is convenient to set
%\[  \phi: \cV \to \bbZ \,, \qquad \phi(v_n)= n\,.\]
%Note that  $N_t= \phi( X_t)$.
Note that 
 the cell number process is in general not a Markovian process and that $ | X_t^* - N_t| \leq 1$.

We introduce the constant $\k$ defined as  $ \k:= \max  \{ r(x)\,:\,  x \in \cV\}$, where $r(x)= \sum_{y: (x,y) \in \cE} r(x,y)$. Due to the periodicity \eqref{simm1}, $\k$ is a well defined constant in $(0,+\infty)$.

\begin{Lemma} \label{celletta} For each $n \in \bbZ\setminus\{0\}$ and $t \in \bbR_+$ it holds $ \bbP( N_t=n) \leq e^{\k t |n| - | n| \log |n|}$. In particular, 
for each $\l \in \bbR$ and $t \in \bbR_+$ it holds $\bbE( e^{\l N_t} ) <+\infty$.
\end{Lemma}
\begin{proof} The event $\{N_t=n \}$ implies that the r.w. $X$ has performed at least $|n|$ jumps within time $t$. On the other hand, by definition of $\k$,  the random walk $X$ waits at each $x\in \cV$ an exponential time of mean at least $1/\k$. Hence  (by a coupling argument) $\bbP(N_t=n ) \leq P( \cZ_t\geq |n|)$, where $\cZ_t$ is a Poisson random variable with intensity $\k t$.  Since  $E\left( e^{a \cZ_t}\right) = e^{ \k t ( e^a-1)}$, by Chebyshev inequality with $a=\log |n|$ we get
\[ \bbP(N_t=n) \leq P( \cZ_t\geq |n|)\leq e^{- |n| \log |n|} E \left( e^{ \cZ_t \log|n|}\right) = 
e ^{- |n| \log |n|+ \k t (|n|-1) }\,, 
\]
thus proving the bound on $\bbP (N_t=n)$. As a consequence, we obtain 
 \[\bbE( e^{\l N_t} ) \leq 1+ 2 \sum _{n=1}^\infty e^{ \l |n| + \k t |n|- |n| \log |n| } < +\infty\,.
 \qedhere
 \]
\end{proof}
We now define a new  function $F: \left(V\setminus\{\overline{v}\} \right) \times \bbR \times \bbR_+ \ni (v, \l, t) \to  F(v, \l ,t ) \in \bbR_+$ 
as
\begin{equation}\label{def_F} F(v, \l, t)= \sum_{n \in \bbZ} e^{\l n } \bbP( X_t=v_n) = \bbE\left( e^{\l N_t} \mathds{1}( X_t=v_n \text{ for some } n \in \bbZ)  \right)\,.\end{equation}
Recall that, given $v \not = w$ in $ V \setminus \{ \overline{v}\}$, we  have set \begin{equation*}%\label{venticello}
 r(v):= r(v_n) \,, \;\;  r_-(w,v):= r(w_{n-1},v_n)\,,\;\;
 r_0(w,v):=r( w,v)\,, \; \; r_+(w,v):= r(w_{n+1},v_n)\,. \end{equation*}
 
\begin{Lemma}\label{speck} Given $\l \in \bbR$ 
consider the finite  matrix $\cA(\l)$ defined in \eqref{Freitag}. Consider the vector--valued  function $ \bbR_+ \ni t \mapsto F^{(\l)} (t)\in \bbR^{ V \setminus \{\overline{v} \} }$ defined as $F^{(\l)} (t)_v := F(v, \l, t)$. Then $F^{(\l)}(\cdot)$ is $C^1$ in $t$ and 
\begin{equation}\label{paoletta}
\partial _t F^{(\l)}(t)= \cA(\l) F^{(\l) } (t )\,.
\end{equation}
\end{Lemma}
\begin{proof} Fixed $v ,\l $, we write $F(v, \l , \cdot)$ as the  function series $F(v, \l , t)=\sum _{n\in \bbZ} f_n(t)$, where $f_n(t)= e^{\l n } \bbP (X_t=v_n)$.  
By \cite{N}[Theorem 2.8.2], the function $\bbR_+ \ni t \mapsto \bbP( X_t=v_n) \in [0,1]$ is differentiable  and moreover
\begin{multline}
 \partial _t  \bbP( X_t=v_n) = - r(v_n) \bbP( X_t=v_n)+ \sum _{w\in V \setminus \{\overline{v},v  \}} \Big[  r(w_{n-1},v_n) \bbP( X_t=w_{n-1})\\
 +  r(w_n,v_n)\bbP( X_t=w_n)  + r(w_{n+1},v_n) \bbP( X_t=w_{n+1})\Big]
 %=  - r(v) \bbP( X_t=v_n)\\
 %  + \sum _{w\in V \setminus \{\overline{v},v  \}} \Big[  r_-(w,v) \bbP( X_t=w_{n-1})+  r_0(w,v)\bbP( X_t=w) + r_+(w,v) \bbP( X_t=w_{n+1})\Big]
 \,.
 \end{multline}
 Then, by Lemma \ref{celletta}, we conclude that, for  $M>0$ and $n \in \bbZ$ with $|n| \geq 2$, it holds  
 \begin{align*}
&  \|f_n \|_{ L^\infty [-M,M] } \leq   e^{ |\l|\cdot  |n|+ \k M  -|n| \log |n|}\,, \\
&  \| \partial_t f_n \|_{ L^\infty [-M,M] } \leq 4 \k |V| e^{ |\l|\cdot  |n|+ \k( |n|+1) M  -(|n|-1) \log (|n|-1) }\,,
\end{align*}
The  space $C^1[-M,M]$ (of functions $C^1$ on  $(-M,M)$, such that they and their first derivates have  continuous extensions to $[-M,M])$ is a Banach space  endowed with the norm  $\| f\|:= \|f \|_{ L^\infty [-M,M] }
+ \|\partial_t f \|_{ L^\infty [-M,M] }$. We therefore conclude that $F(v, \l , t)=\sum _{n\in \bbZ} f_n(t)$ belongs to $C^1(\bbR)$ and $\partial _t F(v, \l , t)=\sum _{n\in \bbZ} f'_n(t)$, i.e. 
 \begin{equation*}
\begin{split}
\partial_t F(v,\l, t)& =  \sum_{n \in \bbZ} e^{\l n}  \partial _t  \bbP( X_t=v_n) \\
&=- r(v) F(v,\l,t)+  \sum _{w\in V \setminus \{\overline{v},v  \}}\bigl[ e^\l r_-(w,v)  +r_0(w,v)+e^{-\l} r_+(w,v)\bigr] F(w,\l,t)\,.
\end{split}
\end{equation*}
This concludes the proof.
\end{proof}

We can now conclude the proof of Theorem \ref{GE}. Since $ | X_t^* - N_t| \leq 1$, it is  enough to prove the same LDP for $N_t/t$.
Due to Lemma  \ref{speck} we have $ F^{(\l)}(t)=e^{(t-1) \cA(\l) } F^{(\l) } (1 ) $. Since the graph $\cG=(\cV,\cE)$ is connected, Definition \eqref{def_F} implies  that the vector $ F^{(\l) } (1 )$ has strictly positive entries. In particular we can write
 $$ \bbE( e^{\l N_t} ) = e^{-\k (t-1)}  \sum _{v,v' \in V \setminus \{\overline{v} \}} \bigl[e^{[ \cA(\l)+\k ] ( t-1)}\bigr]_{v,v'}F(v',\l,1)\,,$$
where $ \k:= \max  \{ r(x)\,:\,  x \in \cV\}$, as above. Note that  $\cA(\l)+\k$ is an irreducible matrix with nonnegative  entries and therefore, by Perron--Frobenius theorem, it has  a simple positive eingevalue $\bar \g $  and an  associated eigenvector with  strictly positive entries $(a(v) )_{v \in V \setminus\{\overline{v} \} }$, while any other eigenvalue $\bar \g'$ is such that $|\bar  \g' | \leq \bar \g$ (in particular, $\cR(\bar \g') < \cR(\bar \g)$). The above eigenvalue $\bar \g$ is the so called Perron--Frobenius eigenvalue and equals the spectral radius of $\cA(\l)+\k$ (note that $\bar \g=\bar\g(\l))$.
 Call $$
 \begin{cases}
 C(\l) := \max\{   F(v,\l,1)/a_v \,: v \in V \setminus\{\overline{v} \}\}\,,\\
 c(\l)  \;:= \min\{  F(v,\l,1)/a_v \,: v \in V \setminus\{\overline{v} \}\}\,.
\end{cases}$$
Note that $C(\l) , c(\l)$ are positive constants.
Then
\begin{align*}
& \bbE( e^{\l N_t} )\leq C(\l) e^{-\k (t-1)}   \sum _{v,v' \in V \setminus \{\overline{v} \}} \bigl[e^{[ \cA(\l)+\k ]  (t-1)}\bigr]_{v,v'} a_{v'}=e^{(\bar\g-\k)( t-1) }   C(\l) \sum _{v\in V \setminus \{\overline{v} \}} a_v  \,,\\
& \bbE( e^{\l N_t} )\geq c(\l) e^{-\k (t-1)}   \sum _{v,v' \in V \setminus \{\overline{v} \}} \bigl[e^{[ \cA(\l)+\k ] ( t-1)}\bigr]_{v,v'} a_{v'}=e^{(\bar \g-\k)( t-1) }   c(\l) \sum _{v\in V \setminus \{\overline{v} \}} a_v\,.
\end{align*}
It then follows that the limit $\lim_{t \to \infty} \frac{1}{t} \ln \bbE( e^{\l  N_t} )$ exists and equals $ \bar\g(\l) -\k$, which  corresponds to $\L(\l)$ by the previous discussion.

 By finite--dimensional perturbation theory \cite{Katino},   the  Perron--Frobenius eigenvalue  $\bar \g= \bar \g(\l)$ is differentiable in $\l$, thus implying that $\L(\l)$ is differentiable in $\l$. At this point  the thesis   follows from G\"artier--Ellis theorem (cf. \cite{dH}[Lemma V.4 and Theorem V.6]).

%%%%%%%%%%%%%%%%%%%%%%%%%%%%%%%%%%%%%%%%%%%%%%%%%

\section{Proof of Theorem \ref{teo3} (GC  type symmetry for Markov rw's)}\label{secGC_bis}

We start with   a technical result, that is also useful in the applications for the computation of the functions $f_\pm (\l)$.  Consider a generic stochastic process $(X_t)_{t \in \bbR_+}$ as in Definition \ref{corpo1}.  Define 
\[J_1:= \inf\left\{ t >0\,:\, X_t \in \{-1_*,0_*,1_*\}, \; \exists s \in (0,t) \text{ with } X_s \not = X_0 \right\}\,,\]
and set
	\begin{equation}\label{caffe}
	\begin{split}
	\tilde{f}_\pm (\l)  & := \bbE_{0_*} ( e^{\l J_1} \mathds{1} (X_{J_1} = \pm 1_*) )\, , \\
	\tilde{f}_0 (\l)  & := \bbE_{0_*} ( e^{\l J_1} \mathds{1} (X_{J_1} = 0_*)) \, . 
	\end{split}
	\end{equation}
\begin{Lemma}\label{calcolo_LD}
If $ \tilde f_0(\l) < 1$, then 
\begin{equation*}
	f_+(\l) =% \frac{  \tilde{f}_+(\l)}{1-(1-\tilde{p} -\tilde{q})\tilde{f}_0 (\l)} 
	\frac{ \tilde{f}_+(\l)}{1- \tilde{f}_0 (\l)} \, ,
	\qquad
	 f_-(\l) = \frac{  \tilde{f}_-(\l)}{1-\tilde{f}_0 (\l)} \, .
	\end{equation*}
	If 
	$ \tilde f_0(\l) \geq1$,  then  $f_+(\l)= f_-(\l)= +\infty$.
	\end{Lemma}	

\begin{proof}
We call $J_k$'s  the consecutive  times at which the stochastic process $\bigl( X_t \bigr)_{t \geq 0}$   hits  the states of type $n_*$:
 \begin{equation*}\label{jumping}
	\begin{cases}
	J_0 := 0 \\
	J_k := \inf \{ t > J_{k-1}\,:\, X_t \in \{ -1_* , 0_* , 1_* \}\,,\; \exists s \in (J_{k-1},t)  \text{ with } X_s \not = X_{J_{k-1} } 
  \}  \quad k\geq 1 \, . 
	\end{cases} 
	\end{equation*}
We can write
	\begin{equation} \label{Sdecomp} 
	S = \sum_{k=0}^\infty \mathds{1} ( X_{J_0} = \ldots = X_{J_k} = 0_* , X_{J_{k+1}} \in \{ -1_* , 1_* \}  ) J_{k+1} \, .
	\end{equation}
Taking the exponential  at both sides and multiplying by $\mathds{1} (X_S=1_*)$ we get
\[ e^{\l S} \mathds{1}(X_S = 1_* ) 
	= \sum_{k=0}^\infty \mathds{1} ( X_{J_0} = \ldots = X_{J_k} = 0_* , X_{J_{k+1}} =1_*  ) e^{\l J_{k+1} } \]
Note that, by Definition \ref{corpo1}, w.r.t. the the probability measure $\bbP_{0_*} ( \cdot | X_{J_0} = \ldots = X_{J_k} = 0_* , X_{J_{k+1}} =1_*)$, the random variables $e^{\l (J_i-J_{i-1})}$, $1\leq i \leq k+1$, are independent with expectation $\bbE_{0_*}( e^{\l J_1} | X_{J_1}=0_*$) if $1\leq i \leq k$ and  
 $\bbE_{0_*}( e^{\l J_1} | X_{J_1}=1_*$) if $i=k+1$.
 Hence,
 \[ \begin{split}
	f_+ (\l)& =\bbE_{0_*} ( e^{\l S} \mathds{1} (X_{S} = 1_*) ) 
	\\ &= \sum_{k=0}^\infty \bbP_{0_*} (X_{J_1} = 0_*)^k \bbP_{0_*} (X_{J_1} = 1_*) 
	\bbE_{0_*} ( e^{\l J_1} |X_{J_1} = 0_*)^k \bbE _{0_*}( e^{\l J_1} |X_{J_1} = 1_* ) 
	\\ & = \sum_{k=0}^\infty \tilde{f}_0 (\l)^k \tilde{f}_+(\l) \,. %= \frac{\tilde{f}_+(\l)}{1- \tilde{f}_0 (\l)}
	\end{split}\, . \]
A similar expression holds for $f_-(\l)$. At this point it is immediate to derive the thesis.\end{proof}

\medskip

Let us now come back to the same context of Section \ref{amiciGC}: $(X_t)_{t \in \bbR_+}$ is a Markov random walk on the quasi 1d lattice $\cG= (\cV, \cE)$, with positive rates $r(x,y)$, $(x,y) \in \cE$, such that \eqref{simm1} and \eqref{mercato} hold. For the rest of this section, we refer to Markov random walk without state explicitly that they are Markov.  Due to \eqref{mercato} in the figures of $G,\cG$ we draw only unoriented edges with the convention that for  each unoriented edge $\{x,y\}$ the graph in consideration presents both the edge $(x,y)$ and the edge $(y,x)$. 

Recall that
given an edge  $(u,v) \in E$  in the fundamental graph $G=(V,E)$,  we have defined (cf. \eqref{def_r_G}) $ r(u,v) = r (\pi(u), \pi(v) )$ where
 $\pi$ is the map $ V \to \cV$ such that $\pi(u)= u_0$ if $ u\not = \overline{v}$ and $\pi(\overline{v}) = \underline{v}_1=1_*$.  Given $v \in V$ we set  
 \begin{equation}\label{pierre}
  r(v):= \sum _{y : ( \pi(v),y) \in \cE} r( \pi(v), y) \,.
 \end{equation}
 Note that $r(\underline{v})= r(\overline{v})$. We point out that   the map $\pi: V \to \cV$ does not induce a graph embedding of $G$ into $\cG$. 
\begin{figure}[!ht]
    \begin{center}
     \centering
  \mbox{\hbox{
 \boxed{ \includegraphics[width=.3\textwidth]{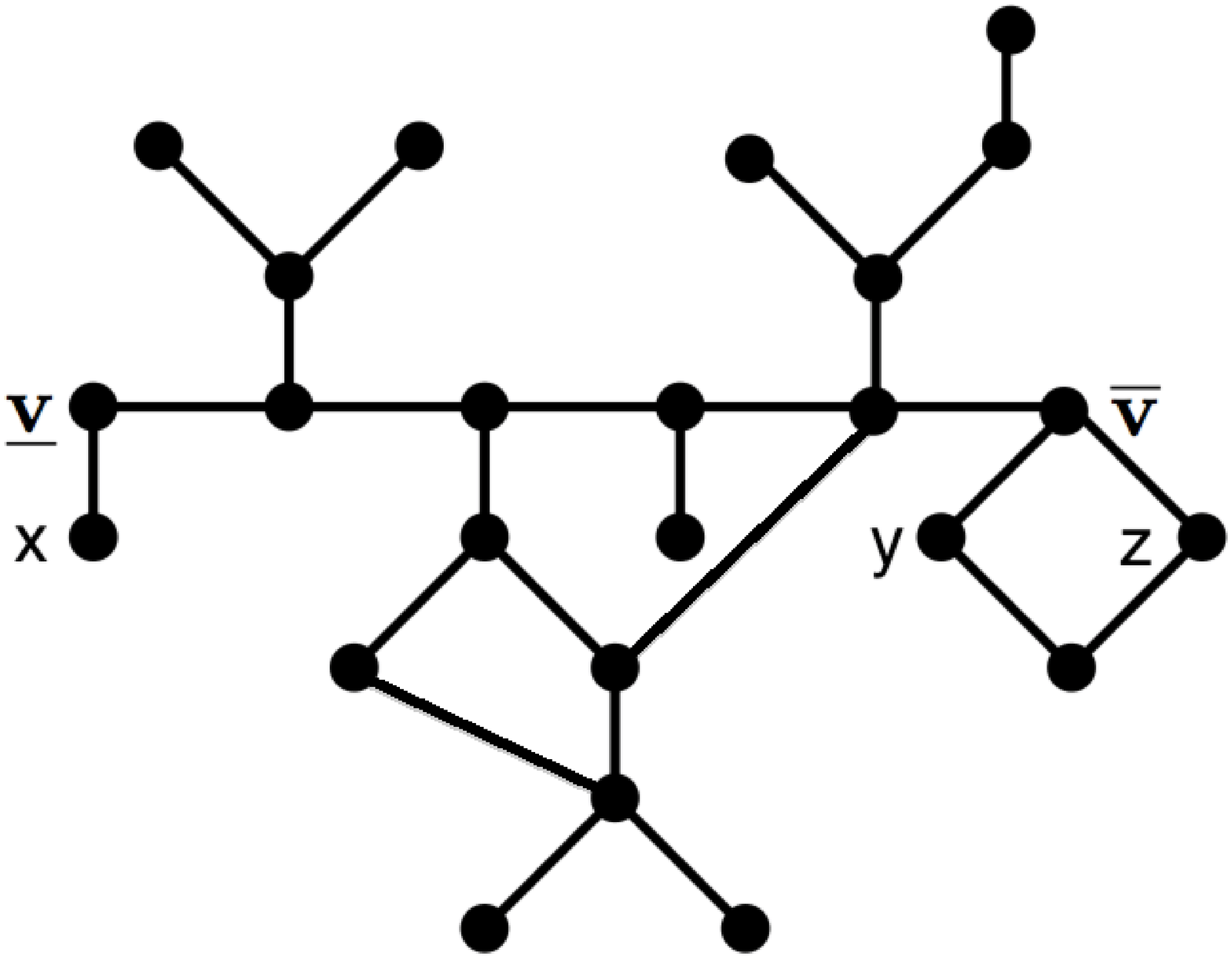}}}}
            \end{center}\caption{Example of fundamental graph $G$}\label{vendemmia}
        %    \caption{Two examples of $(\underline{v}, \overline{v})$--minimal graphs for which the functions $f_+$, $f_-$ have the same constant of proportionality.}
  \end{figure} Indeed,  problems come from the neighbors of  $\underline{v}, \overline{v}$ in  $G$. Consider for example the fundamental graph $G$ in Figure \ref{vendemmia}.  Then $x_0, y_{-1} ,z_{-1}$ are  neighboring points of  $\underline{v}_0$, while  $x_1, y_0,z_0$ are neighboring points of $\underline{v}_1$. 
Despite this phenomenon, the map $\pi$ induces an isomorphism between the family of paths 
$(x_0,x_1, \dots, x_m)$ 
in $G$  from $\underline{v}$ to $\overline{v}$ with interior points in  $V \setminus\{\underline{v}, \overline{v}\}$ and the  family of paths $(x_0',x'_1, \dots, x_m')$ in $\cG$  from $\underline{v}_0=0_*$ to $\underline{v}_1=1_*$ with interior points in  $\cV \setminus\{0_*,1_*\}$, moreover it holds $r(x_i,x_{i+1})= r\bigl( \pi(x_i), \pi(x_{i+1}) \bigr)$ for $0 \leq i <m$ and $r(x_i)=r(\pi(x_i))$ for $0 \leq i \leq m$. This property will be used below. 

\medskip

By Theorem   \ref{GC}   the Gallavotti--Cohen type symmetry  \eqref{verabila} is satisfied   for some constant $\D$ if and only if  $\f_+(\l)/\f_-(\l)= e^{\D}$ for all $\l \leq \l_c$. On the other hand, 
 by \eqref{cena} and the above Lemma \ref{calcolo_LD}, it holds
\begin{equation}\label{salatini}
\frac{\varphi_+(\l)}{\varphi_-(\l)}=
\frac{f_+(\l)}{f_-(\l)} =\frac{ \tilde{f}_+(\l) }{ \tilde{f}_-(\l)}\,, \qquad \forall \l \leq \l_c\,.
\end{equation}

 Given an integer $m \geq 1$, let $\cA_m$ be the family of sequences $(x_0,x_1, \dots, x_m)$ such that
 $x_0= \underline{v}$, $x_m = \overline{v}$, $(x_i,x_{i+1}) \in E$ for all $i:0\leq i <m$ and $ x_i \in V \setminus \{\underline{v}, \overline{v} \} $ for all $0<i < m$. We call $\cA_m^*$ the family of 
 sequences  satisfying the same properties as above  when exchanging the role of  $\underline{v}$ and $\overline{v}$. 
 Then we can write
\begin{equation}\label{giaco0}
\begin{split}
\tilde{f}_+(\l)& = \sum_{m=1}^\infty \sum _{(x_0, x_1, \dots, x_m) \in \cA _m}
 \int _{\bbR_+^{m-1}} dt_1 dt_2 \dots dt_{m-1}  e^{ \sum_{i=0}^{m-1}(\l-r(x_i) )  t_i  } \prod _{i=0}^{m-1} r(x_i,x_{i+1}) \\
 &=
 \begin{cases}
  \sum_{m=1}^\infty {\sum}_{(x_0, x_1, \dots, x_m) \in \cA _m}
  \prod _{i=0}^{m-1} r(x_i,x_{i+1}) \prod _{i=0}^{m-1} \frac{1}{r(x_i)-\l} & \text{ if } \l < \min_{x \in V} r(x) \\
  +\infty & \text{ otherwise}\,.
  \end{cases}
 \end{split}
 \end{equation}
% The above computation comes from the interpretation that $(x_0, x_1,\dots, x_m)$ are the states visited by $X$ when going from $\underline{v}$ to $\overline{v}$ through interior points of $G$, and $t_i$ is the holding time at $x_i$. 
Given $\g= (x_0,x_1, \dots, x_m)$ and given $e \in E$ we write $N_e(\g)$ for the number of indices $i :0\leq i \leq m-1$ such that $(x_i,x_{i+1})=e$. Then the above formula can be rewritten as
 \begin{equation}\label{giaco1}
\begin{split}
\tilde{f}_+(\l) &=
 \begin{cases}
  \sum_{m=1}^\infty \sum _{\g=(x_0, x_1, \dots, x_m) \in \cA _m}
  \prod _{e \in E} r(e) ^{N_e(\g)}  \prod _{i=0}^{m-1} \frac{1}{r(x_i)-\l} & \text{ if } \l < \min_{x \in V} r(x) \,,\\
  +\infty & \text{ otherwise}\,.
  \end{cases}
 \end{split}
 \end{equation}
A similar formula holds for $\tilde{f}_-(\l)$.

\medskip

 Suppose now that $G$ is \emph{$(\underline{v}, \overline{v})$}--minimal. We want to prove that \eqref{verabila} is satisfied with $\D$ given by \eqref{alexey}. Call $G_1, G_2, \dots, G_k$ the subgraphs attached to the path $\g_*=(z_0,z_1, \dots, z_n)$ as in Def. \ref{pasopanori}, such that  each $G_a$ has exactly one point in common with $\{z_1,z_2, \dots, z_{n-1}\}$ (recall that $z_0 = \underline{v}$ and $z_n= \overline{v}$). Then given  $(x_0,x_1, \dots, x_m) \in \cA_m$  there exist indices $$0 \leq i_1 <j_1<i_2<j_2<  \dots< i_{r-1}<j_{r-1} < i_r< j_r \leq m$$ such that for any $k: 1\leq  k \leq r$ the subpath (called \emph{excursion})
 $$ (x_{i_k}, x_{i_k+1}, \dots, x_{j_k})$$ 
 satisfies: (i) $x_{i_k}= x_{j_k}$ 
 and such a point belongs to $\g$, (ii) the points  $x_{i_k+1}, \dots,x_{j_k-1}$ are in $G_a\setminus \{z_0,z_1, \dots, z_n\}$ for some $a:1\leq a \leq k$.  When $r=0$ then there is no excursion and the path $(x_0,x_1, \dots, x_m)$ has support in $\{z_0,z_1, \dots, z_n\}$.
 Call $(x_0,x_1, \dots, x_m)^*$ the new path  obtained  by inverting $(x_0,x_1, \dots, x_m)$ with the exception that the  excursions  inside are performed in their original  orientation: 
 \begin{equation*}
 \begin{split}
 (x_0,x_1, \dots, x_m)^*:=
 (& x_m, x_{m-1},\dots, x_{j_r+1}, 
 \underline{x_{i_r},  x_{i_r+1}, x_{i_r+2}, \dots,  x_{ j_r}}, x_{i_r-1},x_{i_r-2}, \dots,\\
& x_{j_{r-1} +1}, \underline{ x_{i_{r-1}} , x_{i _{r-1}+1}, \dots x_{j_{r-1}-1}, x_{j_{r-1} }  }, x_{i_{r-1}-1},\dots  \dots,x_{j_1+1} ,  \\
& \underline{  x_{i_1} ,
 x_{{i_1}+1}, \dots, x_{j_1-1}, x_{j_1} } , x_{i_1-1}, x_{i_1-2} , \dots, x_2,x_1,
  x_0)\,.
  \end{split}
 \end{equation*}
 Above we have underlined the excursions (note they appear in their original orientation).
  We point out that the map $\cA_m \ni  (x_0,x_1, \dots, x_m) \to (x_0,x_1, \dots, x_m)^* \in \cA_m^*$ is a bijection. Hence it holds
 \begin{equation}\label{giaco2}
\tilde{f}_-(\l)=
 \begin{cases}
  \sum_{m=1}^\infty \sum _{\g=(x_0, x_1, \dots, x_m) \in \cA _m}
  \prod _{e \in E} r( e) ^{N_e(\g^*)}  \prod _{i=1}^{m} \frac{1}{r(x_i)-\l} & \text{ if } \l < \min_{x \in V} r(x) \\
  +\infty & \text{ otherwise}\,.
 \end{cases}
 \end{equation}
Note that  $  \prod _{i=0}^{m-1} \frac{1}{r(x_i)-\l}=  \prod _{i=1}^{m} \frac{1}{r(x_i)-\l}$ since $r(\underline{v})= r( \overline{v})$. 
  By construction 
  $N_e(\g)= N_e(\g^*)$ if $e $ is not of the form $(z_i, z_{i\pm 1})$.  On the other hand,  
 \begin{equation}
 \begin{split}
 & \frac{ \prod _{i=0}^{n-1} r(z_i, z_{i+1}) ^{N_{(z_i, z_{i+1} )} (\g)}
  \prod _{i=0}^{n-1} r(z_{i+1},z_i) ^{N_{( z_{i+1},z_i )} (\g)}
  }{
  \prod _{i=0}^{n-1} r(z_i, z_{i+1}) ^{N_{(z_i, z_{i+1} )} (\g_*)}
  \prod _{i=0}^{n-1} r(z_{i+1},z_i) ^{N_{( z_{i+1},z_i )} (\g_*)}
    }=\\
    &\frac{ \prod _{i=0}^{n-1} r(z_i, z_{i+1}) ^{N_{(z_i, z_{i+1} )} (\g)}
  \prod _{i=0}^{n-1} r(z_{i+1},z_i) ^{N_{( z_{i+1},z_i )} (\g)}
  }{
  \prod _{i=0}^{n-1} r(z_i, z_{i+1}) ^{N_{(z_{i+1}, z_{i} )} (\g)}
  \prod _{i=0}^{n-1} r(z_{i+1},z_i) ^{N_{( z_{i},z_{i+1} )} (\g)}
    }=\\
  &  \prod _{i=0}^{n-1}  \left( \frac{ r(z_i, z_{i+1})}{r(z_{i+1},z_i) } \right) ^{ N_{(z_i, z_{i+1} )} (\g)-N_{(z_{i+1}, z_{i} )} (\g)}=  \prod _{i=0}^{n-1}   \frac{ r(z_i, z_{i+1})}{r(z_{i+1},z_i) }\,,
 \end{split}
 \end{equation} 
   since it must be $N_{(z_i, z_{i+1} )} (\g)-N_{(z_{i+1}, z_{i} )} (\g)=1$ for any path $\g\in \cA_m$ for some $m \geq 1$. Due to \eqref{giaco1}, \eqref{giaco2} and the previous observations, we get that $\tilde{f}_+(\l)/ \tilde{f}_-(\l)= e^\D$ with $\D$ given in \eqref{alexey}. Due to \eqref{salatini} and Theorem \ref{GC} we get \eqref{verabila}.

   \medskip
   
  We now prove  the reverse implication.   Consider the oriented subgraph $\hat G = (\hat V, \hat E)$ consisting of the points in $V$ and edges in $E$  that appear in some path $\g $ as $\g $ varies in $\cA_m$ and $m$ varies in $\{1,2,\dots\}$.  
  %Note that in the computation of $\tilde{f}_\pm$ only points in $\hat V$ and  edges in $\hat E$ appear (see \eqref{giaco0} and \eqref{giaco1}). 
 % Note that the map $\pi$ restricted to $\hat G$ gives a graph embedding of $\hat G$ into $\cG$.  
  %In what follows we prove the following
  
  \begin{Proposition}\label{davide} Suppose that the fundamental graph $G$ is not $(\underline{v}, \overline{v})$--minimal.
 Fix
   $\bigl( r(e)\,: \,e \in E \setminus \hat E \bigr) \in (0,+\infty)^{E \setminus \hat E}$. Call $\cR\subset (0,+\infty)^{\hat E}$  the family of  vectors  $\bigl(r(e): e \in  \hat E \bigr)\in (0,+\infty)^{\hat E}$   for which the random walk  on $\cG$ induced by $(r(e): e \in E)$ satisfies  the
  Gallavotti--Cohen type symmetry \eqref{verabila}   for some  constant $\D$,  depending on 
  $(r(e): e \in E)$. Then $\cR$ has 
   zero  Lebesgue measure  in  $(0,+\infty) ^{ \hat E}$. 
   \end{Proposition}    Since $\hat E \not= \emptyset$ and by Fubini theorem, this would conclude the proof of Theorem \ref{teo3}. The proof is  in part based on complex analysis.

  \subsection{Proof of Proposition \ref{davide}}

    From now on $r(e)$, $e \in E \setminus \hat E$, are fixed positive constants. We first prove some preliminary results.

  \begin{Lemma}\label{como} 
  Define the open subset $\O\subset (-\infty,0) \times (0,+\infty)^{\hat E}$ as the family of vectors $ ( \l, (r(e))_{e \in \hat E} )$ with $ r(e)>0\;\;\forall e \in \hat E$ and  $ -\l> 3 \max_{v \in \hat V} r(v) +1 $, where $r(v)$ is the value defined in \eqref{pierre} for the random walk on $\cG$ induced by $\bigl( r(e) \,:\, e \in E \bigr)$.  
  
 Consider the positive function $h_\pm \left( \l, \bigl( r(e)\bigr)_{e \in \hat E} \right)$   defined on $\O$   as the function $\tilde f_\pm(\l)$ for the random walk on $G$   induced by $\left( r(e)\,:\,e \in E\right)$. Then there exists an holomorphic function $h^*_\pm : \O _* \to \bbC$ defined  on an open subset $\O_* \subset \bbC \times \bbC^{\hat E}$ such that $\O= \O_* \cap( \bbR \times \bbR^{\hat E})$ and  
 $h_\pm$ is the restriction to $\O$ of the function $h^*_\pm$.
  \end{Lemma}

   \begin{proof}  In what follows, to simplify the notation, we write $\underline{ r} $ instead of $\bigl( r(e)\,:\, e \in \hat E\bigr)$.  In general $\underline{r}$ will be an element of $\bbC^{\hat E}$. 
 Given $v \in \hat V$ we define the map $ \phi_v: \bbC^{\hat E} \to \bbC$ as 
\begin{equation}
\phi_v (\bar r):= 
\begin{cases} 
\sum _{ (v,y) \in  \hat E} r( v,y)   & \text{ if } v \in \hat V\setminus \{ \underline{v}, \overline{v}\}\,,\\
\sum _{ (v,y) \in  \hat E} r( v,y) + \sum _{ (\underline{v}, y) \in E\setminus \hat E } r(\underline{v},y )+ \sum _{ (\overline{v}, y) \in E\setminus \hat E} r(\overline{v},y ) & \text{ if } v = \underline{v}, \overline{v}\,.
\end{cases}
\end{equation}  
    Recall that $\sum _{ (\underline{v}, y) \in E\setminus \hat E } r(\underline{v},y )$ and $ \sum _{ (\overline{v}, y) \in E\setminus \hat E} r(\overline{v},y ) $ are fixed positive constants since the values $r(e)$, $e \in E \setminus \hat E$, have been fixed once for all.
  Moreover note that  for $\underline{r} \in (0,+\infty)^{\hat E}$ it holds  $\phi_v (\underline{ r})= r(v)$, where $r(v)$ is the value defined in \eqref{pierre} for the random walk on $\cG$ induced by $\bigl( r(e) \,:\, e \in E \bigr)$. 

Given $\underline{r} \in \bbC ^{\hat E} $ we define $\Re (\underline{r}) \in \bbR^{\hat E}$ as the vector whose entries are the real part of the entries of $\underline{r}$, i.e. 
$$ \Re \left(\underline{r}\right)  (e) := \Re\left ( r(e) \right) \,, \qquad e \in \hat E \,.$$
    We define $\O_*\subset \bbC \times \bbC ^{\hat E} $ as the set of vectors $(\l, \underline{r})$ satisfying the following properties:
    \begin{itemize}
   \item[(i)]
   $  \Re ( \underline{ r}) \in (0,+\infty)^{\hat E}$,
   \item[(ii)] $|r(e) | \leq 2 \Re \left(r(e) \right) \; \forall e \in \hat E $,
   \item[(iii)] $- \Re (\l) >  3 \phi_v \left(    \Re \left(\underline{r}\right)  \right)+1 $ for all   $v \in \hat V$.
 \end{itemize}
 Note that 
 $ \O _* \cap  (\bbR \times 
\bbR^{\hat E})= \O$. 
 
    For each $\g =(x_0,x_1, \dots,x_m)\in \cA_m$, $m \geq 1$, we consider the holomorphic function (cf. \cite{GR}) $g_\g : 
   \O_*
   \to \bbC$  defined as 
   \begin{equation}
    g_\g ( \l, \underline{r}) :=
     \prod _{e \in \hat E} r( e) ^{N_e(\g)}  \prod _{i=0}^{m-1} \frac{1}{\phi_{x_i}(\underline{r})-\l} \,.
     \end{equation} 
     Recall that $N_e(\g)$ counts the number of times the edge $e$ appears along the path $\g$ and 
  that $\phi_{x_i}(\underline{r})$ is an affine function of  $\underline{r}$.

   Fix $(\l^*, \underline{r}^*) \in \O_*$. Consider the open  subset $U(\l^*, \underline{r}^*)
  \subset\O_*$ given by the vectors  
     $ (\l, \underline{r}) \in \O_*$ such that $-\Re(\l^*)/ \sqrt{2} < -\Re( \l) < -\sqrt{2} \Re(\l^*)$ 
  and $ \Re \left( \underline{r}^*\right)/ \sqrt{2}< \Re\left ( \underline{r}\right ) < \sqrt{2}\Re\left ( \underline{r}^*\right ) $. Trivially, $(\l^*, \underline{r}^*) \in U(\l^*, \underline{r}^*)
  $.
  
   If $\g \in \cA_m$ and $  (\l, \underline{r})  \in U(\l^*, \underline{r}^*)
$  we can bound
   \begin{equation}\label{pasolini}
   \begin{split}
    \left|  g_\g ( \l, \underline{r})\right| =\prod _{e \in \hat E} | r( e)| ^{N_e(\g)}  \prod _{i=0}^{m-1} \frac{1}{|\phi_{x_i}(\underline{r})-\l|} & \leq  2^m \prod _{e \in \hat E}  \Re\left( r( e)\right) ^{N_e(\g)}  \prod _{i=1}^{m} \frac{1}{\phi_{x_i}\left(\Re(\underline{r}) \right)-\Re(\l)}
  \\ &\leq  4^m \prod _{e \in \hat E}  \Re\left( r^*( e)\right) ^{N_e(\g)}  \prod _{i=0}^{m-1} \frac{1}{\phi_{x_i}\left(\Re(\underline{r}^*) \right)-\Re(\l^*)}\\
  & =  \prod _{e \in \hat E}  \Re\left(4  r^*( e)\right) ^{N_e(\g)}  \prod _{i=0}^{m-1} \frac{1}{\phi_{x_i}\left(\Re(\underline{r}^*) \right)-\Re(\l^*)}\\
  & \leq  \prod _{e \in \hat E}  \Re\left(4  r^*( e)\right) ^{N_e(\g)}  \prod _{i=0}^{m-1} \frac{1}{
  \phi_{x_i}\left(\Re(4 \underline{r}^*) \right)+1 
 }
   \,.
   \end{split}
      \end{equation}
 Indeed the first bound follows from Assumptions (i) and (ii) in the definition of $\O_*$, the second bound follows from the definition of $U(\l^*, \underline{r}^*)$,
 the last identity  follows from the fact that all edges of $\g$ are in $\hat E$, while the last bound  follows from 
        Assumption (iii) in the definition of $\O_*$ since we can bound 
   \begin{equation}\label{angioletto}
   \phi_{x_i}\left(\Re(\underline{r}^*) \right)- \Re(\l^*) \geq 4 \phi_{x_i}\left(\Re(\underline{r}^*) \right)+1 \geq  \phi_{x_i}\left(\Re(4 \underline{r}^*) \right)+1 \,.
   \end{equation}
 
 We are now interested to the infinite series of holomorphic functions 
\begin{equation}\label{fondo}
\sum    _{m=1}^\infty \sum _{\g = (x_0,x_1, \dots, x_m) \in \cA_m} g_\g (\l, \underline{r})\,.
\end{equation}
By \eqref{pasolini} for any $  (\l, \underline{r})  \in U(\l^*, \underline{r}^*)
$   we have
\begin{multline*}\sum    _{m=1}^\infty \sum _{\g = (x_0,x_1, \dots, x_m) \in \cA_m}\bigl| g_\g (\l, \underline{r}) \bigr|  \\ \leq   
\sum    _{m=1}^\infty \sum _{\g = (x_0,x_1, \dots, x_m) \in \cA_m}
\prod _{e \in \hat E}  \Re\left(4  r^*( e)\right) ^{N_e(\g)}  \prod _{i=0}^{m-1} \frac{1}{
  \phi_{x_i}\left(\Re(4 \underline{r}^*) \right)+1 
 }
   \,.
\end{multline*} 
Comparing with \eqref{giaco1}, the above r.h.s. equals 
the function $\tilde f^*_+\left ( \Re(\l^*) \right)$ with $\tilde f^*_+$  defined as the function $\tilde f_+$ referred to the random walk on $\cG$ induced by  weights 
$$ E \ni e \to
 \begin{cases}
4 r^*(e) & \text{ if } e \in \hat E \,,\\
 r(e) & \text{ if } e \in E \setminus \hat E \,.
\end{cases}
$$
Since $\Re(\l^*) <0$ the value $\tilde f^*_+\left ( \Re(\l^*) \right)$  is finite by definition of $\tilde f^*_+$.

  Since each compact subset of $\O_*$ can be covered by  the union of a finite family of sets of the form $U(\l^*, \underline{r}^*)$ we conclude that series \eqref{fondo} converges uniformly on compact subsets of $\O_*$.  By a classical theorem in complex analysis (see e.g.  \cite{Mal}[Ch. I, Prop.2] or \cite{GR}[Ch. I, Lemma 11]), we  conclude that the limiting function $h^*_+$  is holomorphic. Since 
   by \eqref{giaco1} the function $h_+$ in the main statement 
  equals  the series \eqref{fondo} on $ \O= \O_*\cap ( \bbR\times \bbR^{\hat E})$ we conclude that 
 $h_+$ is the restriction of $h^*_+$ on $\O$. By similar arguments,  $h_-$ is the restriction of $h^*_-$ on $\O$,  $h^*_-$ being an holomorphic function on $\O_*$ whose definition is analogous to $h^*_+$.
\end{proof}
\medskip
     
% We say that that family of positive rates $\{r(e)\}_{e\in E}$ is \emph{GC--good } if the associated $\l$--functions $\tilde{f}_+(\l)$ and $\tilde{f}_-(\l)$  are proportional: $\tilde{f}_+(\l )/ \tilde{f}_-(\l) = c( \{r(e)\}_{e\in E})$. 
Since  $h_->0$  on $\O$, there exists an open  subset $\O_{**}\subset \bbC \times \bbC^{\hat E}$ with $\O\subset\O_{**} \subset \O_*$ and such that $h^*_-\not = 0$ on $ \O_{**}$. At cost to restrict $\O_{**}$ we can assume that 
\begin{equation}\label{galline}\{ \l \in \bbC \,:\, (\l, \underline{r})\in \O_{**}\}
\end{equation} is connected for any fixed $\underline{r} \in (0,+\infty)^{\hat E}$.

\begin{Remark}\label{importante}
  By definition of $\O$, given  $\underline{r} \in (0,+\infty)^{\hat E}$, it holds $(\l, \underline{r})\in \O_{**}
$ if $\l $ is real and  $-\l >3 \max_{v \in \hat V} r(v) +1$. \end{Remark}

 The function 
$h_+^*/h_-^*$ is well defined and holomorphic on $\O_{**}$.  As a consequence, also the derivative   $h:=\partial_\l( h^*_+/h^*_-)$ is holomorphic (cf. \cite{C}[Sec. IV.2.2]).    
Note that, due to \eqref{salatini} and Theorem \ref{GC},  the function   $\frac{h_+^*}{h_-^*} (\l , \underline{r})$  restricted to $\O$
does not depend on $\l$  if $\underline{r} \in \cR$, the set defined in Proposition \ref{davide}. 
%In particular, $\cR\subset\bigl \{   \underline{r}\,:\, h(\l , \underline{r}) =0\bigr\}$. To prove Proposition \ref{davide} it is enough to show that the subs
In particular,   $h(\l, \underline{r})=0$ if $(\l, \underline{r} ) \in \O$ and $\underline{r}\in \cR$.
Consider the holomorphic function $\l \to h(\l , \underline{r})$, where $\underline{r}\in \cR$ is fixed. This function is defined on the set $\{\l \in \bbC\,:\, (\l , \underline{r}) \in \O_{**}\}$.  Since it has no isolated zeros  and  since \eqref{galline} is connected, 
we get that $h(\l, \underline{r}) = 0$  for any $\underline{r} \in \cR$ and any $\l \in \bbC\,:\, (\l , \underline{r}) \in \O_{**}$ (see \cite{C}).

Suppose now, by contradiction, that the set $\cR$ has positive Lebesgue measure (here and in what follows  we refer to the $|\hat E|$--dimensional Lebesgue measure).  Fix $\l<0$ and define $ \O_\l :=\left\{ \underline{r} \in \bbR^{\hat E}\,:\, (\l,\underline{r}) \in \O\right\}$ and  the function $h_\l : \O_\l \to \bbR$ as  $h_{\l}(\underline{r}) := h(\l, \underline{r})$.
 Note that $\O_\l$ is connected and that $h_\l$ is a real analytic function (locally  it admits a  convergent power series expansion, since restriction of an holomorphic function).  
 Since $\O_{\l} \subset \O_{\l'}$ if $\l' < \l$ and since $\cup _{\l <0} \O_\l= (0,+\infty)^{\hat E}$, we can find $\l_0<0$ such that $\O_\l \cap \cR$ has positive Lebesgue measure 
for $\l \leq \l_0$. From now on we assume $\l \leq \l_0$. This implies that the set $\{h_\l =0\}$ has positive Lebesgue measure. We claim that it must then be $h_\l \equiv 0$ on the entire connected set $\O_\l$ as a consequence of Weierstrass Preparation Theorem. Indeed, $h_\l$ is the restriction to $\O_\l$ of the holomorphic function $h(\l , \cdot)$ defined on an open subset of $\bbC^{\hat E}$ containing $\O_\l$. Then the thesis follows from this general fact:

\begin{Lemma} Fix $n \geq 1$ integer. 
Let $V$ be an open set of $\bbC^n$ such that $ U:= V\cap \bbR^n$ is connected. Let $f : V \to\bbC$ be an holomorphic function. Then either $f\equiv 0$ on $U$ or the set $\{z \in U\,: \, f(z)=0\}$ has zero  $n$--dimensional Lebesgue measure.
\end{Lemma}
\begin{proof} Note that $U$ is open. Below Lebesgue measure is considered as $n$--dimensional. 
It is enough to prove the following claim:

\begin{Claim}\label{stinco} For any $z \in U$ there is a neighborhood  $B_z$ of $z$ in $U$ such that 
the set $ \{y \in B_z\,:\, f(y) =0\}$  has nonempty  open part  or has zero Lebesgue measure.
\end{Claim} 

Let us first assume the above claim and show how to conclude. If for all $z \in U$ the set $ \{y \in B_z\,:\, f(y) =0\}$ has  zero   Lebesgue measure, then  each compact  subset $K \subset U$   can be covered by a finite family  $B_{z_1}$, $B_{z_2}$,...,$B_{z_r}$, thus implying that $\{y \in K\,:\, f(y)=0\}$ has zero  Lebesgue measure. This trivially leads to the fact that $\{z \in U\,: \, f(y)=0\}$ has zero  Lebesgue measure. On the other hand, if for some $z \in U$ the set  $ \{y \in B_z\,:\, f(z) =0\}$  has nonempty  open part, then the analytic function given by $f$ restricted to $U$ is zero on a ball inside $U$ and therefore is zero on all $U$ (see \cite{C}[Ch. IV.2.3]). 

At this point we only need to prove the above Claim \ref{stinco}.
 If $f(z)\not =0$ then for $B_z$ small the set $ \{y \in B_z\,:\, f(z) =0\}$ is empty and we are done. Suppose that $f(z)=0$ and $f$ not identically zero around $z$. 
By Weierstrass preparation theorem \cite{GR}[Ch. II.B], there exists $\e>0$ such that for all $y=(y_1, y_2, \dots, y_n) \in \bbC^n$ with $|y_i-z_i| < \e$ for all $i$ it holds 
\begin{equation}
\begin{split}
f(y)= h(y) \Big[ (y_n-z_n)^k & + a_{1}(y_1, \dots, y_{n-1})( y_n-z_n)^{k-1}\\ & + \dots
 + a_{k-1} (y_1, \dots, y_{n-1} ) (y_n-z_n)+ a_k(y_1, \dots, y_{n-1})  \Big] \,,
\end{split}
\end{equation}
where $y=(y_1, \dots, y_n)$, $k $ is a suitable integer, $a_1, \dots , a_k$ are holomorphic functions, and $h$ is a  never--zero  holomorphic function. It then follows that, fixed $(y_1, \dots, y_{n-1})$ with $|y_i -z_i |< \e$, the set $\{ y_n \in \bbC\,:\, |y_n -z_n|< \e\,,\; f(y_1, \dots, y_{n-1}, y_n)=0\}$ has cardinality at most $k$ (in particular, it has zero Lebesgue measure when intersected with $\bbR$).  The thesis follows by taking  $B_z:=\{ y \in \bbR^n\,:\, |y_i -z_i| <\e\}$ and applying Fubini theorem.
\end{proof}

Up to now, assuming that $\cR$ has positive Lebesgue measure,   we have proved that for each 
  $\underline{r}\in (0,+\infty)^{\hat E}$
 it holds  $h(\l, \underline{r})=0$ for $\l <0$ and   $|\l|$ large enough:  $\l<\l_0$ and $-\l >3 \max_{v \in \hat V} r(v) +1$ (see Remark \ref{importante}).
 In particular,  it is simple to  define  an increasing function $\varphi: (0,+\infty) \to (0,+\infty)$ such that $h(\l, \underline{r})=0$ for  all  $\underline{r}\in (0,+\infty)^{\hat E}$ and $\l < -\varphi \left( \max _{ e \in \hat E} r(e)\right)$. In particular we have proved the following fact:

   \begin{Fact}\label{tropicale}
 For each fixed  $\underline{r}\in (0,+\infty)^{\hat E}$, the ratio   $h_+(\l, \underline{r}) /h_-(\l , \underline{r})$ is constant for   $\l < -\varphi \left( \max _{ e \in \hat E} r(e)\right)$.
\end{Fact}

 We now show that this is in contradiction with the assumption that the fundamental graph $G$ is not $(\underline{v}, \overline{v})$--minimal. 
Indeed, since $G$ is not $(\underline{v}, \overline{v})$--minimal,  
there exist at least two
 paths $\g^{(1)}=(z_0, z_1, \dots, z_M) $ and $\g^{(2)}=(z'_0, z'_1, \dots, z'_{M'})$ in $\cA_M$ and $\cA_{M'}$ respectively, 
 such that the points $z_i$ 
are all distinct, the points $z'_i$ are all distinct,  
  and for some  non--negative integers $\kappa_1$, $\kappa_2$ with $\kappa_1 + \kappa_2 + 2 \leq M \wedge M'$ it holds 	\begin{align*}
 	& z_i = z'_i  \qquad \quad  \; \forall \; 0\leq i \leq \kappa_1\,, \\
 	& z_{M-i} = z'_{M'-i}  \quad \forall \; 0\leq i \leq \kappa_2\,, \\
 	&\left \{ z_{ \kappa_1 +1} , \ldots , z_{M-\kappa_2-1} \right \} \cap \left\{z'_0, z'_1, \dots, z'_{M'}\right\} = \emptyset\,,\\
 	& \left\{z_0, z_1, \dots, z_M\right\} \cap 	
 	\{ z'_{ \kappa_1 +1} , \ldots , z'_{M'-\kappa_2-1} \} = \emptyset \, . 
 	\end{align*}
	In other words,  $\g^{(1)}$ and $\g^{(2)}$ are  linear chains,  they have in common the first $\k_1+1$ points and the last $\k_2+1$ points,  while they divide  in their interior part. 

Let $E_\star:= \Gamma _1 \cup \Gamma_2$, where 
$$ \Gamma _1:= \{ (z_i, z_{i+1}), (z_{i+1},z_i): 0\leq i <M\} \,, \qquad \G_2:= \{  (z'_j, z'_{j+1}), (z'_{j+1},z'_j): 0 \leq j <M'\}\,.
$$
Note that $E_\star\subset \hat E$. 
Introduce a new connected fundamental graph $G'=(V', E')$, where $E' = (E \setminus \hat E) \cup E_\star$ and $V'$ is given by the vertices appearing in the edges of $E'$.
As marked vertices we take again $\underline{v}, \overline{v}$.

Let $\bar r =\bigl( r(e)\,:\, e \in E_\star\bigr)\in (0,+\infty)^{E_\star}$ and for $k \geq 1$ let $\underline{r}^{(k)} \in (0,+\infty)^{\hat E}$ be defined as $\bar r$ on $E_\star$ and as $1/k$ on $\hat E \setminus E_\star$. 
Then
\begin{equation}\label{cremina}
 \lim_ {k \to \infty } h_\pm\left( \l, \underline{r}^{(k)} \right)= \tilde{ f}^*_\pm(\l)\,, \qquad \l\leq - \varphi \Big( \max _{e\in E_\star} r(e) \Big)\,,\end{equation}
where $\tilde{ f}^* _\pm$ refers to the rw  on the quasi 1d lattice induced by $(G', \underline{v}, \overline{v})$ and  by the weights  $r(e) $ with $e \in E \setminus \hat E$ (that have been fixed once and for all) and the weights $r(e)$ with $e \in E_ \star$. The limit \eqref{cremina} follows from the fact that, as $k \to \infty$, the probability to have a jump along an edge not in $E'$ goes to zero (use  the graphical construction for Markov chains)

Due to \eqref{cremina} and Fact \ref{tropicale} we have that 
the ratio $\tilde{f}^*_+(\l)/\tilde{f}^*_-(\l)$ is constant for $\l <0$  with   $|\l|$ large. 
 At this point, to have a contradiction it is enough to prove that for a suitable choice of $\overline{r}$ the above assertion is impossible. Let $\cA'_m$ be the analogous  of $\cA_m$ referred now to the graph $G'=(V',E')$. 
 Then for each path $\g$ in $\cA_m'$ for some $m$  going from $\underline{v}$ to $\overline{v}$, it holds either 
 	\begin{equation} \label{radio1}
 	\begin{split}
 	& N_{(z_i , z_{i+1})} (\g) - N_{(z_{i+1} , z_{i})} (\g) = 1 \quad \mbox{for all } i:0\leq i <m \, , \\
 	& N_{(x, y)} (\g) - N_{(y,x)} (\g) = 0 \quad \mbox{for all } 
	(x,y)\in \G_2 \setminus \G_1\,,
 	\end{split}
 	\end{equation}
 or 
 \begin{equation} \label{radio2}
 	\begin{split}
 	& N_{(z_i ', z_{i+1}')} (\g) - N_{(z_{i+1}' , z_{i}')} (\g) = 1 \quad \mbox{for all } i:0\leq i <m' \, , \\
 	& N_{(x, y)} (\g) - N_{(y,x)} (\g) = 0 \quad \mbox{for all } 
	(x,y)\in \G_1\setminus \G_2\,,
 	\end{split}
  	\end{equation}
So we can define the two disjoint sets 
	\begin{align*}
	& \mathcal{P}_1 := \{\;\mbox{paths in }  \cup _{m \geq 1} \cA_m' \mbox{ such that \eqref{radio1} holds} \} \\
	& \mathcal{P}_2 := \{\;\mbox{paths in } \cup _{m \geq 1} \cA_m'  \mbox{ such that \eqref{radio2} holds} \}  
	\end{align*}
%Note that $\mathcal{P}_1, \mathcal{P}_2$ partition the set of paths in $G^*=(V,E')$ going from $\underline{v}$ to $\bar{v}$. 
%Now to each $\g_i \in \mathcal{P}_i$ we can associate a path $\g_i^* $ via the bijection defined before \eqref{giaco2}, $i=1,2$.
Inverting the role of $\underline{v}, \overline{v}$ and considering paths from $\overline{v}$ to $\underline{v}$, one can define $(\cA_m')^*, \cP_1^*, \cP_2^*$ analogous of $\cA_m'$, $\cP_1$, $\cP_2$, respectively.   For example, $ \cP_1^*$ is given by the paths $\g$ in $\cup _{m\geq 1}(\cA_m')^*$ such that
\begin{equation} \label{radio100}
 	\begin{split}
 	& N_{(z_{i+1} , z_{i})} (\g) - N_{(z_{i} , z_{i+1})} (\g) = 1 \quad \mbox{for all } i:0  \leq i <m \, , \\
 	& N_{(x, y)} (\g) - N_{(y,x)} (\g) = 0 \quad \mbox{for all } 
	(x,y)\in \G_2 \setminus \G_1\,,
 	\end{split}
 	\end{equation}

Given a path $\g = (x_0, x_1, \dots, x_m)$  we define  the reversed path $\g^*= (x_m,x_1, \dots, x_0)$. Note that if $\g \in \cP_i$ then $\g^* \in \cP^*_i$.
Using formulas similar to \eqref{giaco1} referred now to $G'$  we have 
%%%%%%%%%%%%%%%%%%%%
\begin{equation}\label{fratello}
\frac{ \tilde{f}_+(\l )}{\tilde{f}_-(\l )}  = \frac{ \tilde{f}_{1,+}(\l ) +\tilde{f}_{2,+}(\l )  }{\tilde{f}_{1,-}(\l )
+\tilde{f}_{2,-}(\l)}
\end{equation}
where, for $s=1,2$,
\begin{align*}
&  \tilde{f}_{s,+}(\l )=  \ds \sum_{m=1}^\infty \sum _{\g \in \cP _s \cap \cA'_m}
    \prod _{e \in E_\star} r(e) ^{N_e(\g)}  \prod _{i=0}^{m-1} \frac{1}{r(x_i)-\l}\,,\\
&   \tilde{f}_{s,-}(\l )=  \ds \sum_{m=1}^\infty \sum _{\g \in \cP ^*_s \cap (\cA'_m)^*}
    \prod _{e \in E_\star} r(e) ^{N_e(\g)}  \prod _{i=0}^{m-1} \frac{1}{r(x_i)-\l}\,.
    \end{align*}
    Note that $r(x)$ is now referred to the fundamental graph $G'$ with weights $r(e) $ with $e \in E \setminus \hat E$ (that have been fixed once and for all) and the weights $r(e)$ with $e \in E_ \star$. Simply forget $G$.
    
  Due to \eqref{radio1} and \eqref{radio100} and similar formulas, for $s=1,2$ we get    \begin{equation}\label{sorella}
     \tilde{f}_{s,+}(\l )=\tilde{f}_{s,-}(\l) \cdot \D_s \,, \qquad \D_1:=\prod _{i=0}^{M-1} \frac{r(z_i, z_{i+1}) }{r(z_{i+1},z_i) } \,, \qquad \D_2:=\prod _{i=0}^{M'-1} \frac{r(z'_i, z'_{i+1}) }{r(z'_{i+1},z'_i) } \,.
    \end{equation} 
   Combining \eqref{fratello} and \eqref{sorella} we have
   \begin{equation}\label{mammina}
   \frac{  \tilde{f}_+(\l )}{\tilde{f}_-(\l )}  =
    \frac{\tilde{f}_{1,+}(\l ) + \tilde{f}_{2,+}(\l )  }{\D_1^{-1} \tilde{f}_{1,+}(\l )
+\D_2^{-1} \tilde{f}_{2,+}(\l)}\,.
   \end{equation} 
   %Dividing by $\tilde{f}_{1,+}(\l )$ we conclude that the l.h.s. of \eqref{mammina} does not depend on $\l$ for $\l<0$ small if and only if the same holds for the ratio $\tilde{f}_{2,+}(\l ) /\tilde{f}_{1,+}(\l ) $.
If $\D_1 = \D_2$ then the ratio $\tilde{f}_{+}(\l ) /\tilde{f}_{-}(\l )$ is independent of $\l$, but this happens for a set of rates of Lebesgue measure $0$. Assume now $\D_1 \neq \D_2$. Then dividing by $\tilde{f}_{1,+}(\l )$ we conclude that the l.h.s. of \eqref{mammina} does not depend on $\l$ for $\l<0$    with  $|\l|$ large  if and only if the same holds for the ratio $\tilde{f}_{2,+}(\l ) /\tilde{f}_{1,+}(\l ) $.

We point out that each path in $\cP_1 $ belongs to $\cA_m'$ for some $m \geq M$. The only path in $\cP_1$ belonging to $ \cA_M'$ is $\g^{(1)}$, while there is no path in $\cP_1$  
belonging to $ \cA_{M+1}'$.
 Moreover, $|\cA_m'|\leq 3^m$ since,   when constructing  a path $\g \in \cA_m'$ vertex by vertex, at each step we can choose only among $2$ or $3$ 
 neighbors.  Hence 
\begin{equation}\label{lunare}
\Big|\ds \sum_{m=M+2}^\infty \sum _{\g \in \cP _1 \cap \cA'_m}
    \prod _{e \in E_\star} r(e) ^{N_e(\g)}  \prod _{i=0}^{m-1} \frac{1}{r(x_i)-\l} 
\Big| \leq \ds \sum_{m=M+2}^\infty  \frac{c^m }{|\l|^m }= \frac{(c/|\l| )^{M+2}}{1-c/|\l|}\,,\end{equation}
where $ c:= 3 \max \{ r(e) : e \in E_*\}$.
In particular, separating the contribution of $\g^{(1)}$ from the other paths in the definition of 
$\tilde{f}_{1,+}(\l )$, we have that 
  \[\tilde{f}_{1,+}(\l ) = 
      c_1 \prod _{i=0}^{M-1} \frac{1}{r(z_i)-\l}+ O\left( 
      \frac{1}{|\l|^{M+2} }\right) \,, \qquad  c_1:= \prod _{i=0}^{M-1} r( z_i,z_{i+1}) \] 
     Above $ O\left( 
      \frac{1}{|\l|^{M+2} }\right)$ means that the term in consideration is bounded in modulus by $C/|\l|^{M+2} $.
      Note that 
 for $\l<0$   with  $|\l|$ large  we have
       $$ \frac{1}{r(z_i)-\l} = \frac{1}{|\l|} \frac{1}{1+r(z_i)/|\l|} = \frac{1}{|\l|}
\left( 1- \frac{r(z_i)}{|\l|}+ \cE_i(\l) \right)       
 $$
  where $\lim _{\l \to -\infty}    |\l|   \cE_i(\l)=0$. The same arguments hold for $\tilde{f}_{2,+}$ where
  $c_2:= \prod _{i=0}^{M'-1} r( z'_i,z'_{i+1})$.   In conclusion we have 
  \begin{align*}
  &  \tilde{f}_{1,+}(\l ) = 
      \frac{c_1}{|\l|^M}-\frac{ c_1 }{|\l|^{M+1} } \sum _{i=0}^M r(z_i)+o\left( 
      \frac{1}{|\l|^{M+1} }\right) \,,\\
    &  \tilde{f}_{2,+}(\l ) = 
      \frac{c_2}{|\l|^{M'}}-\frac{ c_2 }{|\l|^{M'+1} } \sum _{i=0}^{M'} r(z'_i)+o\left( 
      \frac{1}{|\l|^{{M'}+1} }\right) \,.
      \end{align*}   
Since $\tilde{f}_{1,+}(\l ) ,\tilde{f}_{2,+}(\l ) $ are proportional for $\l<0$  with  $|\l|$ large,  it must be  
%\begin{equation}\label{alba}
$M=M'$ and $\sum _{i=0}^M r(z_i)=  \sum _{i=0}^{M} r(z'_i)$.
 These identities  cannot be true in general. If $M \not =M'$ trivially we have a contradiction. Otherwise  take $r(e)=1$ for all $e \in \G_1$ and $r(e)=a>0$ for all $a \in \G_2 \setminus \G_1$. If $a$ is large then the identity $\sum _{i=0}^M r(z_i)=  \sum _{i=0}^{M} r(z'_i)$  fails.
 %  Note that $r(z_0)= r(\underline{v})$, $r(z_i)= 2$ for $i=1,\dots, \k_1-1$, $r(z_{\k_1})= 2+a$,
  % $ r( z_{ M-\k_2})= 2+a$, $r( z_i)= 2$ for $i= M-\k_2+1, \dots, M-1$.
 %  Setting $j= M- \k_1-\k_2-2$ we have that $\g^{(1)}$ has $j$ points where $r(\cdot)$ equals 2, and   $\g^{(2)}$ has $j$ points where $r(\cdot)$ equals 2a. Hence
% Then  $c_1 =1$, $c_2= a^{j+1}$ and 
 % \begin{align*}
 % \sum_{i=1}^{M-1} r(z_i) = r(\underline{v})+ 2 [\k_1+\k_2+a] +2j \,,    \\
 % \sum _{i=0}^{M-1} r( z_i')=r(\underline{v})+ 2 [\k_1+\k_2+a] +2a j\,.
  %   \end{align*}
   %  Taking $a$ large enough  we get that      $\sum_{i=1}^{M-1} r(z_i)  \not = \sum_{i=1}^{M-1} r(z'_i)$.
\qed
%%%%%%%%%%%%%%%%%%%%%%%%%%%%%%%%%%%%%%%%%%%%

%%%%%%%%%%%%%%%%%%%%%%%%%%%%%%%%%%%%%%%%%%%%

%%%%%%%%%%%%%%%%%%%%%%%%%%%%%%%%%%%%%%%%%%%%%%%%%

\appendix

\section{Proposition \ref{cuore} implies \eqref{UB}} \label{reduction}
For completeness,  following similar arguments as in \cite{DGZ}, 
  we explain how  one can deduce  from Proposition \ref{cuore} the upper bound \eqref{UB} for all $\th \in \bbR$, assuming $v>0$.
 Recall that    $S_t :=\inf  \{ s\geq t : Z_s \leq 0 \}$,   and observe that for all $u>0$ it holds
	\begin{equation} \label{reverse}
	\bbP \Big( \inf_{s \geq t} Z_s \leq ut \Big) \leq q^{-ut} \bbP (S_t < \infty )\,, \qquad q:= \bbP( w_1=-1) \, . 
	\end{equation}
	 To prove the above bound observe that  one possible way of realizing the event $\big\{ \inf_{s \geq t} Z_s \leq 0\big\} $ is the following.  If $Z_t > \lfloor ut \rfloor$ then  the process  hits $\lfloor ut \rfloor$ after time $t$ and then performs $\lfloor ut \rfloor $ consecutive steps to the left. If  $Z_t \leq  \lfloor ut \rfloor$ then after time $t$  the process  performs 
	 $\lfloor ut \rfloor $ consecutive steps to the left. In particular we get 
	 	\[ \bbP (S_t < \infty ) = \bbP \bigg( \inf_{s \geq t} Z_s \leq 0 \bigg) \geq 
	\bbP \bigg( \inf_{s \geq t} Z_s \leq \lfloor ut \rfloor \bigg) q^{\lfloor ut \rfloor } 
	\geq \bbP \bigg( \inf_{s \geq t} Z_s \leq ut \bigg) q^{ut} \, . \]
%In the second inequality above we have used that, since $Z$ is an integer valued process, 
%$\big\{  \inf_{s \geq t} Z_s \leq ut \big\} = \big\{  \inf_{s \geq t} Z_s \leq \lfloor ut \rfloor  \big\} $. \\	
From \eqref{reverse} and Proposition \ref{cuore} we readily get \eqref{UB} for $\th = 0$:
	\[ \limsup_{t \to \infty} \frac{1}{t} \log \bbP \Big( \frac{Z_t}{t} \in ( -\e , \e ) \Big) 
	\leq
	\limsup_{t \to \infty} \frac{1}{t} \log \bbP \Big( \inf_{s \geq t} Z_s \leq \e t \Big) 
	\leq -\e -I(0) \stackrel{\e \to 0}{\longrightarrow} -I(0) \, . \]

Fix, now, any $\th >0$ and take $\e$ small enough so that $u := \th - \e >0$ and fix $u'\in (0,u)$. Let $m$ be any positive integer.  Then we have  for $t$ large (as we assume)
	\[ \begin{split}
	\bbP \left( \frac{Z_t}{t} \in ( \th -\e ,\th + \e ) \right) & 
	\leq \bbP ( Z_t \in [ut , ut + 2\e t ] ) 
	= \bbP (T_{   \lfloor ut \rfloor  } \leq t , \, Z_t \in [ut , ut + 2\e t ] ) \\
	& \leq \bbP \bigg( \frac{T_{ \lfloor ut \rfloor }}{ \lfloor ut \rfloor} \leq \frac{1}{u'} , \, \inf_{s \geq t} Z_s \leq ut + 2\e t \bigg) \\
	& \leq \sum_{k=1}^m \bbP \bigg( \frac{T_{  \lfloor ut \rfloor }}{ \lfloor ut \rfloor} \in \bigg[ \frac{(k-1)}{u'm} , \frac{k}{u'm} \bigg] \bigg) 
	\bbP \left( \inf_{s \geq t - \frac{kt}{m}} Z_s \leq 2\e t \right) \\
	& \leq q^{-2\e t } \sum_{k=1}^m \bbP \bigg( \frac{T_{ \lfloor ut \rfloor}}{ \lfloor ut \rfloor} \in \bigg[ \frac{(k-1)}{u'm} , \frac{k}{u'm} \bigg] \bigg) 
	\bbP \big( S_{t - \frac{kt}{m}} < \infty \big) \,.
	\end{split}
	\]
We point out that the third inequality above follows from the strong Markov property applied at time $T_{   \lfloor ut \rfloor  }$ and the fact that the probability $ \bbP \left( \inf_{s \geq t-a} Z_s \leq 2\e t \right) $ is
increasing in $a$. The last inequality follows from \eqref{reverse}.

Reasoning as in \eqref{hot_wheels_1} and using  Proposition \ref{cuore},  we get for $1\leq k < m$	\[ \begin{split}
\bbP \bigg( \frac{T_{ \lfloor ut \rfloor}}{ \lfloor ut \rfloor} \in \bigg[ \frac{(k-1)}{u'm} , \frac{k}{u'm} \bigg] \bigg) 
	&
	\bbP \big( S_{t - \frac{kt}{m}} < \infty \big)  
	\leq e^{ t \e + tu w_+ \big( \frac{1}{u'} , \frac{1}{mu'} \big) } e^{-t \big[ \frac{k}{m} I \big( \frac{m u'}{k} \big) + \big( 1-\frac{k}{m} \big) I(0) \big] }\\
	& \leq e^{ t \e + tu w_+ \big( \frac{1}{u'} , \frac{1}{mu'} \big) } e^{-t I(u') } \, , 
	\end{split} \]
where $t$ is taken large enough and $w_+$ is defined as in \eqref{oscillazioni}. 
Note that the last inequality follows from the convexity of $I$. 
When $k=m$, on the other hand, $ \bbP \big( S_{t - \frac{kt}{m}} < \infty \big)  = \bbP \big( S_0 < \infty \big)  = 1$ and, as in \eqref{hot_wheels_1}, for $t$ large we have 
	\[ \bbP \bigg( \frac{T_{ \lfloor ut \rfloor}}{ \lfloor ut \rfloor} \in \bigg[ \frac{(m-1)}{u'm} , \frac{1}{u'} \bigg] \bigg) 
	\leq e^{t\e + tu w_+ \big( \frac{1}{u'} , \frac{1}{mu'} \big) } e^{-t I(u') } \, . \]
Putting all together, we have shown that for any $\e$ small and $t$ large enough it holds
	\[ \bbP \bigg( \frac{Z_t}{t} \in ( \th -\e ,\th + \e ) \bigg) \leq m \cdot q^{-2\e t } \cdot   e^{  t \e + tu w_+ \big( \frac{1}{u'} , \frac{1}{mu'} \big) } e^{-t I(u') } \]
with $u = \th - \e$, and therefore
	\[ \limsup_{t \to \infty} \frac{1}{t} \log \bbP \bigg( \frac{Z_t}{t} \in ( \th -\e , \th + \e ) \bigg)
	\leq
	- 2 \e \log q + \e +  u w_+ \big( \frac{1}{u'} , \frac{1}{mu'} \big) - I(u') \, . \]
Letting, now, $m\to \infty$ and then $\e \to 0 $ (so that also $u \to \th $) and taking $u' \to \th$ gives \eqref{UB}. 

The proof of the same bound for $\th < 0$ follows by similar arguments.

\medskip

\medskip

\noindent {\bf Acknowledgements.} The authors thank  K. Duffy,
D. Fiorenza, N. Gantert, M. Manetti, C. Macci, M. Mariani, J.R.  Norris, G.L. Torrisi for useful discussions.
  V. Silvestri thanks the Department of Mathematics in University ``La Sapienza''
for the hospitality and  acknowledges the support of the UK Engineering and Physical Sciences Research Council (EPSRC) grant EP/H023348/1 for the University of Cambridge Centre for Doctoral Training, the Cambridge Centre for Analysis.

\end{document}